\documentclass[11pt]{article}
\usepackage[T1]{fontenc}
\usepackage[utf8]{inputenc}
\usepackage{lmodern}
\usepackage[margin=1.02in,headheight=14pt]{geometry}
\usepackage{amsmath,amssymb,amsthm,mathtools}
\usepackage{microtype,booktabs,array}
\usepackage{mathrsfs}
\usepackage{enumitem,needspace,etoolbox}
\usepackage{fancyhdr}
\usepackage[hidelinks,pdfencoding=auto]{hyperref}
\usepackage{bookmark}
\usepackage{tikz}
\usetikzlibrary{patterns}
\numberwithin{equation}{section}
\newtheorem{theorem}{Theorem}[section]
\newtheorem{proposition}[theorem]{Proposition}
\newtheorem{lemma}[theorem]{Lemma}
\newtheorem{corollary}[theorem]{Corollary}
\BeforeBeginEnvironment{theorem}{\Needspace{10\baselineskip}}
\BeforeBeginEnvironment{proposition}{\Needspace{8\baselineskip}}
\BeforeBeginEnvironment{lemma}{\Needspace{8\baselineskip}}
\BeforeBeginEnvironment{corollary}{\Needspace{8\baselineskip}}
\theoremstyle{definition}

\theoremstyle{remark}
\newtheorem{remark}[theorem]{Remark}
\DeclareMathOperator{\supp}{supp}
\DeclareMathOperator{\sgn}{sgn}
\DeclareMathOperator{\PV}{PV}

\newcommand{\R}{\mathbb R}
\newcommand{\Sph}{\mathbb S}
\newcommand{\Cyl}{\mathcal C}
\newcommand{\E}{\mathcal E}
\newcommand{\Pp}{\mathcal P}
\newcommand{\Aa}{\mathcal A}

\newcommand{\dd}{\,\mathrm d}
\newcommand{\eref}{\mathcal E_{\mathrm{ref}}}
\newcommand{\sh}{\mathrm{sh}}

\newcommand{\nsc}{\mathfrak n_{p,q}}
\newcommand{\test}{C_c^\infty(\R^d\setminus\{0\})}

\newcommand{\Kij}{Kijaczko~\cite{KijTL}}

\setlist[itemize]{topsep=4pt,itemsep=1pt,leftmargin=1.5em}
\setlist[enumerate]{topsep=4pt,itemsep=2pt,leftmargin=1.8em}
\hypersetup{pdftitle={Sharpness transitions for weighted Triebel--Lizorkin Hardy inequalities with p < q},pdfauthor={Yongyang Jin and Li Tang},pdfsubject={Two-weight Triebel--Lizorkin Hardy inequalities; mixed exponents p<q; sharpness transitions}}

\title{\bfseries\boldmath Sharpness transitions for weighted Triebel--Lizorkin Hardy inequalities with $p<q$}
\author{Jun Cao\textsuperscript{1}\qquad   Yongyang Jin\textsuperscript{1}\qquad
Li Tang\textsuperscript{2}\thanks{Corresponding author.}\qquad Chenfeng Zhu\textsuperscript{1} \\[0.5em]
{\small\textsuperscript{1} School of Mathematical Sciences, Zhejiang University of Technology,}\\
{\small Hangzhou, China}\\[0.25em]
{\small\textsuperscript{2} Institute for Theoretical Sciences, Westlake University,}\\
{\small Hangzhou, 310030, China}\\[0.4em]
{\small E-mail: \href{mailto:caojun1860@zjut.edu.cn}{\texttt{caojun1860@zjut.edu.cn}} (J. Cao);}\\
{\small\href{mailto:yongyang@zjut.edu.cn}{\texttt{yongyang@zjut.edu.cn}} (Y. Jin);}\\
{\small \href{mailto:tangli42@westlake.edu.cn}{\texttt{tangli42@westlake.edu.cn}} (L. Tang);}\\
{\small \href{mailto:chenfengzhu@zjut.edu.cn}{\texttt{chenfengzhu@zjut.edu.cn}} (C. Zhu)}}
\date{}

\begin{document}
\maketitle
\begin{abstract}
For two-weight fractional difference energies with inner and outer
exponents $1<p<q$, we study the optimality of the Hardy constant
generated by a homogeneous power. We prove nontrivial sharpness
intervals for every sign of $D=d+2\beta-sp$, including a two-sided
interval when $D=0$, and strict loss of optimality near both
finite-energy endpoints. For $q\ge2$, sharpness extends beyond the
elementary pointwise scalar criterion through an integrated remainder
estimate: after opposite logarithmic displacements are combined,
a positive contribution at bounded displacements absorbs the negative
contribution at larger displacements. A single quadratic reference
energy controls this comparison and the angular fluctuations of
near-extremizers. When $1<p<q<2$, optimization of the scalar
supporting-line inequality gives the exact power exponent for that
criterion and a larger explicit sharpness interval. Angular variations
provide explicit strict-loss examples. The method extends to flat
singular sets and regional half-space energies, with separate
finiteness ranges for test energies and the homogeneous candidate.
Complementary results for $p>q$ include flat-subspace inequalities. The
constants are optimal on the stated smooth test classes and extend unchanged
to their difference-norm completions whenever they are positive.
\end{abstract}


\section{Introduction}

A homogeneous Hardy inequality often has an optimal constant generated
by a singular power, although that power lies outside the admissible
test class. For fractional energies with distinct inner and outer
exponents, this relation can change with the weights. We study that
change for two-weight Triebel--Lizorkin difference energies: the
homogeneous candidate is sharp on nontrivial weight intervals, but
becomes strictly nonoptimal near both endpoints of the finite-energy range.

The candidate comes from the established fractional Hardy constant
calculation. Kijaczko~\cite{KijTL} proved its mixed-exponent optimality
for $q\le p$ and left the opposite order $p<q$ open. For $q\ge2$, sharpness extends strictly beyond the elementary
scalar criterion: after opposite logarithmic displacements are combined,
a positive remainder at bounded displacements absorbs the negative
remainder at larger ones. For $1<p<q<2$,
optimization of the scalar power exponent instead gives a larger
explicit interval.

A quadratic reference energy makes this comparison quantitative and
controls the angular fluctuations of sequences approaching the optimal
constant. Its translation estimate also controls the tangential shifts
associated with flat singular sets and regional half-space energies.
We first state the sharpness intervals and endpoint-loss theorem,
then discuss the earlier constant formulas and develop the proofs
and geometric consequences.

\subsection{The problem and the homogeneous candidate}\label{sec:problem}

Let $d\ge2$, $0<s<1$, and $1<p<q<\infty$. The $q$th power
of the two-weight Triebel--Lizorkin difference seminorm is
\begin{equation}\label{eq:intro-mixed-functional}
 \int_{\R^d}|x|^\alpha
 \left(\int_{\R^d}\frac{|u(x)-u(y)|^p|y|^\beta}
 {|x-y|^{d+sp}}\dd y\right)^{q/p}\dd x.
\end{equation}
We work on $C_c^\infty(\R^d\setminus\{0\})$. We write
$\widehat x=x/|x|\in\Sph^{d-1}$ for the direction of a point
$x\ne0$, and $\omega_{d-1}=|\Sph^{d-1}|$ for the area of the unit
sphere. For this class the
exact finite-energy conditions are
\begin{equation}\label{eq:intro-weight-range}
 -d<\beta<sp,\qquad -d<\alpha<sq+d(q/p-1).
\end{equation}
The denominator dictated by scaling is
$\int|u(x)|^q|x|^{\alpha+q\beta/p-sq}\dd x$.
With $\beta$ fixed, the natural exponent and candidate integral are
\begin{equation}\label{eq:intro-weighted-J}
\begin{aligned}
 a&=\frac{\alpha+d+q\beta/p-sq}{q},\\
 J_{d,s,p}^{(\beta)}(a)
 &=\int_{\R^d}\frac{|1-|z|^{-a}|^p|z|^\beta}
 {|e-z|^{d+sp}}\dd z,\qquad |e|=1.
\end{aligned}
\end{equation}

The integral in \eqref{eq:intro-weighted-J} is part of the established
fractional Hardy constant calculation. When $q\le p$, Kijaczko~\cite[Theorem~1 and Eq.~(4)]{KijTL}
identified its sharp mixed-exponent role.
For $p<q$, logarithmic cutoffs still bound the best constant from
above by this candidate, while the matching lower bound depends on
the weights. We call the candidate \emph{sharp} when it equals the best
constant on the stated test class. \emph{Loss of sharpness} means that
the best constant is strictly smaller than the candidate. The main results
below describe where sharpness persists and where it must fail.
Historical comparisons and explicit evaluations are given after the
statements, with their original attributions.

\paragraph{Triebel--Lizorkin difference spaces.}
We use the mixed difference seminorm appearing in Prats and
Saksman~\cite[Theorems~1.2 and~1.3]{PS}, and adopt
Kijaczko's~\cite{KijTL} convention of placing the inner exponent first.
Let $\Omega\subset\R^d$ be open, $1\le p,q<\infty$, and
\[
 \max\left\{0,\frac{d(p-q)}{pq}\right\}<s<1.
\]
For a measurable function $u$ on $\Omega$, define the difference seminorm
\[
 [u]_{F^s_{p,q}(\Omega)}
 :=\left(\int_\Omega
       \left(\int_\Omega
       \frac{|u(x)-u(y)|^p}{|x-y|^{d+sp}}\dd y\right)^{q/p}
       \dd x\right)^{1/q}.
\]
Then
\[
 F^s_{p,q}(\Omega)
 :=\left\{u\in L^{\max\{p,q\}}(\Omega):
       \|u\|_{L^q(\Omega)}+[u]_{F^s_{p,q}(\Omega)}<\infty\right\}.
\]
In this convention, $p$ is the exponent in the inner integral and $q$
is the exponent in the outer integral. When $p=q$, this is the
fractional Sobolev space $W^{s,p}(\Omega)$ with its Gagliardo seminorm.
For $p<q$, the condition on $s$ reduces to $0<s<1$.
The $q$th power of the corresponding two-weight difference seminorm
on $\R^d$ is the integral in \eqref{eq:intro-mixed-functional}.

\paragraph{The variational quotient and weight parameters.}

In the $p<q$ part of the manuscript we assume
\begin{equation}\label{eq:parameters}
 d\ge2,\quad 0<s<1,\quad 1<p<q<\infty,\quad
 -d<\beta<sp.
\end{equation}
For $u\in\test$, define
\begin{align}
 A_{u,\beta}(x)&=\int_{\R^d}
 \frac{|u(x)-u(y)|^p|y|^\beta}{|x-y|^{d+sp}}\dd y,
 \label{eq:pointenergy}\\
 \E_{\alpha,\beta}[u]&=\int_{\R^d}|x|^\alpha A_{u,\beta}(x)^{q/p}\dd x,
 \label{eq:energy}\\
 \Pp_{\alpha,\beta}[u]&=\int_{\R^d}|u(x)|^q
 |x|^{\alpha+q\beta/p-sq}\dd x.\label{eq:denominator}
\end{align}
The subscript $\beta$ records the inner weight; the derived
parameters are given in \eqref{eq:reparam}.
The best constant is
\begin{equation}\label{eq:Hdef}
 H_{\alpha,\beta}^{p,q}
 =\inf_{0\ne u\in\test}\frac{\E_{\alpha,\beta}[u]}{\Pp_{\alpha,\beta}[u]}.
\end{equation}
We suppress the superscripts $p,q$ when they are fixed and use a
difference kernel with coefficient one. Introduce
\begin{align}
 D&=d+2\beta-sp,\qquad
 a=\frac{\alpha+d+q\beta/p-sq}{q},\label{eq:reparam}\\
 \alpha_-&=-d,\qquad \alpha_H=sq-d-q\beta/p,\qquad
 \alpha_+=sq+d(q/p-1),\nonumber\\
 \lambda_+&=sp-\beta>0,\qquad \lambda_-=d+\beta>0,
 \qquad D=\lambda_- -\lambda_+.\nonumber
\end{align}
Thus $\alpha=\alpha_H+qa$, and
\begin{equation}\label{eq:a-range}
 \alpha_-<\alpha<\alpha_+
 \quad\Longleftrightarrow\quad
 -\frac{\lambda_+}{p}<a<\frac{\lambda_-}{p}.
\end{equation}
For the proof we use the shorter notation
\begin{equation}\label{eq:Ja}
 J_{a,\beta}=J_{d,s,p}^{(\beta)}(a)
 =\int_{\R^d}\frac{|1-|z|^{-a}|^p|z|^\beta}
 {|e-z|^{d+sp}}\dd z,\qquad |e|=1.
\end{equation}
It is independent of $e$, finite on \eqref{eq:a-range}, zero at
$a=0$, and positive elsewhere on that interval.

All optimal constants below refer to the coefficient-one difference
seminorm and the specified smooth test class.

\Needspace{14\baselineskip}
\subsection{Sharpness intervals and strict extension}

We use \emph{elementary scalar criterion} to refer to the following
sufficient condition for nonnegativity of the scalar remainder
$G_{\vartheta,b}$ defined in \eqref{eq:Gtwo}:
\[
 0\le b\le\vartheta^{n_{p,q}}
 \quad\Longrightarrow\quad
 G_{\vartheta,b}(U,V)\ge0
 \qquad(0<\vartheta<1,\ U,V\ge0).
\]
This implication is proved in Lemma~\ref{lem:scalar-support}.
The exponent $n_{p,q}$, which also enters the proof of extension
for each sign of $D$, is defined by
\begin{equation}\label{eq:n-pq}
 n_{p,q}=\begin{cases}
 1+\dfrac{\log(q-1)}{\log(1-2^{-1/(p-1)})},&1<p<q<2,\\[2.5mm]
 1,&2\le q\le2p,\\
 q-p,&q>2p.
 \end{cases}
\end{equation}
The first line is finite and greater than one. The optimized exponent
$\nsc$ in \eqref{eq:nsc-main} is strictly smaller in the
low-exponent range.

\begin{theorem}[Sharpness intervals for all signs of $D$]\label{thm:extension}
Assume \eqref{eq:parameters}.
\begin{enumerate}[label=\textup{(\roman*)}]
\item If $D\ne0$, let
\[
 a_0=\frac{|D|}{p+n_{p,q}},\qquad \varepsilon_D=\sgn D.
\]
There exists $\eta=\eta(d,s,p,q,\beta)>0$ such that
\begin{equation}\label{eq:extension-main}
 H_{\alpha_H+qa,\beta}^{p,q}=J_{a,\beta}^{q/p}
 \qquad (0\le\varepsilon_D a\le a_0+\eta).
\end{equation}
The width $\eta$ is chosen so that all exponents remain in \eqref{eq:a-range}.
\item If $D=0$, there exists $\eta>0$ such that
\begin{equation}\label{eq:balanced-main}
 H_{\alpha_H+qa,\beta}^{p,q}=J_{a,\beta}^{q/p}\qquad (|a|\le\eta).
\end{equation}
\end{enumerate}
\end{theorem}

Figure~\ref{fig:sharpness-map} records the three cases on the $a$-axis
of \eqref{eq:reparam}, together with the proved strict-loss zones near
the endpoints of \eqref{eq:a-range}.

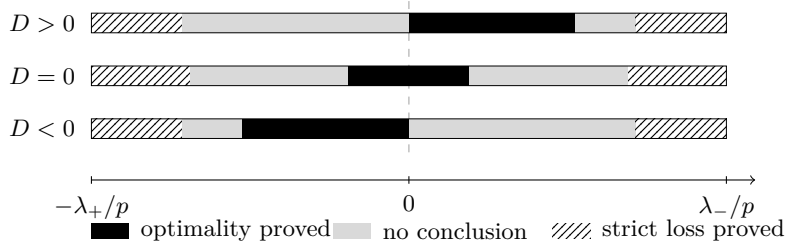
\begin{figure}[t]
\centering
\begin{tikzpicture}[x=1cm,y=1cm,font=\footnotesize]
\tikzset{bar/.style={line width=0.3pt}}
\draw[black!35,dashed] (4.2,0.42) -- (4.2,-1.62);
% D>0
\fill[pattern=north east lines] (0,0) rectangle (1.2,0.26);
\fill[black!15] (1.2,0) rectangle (4.2,0.26);
\fill[black] (4.2,0) rectangle (6.4,0.26);
\fill[black!15] (6.4,0) rectangle (7.2,0.26);
\fill[pattern=north east lines] (7.2,0) rectangle (8.4,0.26);
\draw[bar] (0,0) rectangle (8.4,0.26);
\node[left] at (-0.08,0.13) {$D>0$};
% D=0
\fill[pattern=north east lines] (0,-0.7) rectangle (1.3,-0.44);
\fill[black!15] (1.3,-0.7) rectangle (3.4,-0.44);
\fill[black] (3.4,-0.7) rectangle (5.0,-0.44);
\fill[black!15] (5.0,-0.7) rectangle (7.1,-0.44);
\fill[pattern=north east lines] (7.1,-0.7) rectangle (8.4,-0.44);
\draw[bar] (0,-0.7) rectangle (8.4,-0.44);
\node[left] at (-0.08,-0.57) {$D=0$};
% D<0
\fill[pattern=north east lines] (0,-1.4) rectangle (1.2,-1.14);
\fill[black!15] (1.2,-1.4) rectangle (2.0,-1.14);
\fill[black] (2.0,-1.4) rectangle (4.2,-1.14);
\fill[black!15] (4.2,-1.4) rectangle (7.2,-1.14);
\fill[pattern=north east lines] (7.2,-1.4) rectangle (8.4,-1.14);
\draw[bar] (0,-1.4) rectangle (8.4,-1.14);
\node[left] at (-0.08,-1.27) {$D<0$};
% axis
\draw[->] (0,-1.95) -- (8.78,-1.95);
\foreach \xp/\lb in {0/{$-\lambda_+/p$},4.2/{$0$},8.4/{$\lambda_-/p$}}
  {\draw (\xp,-1.90) -- (\xp,-2.00); \node[below] at (\xp,-2.02) {\lb};}
% legend
\begin{scope}[yshift=-2.72cm]
 \fill[black] (0,0) rectangle (0.5,0.2);
 \node[right] at (0.52,0.1) {optimality proved};
 \fill[black!15] (3.2,0) rectangle (3.7,0.2);
 \node[right] at (3.72,0.1) {no conclusion};
 \fill[pattern=north east lines] (6.1,0) rectangle (6.6,0.2);
 \node[right] at (6.62,0.1) {strict loss proved};
\end{scope}
\end{tikzpicture}
\caption{Schematic picture on the $a$-axis of \eqref{eq:reparam}, whose
admissible range is \eqref{eq:a-range}; $a=0$ corresponds to
$\alpha=\alpha_H$. Solid: $J_{a,\beta}^{q/p}$ is proved to be the exact
constant. Hatched: strict loss of optimality is proved near the
finite-energy endpoints. Light grey: no conclusion. For $D\ne0$ the
proved interval is $[0,a_0+\eta]$ or $[-a_0-\eta,0]$ with
$a_0=|D|/(p+n_{p,q})$, and for $D=0$ it is $[-\eta,\eta]$; neither
$\eta$ nor the first sharpness threshold is explicit.}
\label{fig:sharpness-map}
\end{figure}

\Needspace{12\baselineskip}
\subsubsection{The original-weight form when \texorpdfstring{$D>0$}{D>0}}

\begin{theorem}[Sharp Hardy inequalities and enlarged weight intervals]\label{thm:main}
Let $d\ge2$, $0<s<1$, $1<p<q<\infty$, and $-d<\beta<sp$ with $d+2\beta>sp$.
For the weight parameters $\alpha$ specified below, set
$a=(\alpha+d+q\beta/p-sq)/q$, and let $J_{d,s,p}^{(\beta)}(a)$ be the homogeneous
integral defined in \eqref{eq:intro-weighted-J}.
For every $u\in C_c^\infty(\R^d\setminus\{0\})$, the inequality
\begin{equation}\label{eq:main-original-integral}
\begin{aligned}
 &\int_{\R^d}|x|^\alpha
 \left(\int_{\R^d}
 \frac{|u(x)-u(y)|^p|y|^\beta}{|x-y|^{d+sp}}\dd y\right)^{q/p}\dd x
 \\
 &\quad\ge
 J_{d,s,p}^{(\beta)}(a)^{q/p}
 \int_{\R^d}|u(x)|^q|x|^{\alpha+q\beta/p-sq}\dd x
\end{aligned}
\end{equation}
holds with the best possible multiplicative constant in each of the following ranges.
\begin{enumerate}[label=\textup{(\roman*)}]
\item If $2\le q\le2p$, there exists $\varepsilon=\varepsilon(d,s,p,q,\beta)>0$ such that
\begin{equation}\label{eq:main-middle-alpha}
 sq-d-q\beta/p\le\alpha\le
 \frac{sq+(q-p-1)d+q(p-1)\beta/p}{p+1}+\varepsilon.
\end{equation}
\item If $q>2p$, there exists $\varepsilon=\varepsilon(d,s,p,q,\beta)>0$ such that
\begin{equation}\label{eq:main-high-alpha}
 sq-d-q\beta/p\le\alpha\le s(q-p)+(2-q/p)\beta+\varepsilon.
\end{equation}
\item If $1<p<q<2$, define
\begin{equation}\label{eq:nsc-main}
 \nsc=
 \max_{0<t<1}
 \frac{\displaystyle\log\!\left(
 \frac{q\bigl((1-t)^{1-p}-1\bigr)+(q-p)t}{p}
 \right)}{\log t}.
\end{equation}
Then the inequality holds throughout
\begin{equation}\label{eq:main-low-alpha}
 sq-d-q\beta/p\le\alpha\le
 sq-d-q\beta/p+\frac{q(d+2\beta-sp)}{p+\nsc}.
\end{equation}
\end{enumerate}
The value $a=(\alpha+d+q\beta/p-sq)/q$ lies in $((\beta-sp)/p,(d+\beta)/p)$ on all the displayed intervals.
The widths in \eqref{eq:main-middle-alpha}--\eqref{eq:main-high-alpha} are measured in $\alpha$ and are chosen so that the upper endpoints remain below $sq+d(q/p-1)$.
\end{theorem}

The maximum in \eqref{eq:nsc-main} is attained in $(0,1)$ and satisfies
\begin{equation}\label{eq:nsc-improvement-main}
 1<\nsc<
 1+\frac{\log(q-1)}{\log\!\left(1-2^{-1/(p-1)}\right)}.
\end{equation}
It is the optimal exponent in the pointwise scalar supporting-line criterion established in Proposition~\ref{prop:nsc}.

At $\alpha=sq-d-q\beta/p$, the displayed best constant is zero; it is strictly positive at the other parameters in the theorem. Optimality is witnessed by long logarithmic cutoffs of the homogeneous power: for any nonzero nonnegative $\zeta\in C_c^\infty(\R)$,
\begin{equation}\label{eq:main-original-cutoff}
 u_T(x)=|x|^{-\frac{\alpha+d+q\beta/p-sq}{q}}
 \zeta\!\left(\frac{\log|x|}{T}\right),\qquad T\longrightarrow\infty,
\end{equation}
has quotient converging to $J_{d,s,p}^{(\beta)}((\alpha+d+q\beta/p-sq)/q)^{q/p}$.
Appendix~\ref{app:evaluations} records the established radial and
quadratic evaluations for the kernel used here, whose multiplicative factor is one.

The low-exponent interval in \eqref{eq:main-low-alpha} strictly contains the interval obtained from the elementary scalar criterion in Lemma~\ref{lem:scalar-support},
whose upper endpoint is
\[
 sq-d-q\beta/p+\frac{q(d+2\beta-sp)}{
 p+1+\dfrac{\log(q-1)}{\log(1-2^{-1/(p-1)})}}.
\]
In this range, the enlargement follows from a sharper pointwise support criterion at every logarithmic scale. For $q>2p$, the extension instead retains positivity on a finite range of scales and absorbs the negative remote scales.

\paragraph{Relation between Theorems~\ref{thm:main} and~\ref{thm:extension}.}
Both theorems identify the same best constant, namely
$J_{d,s,p}^{(\beta)}(a)^{q/p}$, on the same smooth test class.
Their different roles concern the intervals on which sharpness is proved, not the
value of the constant or the class of test functions. More precisely:
\begin{enumerate}[label=\textup{(\alph*)}]
\item \emph{The right-hand interval when $D>0$.}
Here $a\ge0$ and $\alpha=\alpha_H+qa\ge\alpha_H$.
For $2\le q\le2p$, the elementary endpoint is $D/(p+1)$;
for $q>2p$, it is $D/q$. Theorem~\ref{thm:extension}(i)
extends sharpness by a positive width beyond these endpoints.
After the change of variables $\varepsilon=q\eta$, these are
exactly the assertions of Theorem~\ref{thm:main}(i)--(ii).
Thus, in these two exponent ranges, Theorem~\ref{thm:extension}
does not give an additional enlargement of Theorem~\ref{thm:main}.
\item \emph{The low-exponent range, $D>0$ and $1<p<q<2$.}
Theorem~\ref{thm:extension}(i) asserts the existence of a positive
extension beyond $D/(p+n_{p,q})$. Theorem~\ref{thm:main}(iii)
instead gives the specified larger endpoint
\[
 \frac{D}{p+\nsc}>\frac{D}{p+n_{p,q}},
\]
by optimizing the scalar supporting-line criterion.
The increase over the elementary endpoint is therefore
\[
 \frac{qD(n_{p,q}-\nsc)}
 {(p+\nsc)(p+n_{p,q})}>0
\]
in the original variable $\alpha$. This explicit enlargement
supplies the existence assertion in Theorem~\ref{thm:extension}(i)
after choosing $\eta$ sufficiently small.
\item \emph{The left-hand and two-sided intervals when $D\le0$.}
These are outside the assumptions of Theorem~\ref{thm:main}.
For $D<0$, Theorem~\ref{thm:extension}(i) gives the left-hand interval
\[
 \alpha_H-q\left(\frac{|D|}{p+n_{p,q}}+\eta\right)
 \le\alpha\le\alpha_H.
\]
For $D=0$, part~\textup{(ii)} gives a two-sided interval
$\alpha_H-q\eta\le\alpha\le\alpha_H+q\eta$.
The condition $D>0$ in Theorem~\ref{thm:main} thus selects its
right-hand sharpness branch; it is not required for the existence
of a Hardy inequality or for finiteness of the test-function energy.
\end{enumerate}

\Needspace{14\baselineskip}
\subsection{Endpoint loss and the first sharpness threshold}

\begin{theorem}[Loss of sharpness near the endpoints]\label{thm:loss}
Let $\Lambda=\alpha_+>0$ and $\theta=sq/\Lambda\in(0,1)$. As $\epsilon\downarrow0$,
\begin{equation}\label{eq:endpoint-upper}
 H_{\alpha_+-\epsilon,\beta}^{p,q}+H_{\alpha_-+\epsilon,\beta}^{p,q}
 \le C\epsilon^{-\theta},
\end{equation}
whereas at either endpoint the homogeneous candidate satisfies
\begin{equation}\label{eq:Jendpoint}
 J_{a,\beta}\sim\frac{q\omega_{d-1}}{p\epsilon}.
\end{equation}
Consequently $H_{\alpha,\beta}^{p,q}<J_{a,\beta}^{q/p}$ sufficiently close to either endpoint.
If $D>0$, define $a_{\sh}$ as the supremum of all
$b\in[0,(d+\beta)/p)$ such that
\[
 H_{\alpha_H+qa,\beta}^{p,q}=J_{a,\beta}^{q/p}
 \qquad\text{for every }a\in[0,b].
\]
Then
\begin{equation}\label{eq:threshold-existence}
 \frac{D}{p+n_{p,q}}<a_{\sh}<\frac{d+\beta}{p}.
\end{equation}
Sharpness holds for every $a\in[0,a_{\sh}]$.
For every $0<\delta<(d+\beta)/p-a_{\sh}$, there exists
$a\in(a_{\sh},a_{\sh}+\delta)$ such that
\[
 H_{\alpha_H+qa,\beta}^{p,q}<J_{a,\beta}^{q/p}.
\]
\end{theorem}

\begin{remark}[The initial sharpness interval]
The threshold $a_{\sh}$ marks the end of the sharpness interval
starting at zero. Its relation to the full sharpness set and to the
scalar and angular criteria is discussed in
Remark~\ref{rem:threshold-comparison}.
\end{remark}

\Needspace{12\baselineskip}
\subsection{Finiteness, positivity, and the zero-constant weight}

\begin{theorem}[Finiteness and positivity]\label{thm:basic}
For every nonzero $u\in\test$, the energy $\E_{\alpha,\beta}[u]$ is finite if and only if $\alpha_-<\alpha<\alpha_+$. On this interval,
\begin{equation}\label{eq:positive-range}
 0<H_{\alpha,\beta}^{p,q}<\infty\quad\text{if }\alpha\ne\alpha_H,
 \qquad H_{\alpha_H,\beta}^{p,q}=0.
\end{equation}
Moreover, $H_{\alpha,\beta}^{p,q}\le J_{a,\beta}^{q/p}$, and for fixed $d,s,p,q,\beta$,
\begin{equation}\label{eq:vanishing}
 H_{\alpha,\beta}^{p,q}\asymp |\alpha-\alpha_H|^q
 \qquad\text{as }\alpha\to\alpha_H.
\end{equation}
\end{theorem}

The sharpness and strict-loss conclusions above concern parameters
inside this finite-energy interval. In particular, the loss of
homogeneous optimality in Theorem~\ref{thm:loss} occurs while
both the optimal constant and the candidate remain finite.

\subsection{Proof strategy}

For $D>0$ and $a>0$, applying this criterion at every logarithmic
displacement means requiring
\[
 e^{-(D-pa)t}\le(e^{-at})^{n_{p,q}}
 \quad\text{for every }t>0,
 \qquad\text{equivalently}\qquad
 a\le\frac{D}{p+n_{p,q}}.
\]
By Lemma~\ref{lem:scalar-support}, it guarantees that the paired
scalar remainder in \eqref{eq:Gtwo} is nonnegative for every pair
of nonnegative real numbers and at every logarithmic displacement.
Integrating this inequality gives the Hardy lower bound with the
homogeneous candidate constant; the cutoff sequence then proves
optimality.

To extend beyond this elementary range when $q\ge2$, the proof uses
$u(e^\tau\omega)=e^{-a\tau}v(\tau,\omega)$ and the kernel
\[
 K_{\beta,t}(\omega,\sigma)
 =\frac{e^{(d+\beta)t}}{|\omega-e^t\sigma|^{d+sp}},\qquad
 K_{\beta,-t}(\sigma,\omega)=e^{-(d+2\beta-sp)t}
 K_{\beta,t}(\omega,\sigma).
\]
This reversal explains the appearance of $d+2\beta-sp$ in the
sharp-interval formulas. Pairing opposite displacements reduces the
lower bound to the scalar remainder $G$ defined in \eqref{eq:Gtwo}. At the endpoint
of the elementary criterion, $G>0$ whenever its two arguments differ.
Integrating over $1\le t\le2$ therefore controls the reference energy
$\eref[v^{q/2}]$, defined in \eqref{eq:Eref}. This energy measures
squared differences of $v^{q/2}$ between pairs of points at those
displacements. For nearby parameters, the remainder stays nonnegative
for small $t$, while the negative contributions at larger $t$ are
bounded by a coefficient tending to zero times the same reference
energy. The positive contribution can therefore absorb them.

For a flat singular set of codimension $m$, tangential integration
preserves the scalar structure and gives $m+2\beta-sp$ in the reversal.
The needed estimate for tangential differences follows by composing
the shifts $(y,\tau)\mapsto(y+e^\tau\eta,\tau+t)$. The ambient dimension
still controls the far field of compact tests, whereas $m$ controls
the homogeneous profile. Section~\ref{sec:flat} states both ranges
separately and proves the flat inequality also for $p>q$.
Section~\ref{sec:halfspace} treats the regional half-space kernel,
which contains only same-side normal interactions. It yields sharp
$p<q$ intervals and strict extension, together with loss of sharpness
near the endpoints of the finiteness range of the homogeneous integral.

\subsection{Historical context: sharp constants and homogeneous powers}

Hardy's integral inequality~\cite{Hardy1920} led to the
multidimensional model
\begin{equation}\label{eq:intro-local-Hardy}
 \int_{\R^d}|\nabla u|^p\dd x
 \ge\left(\frac{d-p}{p}\right)^p
 \int_{\R^d}\frac{|u(x)|^p}{|x|^p}\dd x,
 \qquad 1<p<d,\quad u\in C_c^\infty(\R^d\setminus\{0\}).
\end{equation}
The singular power predicts the constant, but its truncations, not the
power itself, are admissible tests; this distinction persists in the
fractional theory. Table~\ref{tab:sharp-constants} collects the sharp
constants known for the weighted energies used below, in the
coefficient-one normalization of Appendix~\ref{app:evaluations}.

The single-exponent fractional case is the baseline. For $0<s<1$,
$1<p<\infty$ and $sp\ne d$,
\begin{equation}\label{eq:intro-fractional-Hardy}
 \int_{\R^d}\int_{\R^d}
 \frac{|u(x)-u(y)|^p}{|x-y|^{d+sp}}\dd y\dd x
 \ge C_{d,s,p}\int_{\R^d}\frac{|u(x)|^p}{|x|^{sp}}\dd x,
\end{equation}
with the optimal constant $C_{d,s,p}=J_{d,s,p}((d-sp)/p)$
of~\cite[Theorem~1.1 and Eqs.~(1.4)--(1.5)]{FS}. The angular kernel of
the operator energy is, with $|e|=1$ and ordinary surface measure,
\begin{equation}\label{eq:intro-Phi}
\begin{aligned}
 \Phi_{d,s,p}(r)
 &=\int_{\Sph^{d-1}}|e-r\omega|^{-d-sp}\dd\omega\\
 &=|\Sph^{d-2}|\int_{-1}^{1}
 \frac{(1-t^2)^{(d-3)/2}}{(1-2rt+r^2)^{(d+sp)/2}}\dd t
 \qquad(d\ge2),
\end{aligned}
\end{equation}
and $\Phi_{1,s,p}(r)=(1-r)^{-1-sp}+(1+r)^{-1-sp}$. The radial
inversion of~\cite[Section~3]{FS}, extended to the mixed formula
by \Kij\ (Theorem~1, Eq.~(4)), evaluates the weighted candidate
$J^{(\beta)}_{d,s,p}(a)$ of \eqref{eq:intro-weighted-J} as a
one-dimensional integral against $\Phi_{d,s,p}$;
Appendix~\ref{app:evaluations} records that formula and the Gamma
evaluations at $p=2$ of~\cite[Eqs.~(3.3)--(3.5)]{FLS}. These
established calculations identify the candidate. The present question
is its global sharpness when $p<q$.

Quantitative fractional estimates and weighted
Caffarelli--Kohn--Nirenberg inequalities were obtained by
Heinig--Kufner--Persson~\cite{HKP},
Maz'ya--Shaposhnikova~\cite{MS} and
Abdellaoui--Bentifour~\cite{AB}; a framework for irregular open sets
is due to Dyda--V\"ah\"akangas~\cite{DV}, and the corresponding
function-space conditions to Ihnatsyeva--V\"ah\"akangas~\cite{IV,IV2}.
For $p=2$ the spectral normalization gives the constant
$4^s[\Gamma((d+2s)/4)/\Gamma((d-2s)/4)]^2$ of
Herbst~\cite{Herbst}, Beckner~\cite{Beckner95} and
Yafaev~\cite{Yafaev}, which is $C^{\mathrm{FL}}_{d,s}/2$ times the
coefficient-one quadratic energy; see Appendix~\ref{app:evaluations}.

\begin{table}[tbp]
\centering
\small
\caption{Sharp constants for related energies; all are optimal on the
stated test class, in the normalization of
Appendix~\ref{app:evaluations}. The exponent is $p$ unless a mixed
order is indicated.}
\label{tab:sharp-constants}
\begin{tabular}{@{}p{1.9in}p{2.8in}p{1.15in}@{}}
\toprule
Setting & Sharp constant or conclusion & Reference\\
\midrule
$\R^d$, one power weight $|x|^{-sp}$
 & $C_{d,s,p}=J_{d,s,p}\bigl((d-sp)/p\bigr)$ & \cite{FS}\\
$\R^d$, two power weights
 & $C^{\mu,\nu}_{d,s,p}$ by radial inversion & \cite{DKWeighted}\\
$\R^d$, $p=2$
 & Gamma and trigonometric forms & \cite{FLS}\\
$\R^d$, remainder terms
 & $c_p$ times the weighted difference energy & \cite{FS, DKLow}\\
half-space, regional energy
 & $D_{d,s,p}$ & \cite{BD, FSHalf}\\
convex $\Omega$, regional energy, $sp>1$
 & $D_{d,s,p}$ with the distance weight & \cite{LS}\\
convex $\Omega$, zero extension
 & half-space value $h_{d,s,p}$ if $p=2$ or $1/p\le s$ & \cite{BBZ}\\
submanifold $K$, codimension $m$
 & $\operatorname{dist}(x,K)^{-sp}$, log weight if $sp=m$ & \cite{ARS}\\
flat $\Sigma=\R^{d-m}\times\{0\}$
 & tangential Gamma factor times $C^{\mu,\nu}_{m,s,p}$
 & \cite{KS, Sahu, Mallick}\\
point, mixed exponents $q\le p$
 & $J^{(\beta)}_{d,s,p}(a)^{q/p}$ on his test classes & \cite{KijTL}\\
half-space, mixed exponents $p\ge q$
 & comparison on the overlap of hypotheses & \cite{KijTL}\\
\bottomrule
\end{tabular}
\end{table}

For boundary singularities the regional energy and the energy of the
zero extension must be distinguished. For $u\in C_c^\infty(\Omega)$,
extended by zero to $\R^d$, the exact relation is
\begin{equation}\label{eq:intro-regional-full}
\begin{aligned}
 &\int_{\R^d}\int_{\R^d}
 \frac{|u(x)-u(y)|^p}{|x-y|^{d+sp}}\dd y\dd x\\
 &=\int_\Omega\int_\Omega
 \frac{|u(x)-u(y)|^p}{|x-y|^{d+sp}}\dd y\dd x
 +2\int_\Omega |u(x)|^p
 \int_{\R^d\setminus\Omega}|x-y|^{-d-sp}\dd y\dd x.
\end{aligned}
\end{equation}
Dyda~\cite{Dyda04} established regional fractional boundary Hardy
inequalities on several classes of domains, and lower-order
Hardy--Sobolev--Maz'ya improvements were obtained
by~\cite{DydaRem,Sloane}.

The two energies have different sharp constants, so the half-space
values in Table~\ref{tab:sharp-constants} do not transfer directly to
the regional functional of Section~\ref{sec:halfspace}. For convex
domains the case $sp<1$, $p\ne2$ was left open by the argument
of~\cite{BBZ} and received a partial negative answer at $p=1$
in~\cite{BSZ}.

Two aspects are not covered by the cited work. In the
weighted mixed-exponent setting only the order $q\le p$ was settled;
the sharpness transitions for $p<q$ are the subject of this paper.
For flat singular sets the listed constants concern a single exponent,
whereas Section~\ref{sec:flat} treats the mixed-exponent flat energy:
the effective codimension and the tangential Gamma factor are not new,
but their mixed-exponent counterparts are.

\subsection{Mixed exponents and the case \texorpdfstring{$p>q$}{p>q}}\label{sec:mixed-background}

For $1<q\le p$, \Kij\ (Theorem~1) proved
\begin{equation}\label{eq:intro-Kij}
\begin{aligned}
 &\int_{\R^d}|x|^\alpha
 \left(\int_{\R^d}\frac{|u(x)-u(y)|^p|y|^\beta}
 {|x-y|^{d+sp}}\dd y\right)^{q/p}\dd x\\
 &\quad\ge J_{d,s,p}^{(\beta)}(a)^{q/p}
 \int_{\R^d}|u(x)|^q|x|^{\alpha+q\beta/p-sq}\dd x
\end{aligned}
\end{equation}
with the optimal constant, under \eqref{eq:intro-weight-range},
$\alpha+q\beta/p>-d$, and away from the zero-constant weight.
His choice of test class depends on the sign of
$sq-\alpha-q\beta/p-d$. Our smooth test class, whose functions vanish near the origin, is contained
in either of his classes. For $p>q$ his theorem therefore applies
directly to our test class, and Appendix~\ref{sec:reverse-order}
collects the scalar estimates that the flat and half-space arguments
use.

Kijaczko's Theorem~2 treats the regional half-space energy, again for
$p\ge q$. Thus his mixed-exponent results cover the whole space and
the half-space; they do not include a singular flat subspace of
arbitrary codimension in the whole space. The equal-exponent flat
results of~\cite{Sahu,KS} concern a separate exponent regime.
Theorem~\ref{thm:flat-reverse} below supplies the flat $p>q$
inequality on its stated smooth test class. The half-space
$p>q$ statement in Theorem~\ref{thm:halfspace} is a comparison with
Kijaczko's theorem on the overlap of hypotheses; the new half-space
sharpness intervals concern $p<q$. In both geometries we verify
ambient far-field integrability directly.

The scalar background includes Brasco--Franzina's discrete Picone
inequality~\cite[Proposition~4.2]{BF}. For $1<q\le p$, $w>0$, and
$v\ge0$, it states
\begin{equation}\label{eq:intro-Picone}
\begin{aligned}
 &|w(x)-w(y)|^{p-2}(w(x)-w(y))
 \left(\frac{v(x)^q}{w(x)^{q-1}}-
       \frac{v(y)^q}{w(y)^{q-1}}\right)\\
 &\qquad\le |v(x)-v(y)|^q|w(x)-w(y)|^{p-q}.
\end{aligned}
\end{equation}
The exponent ordering in this principle and the mixed-norm argument
cannot simply be reversed. Kijaczko explicitly left the sharpness of
his formulas for $p<q$ open. We prove sharpness of the established
candidate on nontrivial two-weight regions and loss of sharpness near
both finite-energy endpoints in $\alpha$, for each fixed admissible
$\beta$. The parameter controlling paired scales is
$d+2\beta-sp$, whereas the homogeneous exponent involves
$\alpha+q\beta/p$. These are different effects of the inner weight.

In the complementary order $p>q$ the scalar step is the discrete
Picone inequality, whose weight is chosen from the homogeneous power.
The estimates that the flat and half-space arguments need are collected
in Appendix~\ref{sec:reverse-order}. The sharpness transitions studied
in the main part of the paper concern $p<q$.

\subsection{Triebel--Lizorkin spaces and the choice of seminorm}
\label{sec:TL-setting}

Section~\ref{sec:problem} uses Kijaczko's difference-space notation
$F^s_{p,q}$, with inner exponent $p$ and outer exponent $q$.
Here we use $\widetilde F^s_{r,\ell}$ for the \emph{classical
inhomogeneous Triebel--Lizorkin space}: its first subscript $r$ is
the spatial $L^r$ integrability exponent, and its second subscript
$\ell$ is the summability exponent over dyadic frequency scales.

The unweighted first-order difference characterization of
Prats and Saksman~\cite[Theorem~1.2]{PS} states that, for
$1<p,q<\infty$ and $\max\{0,d/q-d/p\}<s<1$,
\begin{equation}\label{eq:intro-TL-characterization}
 \|u\|_{\widetilde F^s_{q,p}(\R^d)}
 \asymp \|u\|_{L^q(\R^d)}
       +\E_{0,0}[u]^{1/q}.
\end{equation}
Thus the inner exponent $p$ and outer exponent $q$ used in this paper
correspond to the classical space $\widetilde F^s_{q,p}$, while our
difference-space notation is $F^s_{p,q}$. The smoothness restriction
in this characterization is automatic throughout the main regime
$p<q$, $0<s<1$.

For general $\alpha,\beta$, the optimal constants are attached to the
exact difference seminorm $\E_{\alpha,\beta}^{1/q}$, with the stated
kernel normalization. Corollary~\ref{cor:TL-completion} realizes the
completion of the smooth test class in this norm and extends the same
positive optimal constant to it. Identification with other weighted
function spaces entails a separate norm or density comparison.

\subsection{Organization}

Section~\ref{sec:elementary} establishes the finite-energy range,
positivity, the exact cylindrical representation, and the homogeneous
cutoff upper bound. It also identifies the completion of the smooth
test class and extends positive optimal constants to that space.
Sections~\ref{sec:scalar} and~\ref{sec:reference} give the complete
scalar support estimates and the reference-energy control of all
logarithmic displacements. Section~\ref{sec:proofs} proves the
sharpness and strict-extension results, including the low-exponent
optimization, the negative and balanced cases, the alternative
high-exponent proof, and the quantitative remainder.

Sections~\ref{sec:endpoint-loss} and~\ref{sec:explicit} complete the
point-singularity analysis with endpoint loss, the first sharpness
threshold, and angular tests. The geometric applications follow:
Section~\ref{sec:flat} treats flat singular sets for both exponent
orders, and Section~\ref{sec:halfspace} treats regional half-space
energies. Section~\ref{sec:questions} records the remaining questions.
The appendices record the established constant evaluations, the
all-parameter scalar comparison for $p>q$, and the auxiliary estimates
for that exponent order. All sharpness constructions use the original
smooth test classes.

\section{Elementary estimates and the logarithmic cylinder}\label{sec:elementary}

\subsection{Finiteness of the energy}

We first prove the energy-finiteness assertion of
Theorem~\ref{thm:basic}.

First, if $\beta\le-d$ or $\beta\ge sp$, then $A_{u,\beta}(x)=\infty$
at every point with $u(x)\ne0$: integrate respectively near $y=0$
or at $|y|=\infty$, where the test function
vanishes. Thus $-d<\beta<sp$ is necessary.
Fix such a $\beta$ and $0\ne u\in\test$. Near $x=0$, $u(x)=0$,
and dominated convergence yields
\[
 A_{u,\beta}(x)\longrightarrow
 \int_{\R^d}|u(y)|^p|y|^{\beta-d-sp}\dd y>0.
\]
The outer integral is finite there exactly when $\alpha>-d$.
At infinity, uniformly in direction,
\begin{equation}\label{eq:compacttail}
 A_{u,\beta}(x)=|x|^{-d-sp}
 \left(\int_{\R^d}|u(y)|^p|y|^\beta\dd y\right)
 (1+O(|x|^{-1})).
\end{equation}
Consequently the outer tail is finite exactly when
$\alpha<q(d+sp)/p-d=\alpha_+$. On bounded sets separated from
zero, split the inner integral into a diagonal neighborhood, a
neighborhood of zero, and a tail. Smoothness controls the first by
$\int_0^1r^{p-sp-1}\dd r$; the other two are controlled by
$\beta>-d$ and $\beta<sp$. These bounds prove the exact energy
range in Theorem~\ref{thm:basic}. The denominator is finite and
positive on the stated test class for every $\alpha,\beta$.
In particular, $\alpha+q\beta/p>-d$ is unnecessary on this class.

\subsection{A positive, nonsharp lower bound}

We next prove the positivity assertion of Theorem~\ref{thm:basic}
for $a\ne0$, as stated in \eqref{eq:positive-range}.

Let $\delta=\alpha+d+q\beta/p-sq=qa$. For $\delta>0$, define the outward annular average
\[
 T_+u(x)=\frac{1}{V_+|x|^d}\int_{|x|<|y|<2|x|}u(y)\dd y,
 \qquad V_+=\frac{\omega_{d-1}}d(2^d-1).
\]
Write $E_x^+=\{y:|x|<|y|<2|x|\}$, so
$|E_x^+|=V_+|x|^d$. Jensen's inequality for the probability measure
$\dd y/(V_+|x|^d)$ on this annulus gives
\[
 |u(x)-T_+u(x)|^p
 \le\frac1{V_+|x|^d}\int_{E_x^+}|u(x)-u(y)|^p\dd y.
\]
On $E_x^+$ we also have $|x-y|\le3|x|$ and
$|y|^\beta\ge c_{\beta,+}|x|^\beta$, where
$c_{\beta,+}=\min\{1,2^\beta\}>0$. Hence
\begin{equation}\label{eq:weighted-annular-pointwise}
\begin{aligned}
 A_{u,\beta}(x)
 &\ge\frac{c_{\beta,+}|x|^{\beta-d-sp}}{3^{d+sp}}
       \int_{E_x^+}|u(x)-u(y)|^p\dd y\\
 &\ge\frac{c_{\beta,+}V_+}{3^{d+sp}}
       |x|^{\beta-sp}|u(x)-T_+u(x)|^p.
\end{aligned}
\end{equation}
Raising this inequality to $q/p$ and multiplying by $|x|^\alpha$
produces precisely the denominator weight
$|x|^{\alpha+q\beta/p-sq}$.
Hence, with $\|\cdot\|_{q,\alpha,\beta}$ denoting the norm in $L^q(|x|^{\alpha+q\beta/p-sq}\dd x)$,
\begin{equation}\label{eq:annularlower}
 \E_{\alpha,\beta}[u]\ge
 \left(\frac{c_{\beta,+}V_+}{3^{d+sp}}\right)^{q/p}
 \|u-T_+u\|_{q,\alpha,\beta}^q.
\end{equation}
To justify the contraction explicitly, let
\[
 Mu(r)=\frac1{\omega_{d-1}}\int_{\Sph^{d-1}}u(r\sigma)\dd\sigma.
\]
This is a function of the radius $r$. We distinguish its one-dimensional
weighted norm from the full-space norm by writing
\[
 \|h\|_{\mathrm{rad}}
 :=\left(\int_0^\infty |h(r)|^q
       r^{\alpha+q\beta/p-sq+d-1}\dd r\right)^{1/q}.
\]
For a radial function $f(x)=h(|x|)$, polar coordinates give
\[
 \|f\|_{q,\alpha,\beta}
 =\omega_{d-1}^{1/q}\|h\|_{\mathrm{rad}}.
\]
Polar coordinates in the averaging annulus give
\[
 T_+u(r\omega)=\frac d{2^d-1}\int_1^2 t^{d-1}Mu(rt)\dd t.
\]
Jensen's inequality on each sphere, followed by radial integration, yields
\[
 \begin{aligned}
 \|Mu\|_{\mathrm{rad}}^q
 &\le\frac1{\omega_{d-1}}
   \int_0^\infty\int_{\Sph^{d-1}}|u(r\sigma)|^q
        r^{\alpha+q\beta/p-sq+d-1}\dd\sigma\dd r\\
 &=\frac1{\omega_{d-1}}\|u\|_{q,\alpha,\beta}^q.
 \end{aligned}
\]
The change of variables $\rho=tr$, together with
$\alpha+d+q\beta/p-sq=qa$, gives the radial dilation identity
\[
 \|Mu(t\,\cdot)\|_{\mathrm{rad}}
 =t^{-a}\|Mu\|_{\mathrm{rad}}.
\]
Minkowski's inequality for the probability measure
$\frac{d}{2^d-1}\,t^{d-1}\dd t$ on $(1,2)$ therefore yields
\[
 \begin{aligned}
 \|T_+u\|_{q,\alpha,\beta}
 &\le\omega_{d-1}^{1/q}\frac d{2^d-1}
       \int_1^2t^{d-1}\|Mu(t\,\cdot)\|_{\mathrm{rad}}\dd t\\
 &=\omega_{d-1}^{1/q}k_+(a)\|Mu\|_{\mathrm{rad}}\\
 &\le k_+(a)\|u\|_{q,\alpha,\beta},
 \end{aligned}
 \qquad
 k_+(a)=\frac d{2^d-1}\int_1^2t^{d-1-a}\dd t<1
\]
when $a>0$. Thus
\begin{equation}\label{eq:positiveplus}
 H_{\alpha,\beta}^{p,q}\ge
 \left(\frac{c_{\beta,+}V_+}{3^{d+sp}}\right)^{q/p}(1-k_+(a))^q>0.
\end{equation}
For $a<0$, define the inward annular average by
\[
 T_-u(x)=\frac1{V_-|x|^d}\int_{|x|/2<|y|<|x|}u(y)\dd y,
 \qquad V_-=\frac{\omega_{d-1}}d(1-2^{-d}).
\]
On this annulus $|x-y|\le2|x|$ and
$|y|^\beta\ge c_{\beta,-}|x|^\beta$, with
$c_{\beta,-}=\min\{1,2^{-\beta}\}>0$. The same normalized Jensen
argument therefore gives
\[
 A_{u,\beta}(x)\ge\frac{c_{\beta,-}V_-}{2^{d+sp}}
 |x|^{\beta-sp}|u(x)-T_-u(x)|^p.
\]
The corresponding contraction factor is
\[
 k_-(a)=\frac d{1-2^{-d}}\int_{1/2}^1t^{d-1-a}\dd t<1.
\]
Consequently
\[
 H_{\alpha,\beta}^{p,q}\ge
 \left(\frac{c_{\beta,-}V_-}{2^{d+sp}}\right)^{q/p}
 (1-k_-(a))^q>0\qquad(a<0).
\]
Together with \eqref{eq:positiveplus}, this proves the positivity
assertion of Theorem~\ref{thm:basic} for $a\ne0$.

\subsection{Exact cylindrical representation}

We derive the cylindrical identities used in
Lemma~\ref{lem:cutoff} and in the later sharpness proofs.

Let $\mathcal C=\R\times\Sph^{d-1}$ with measure
$\dd\tau\dd\omega$, and set
\begin{equation}\label{eq:cylinderchange}
 u(e^\tau\omega)=e^{-a\tau}v(\tau,\omega),\qquad
 K_{\beta,t}(\omega,\sigma)=\frac{e^{(d+\beta)t}}
 {|\omega-e^t\sigma|^{d+sp}}.
\end{equation}
For $x=e^\tau\omega$, put $y=e^{\tau+t}\sigma$. Under this
identification the spherical coordinate is the direction of the
point, that is, $\widehat x=\omega$; correspondingly
$\widehat y=\sigma$. The inner
Jacobian and weight are
$\dd y=e^{d(\tau+t)}\dd t\dd\sigma$ and
$|y|^\beta=e^{\beta(\tau+t)}$. Hence
\[
 A_{u,\beta}(e^\tau\omega)=e^{-(pa+sp-\beta)\tau}\Aa_{a,\beta}[v](\tau,\omega).
\]
The outer exponent vanishes because
$\alpha+d-q(a+s-\beta/p)=0$. Writing $\rho=q/p>1$, we obtain
\begin{gather}
 \E_{\alpha,\beta}[u]=\int_{\mathcal C}\Aa_{a,\beta}[v]^\rho,
 \qquad
 \Pp_{\alpha,\beta}[u]=\int_{\mathcal C}|v|^q,
 \label{eq:cylinderenergy}\\
 \Aa_{a,\beta}[v](\tau,\omega)
 =\int_\R\int_{\Sph^{d-1}}K_{\beta,t}(\omega,\sigma)
 |v(\tau,\omega)-e^{-at}v(\tau+t,\sigma)|^p\dd\sigma\dd t.
 \label{eq:cylinderA}
\end{gather}
Since $|\sigma-e^{-t}\omega|=e^{-t}|\omega-e^t\sigma|$,
\begin{equation}\label{eq:reversal}
 K_{\beta,-t}(\sigma,\omega)
 =e^{-(d+2\beta-sp)t}K_{\beta,t}(\omega,\sigma)=e^{-Dt}K_{\beta,t}(\omega,\sigma).
\end{equation}
This is why the relevant parameter is $d+2\beta-sp$.
Integrating the kernel in either angular variable gives the same
quantity, independent of the remaining angle; only rotation invariance
is used. Denote this quantity by $k_\beta(t)$:
\[
 k_\beta(t)=\int_{\Sph^{d-1}}K_{\beta,t}(\omega,\sigma)\dd\sigma.
\]
Integrating instead in the base variable,
$\int_{\Sph^{d-1}}K_{\beta,t}(\omega,\sigma)\dd\omega$, gives the
same value. By Fubini's theorem, integrating the kernel against a
function of $\omega$ alone, or of $\sigma$ alone, therefore reduces
to a single integral against $k_\beta(t)$.
It satisfies
\begin{equation}\label{eq:k0bounds}
 k_\beta(t)\asymp|t|^{-1-sp}\ (t\to0),\quad
 k_\beta(t)\sim\omega_{d-1}e^{-\lambda_+t}\ (t\to+\infty),\quad
 k_\beta(t)\sim\omega_{d-1}e^{\lambda_-t}\ (t\to-\infty).
\end{equation}
Indeed, $K_{\beta,t}=e^{\beta t}K_{0,t}$. At small $t$, use
$|\omega-e^t\sigma|^2=(1-e^t)^2+e^t|\omega-\sigma|^2$;
a local angular chart reduces the singularity to
$\int_{|z|<c}(t^2+|z|^2)^{-(d+sp)/2}\dd z\asymp|t|^{-1-sp}$.
At positive infinity factor $e^{(d+sp)t}$ from the denominator;
at negative infinity the denominator tends uniformly to one.
Finally,
\begin{equation}\label{eq:Jcylinder}
 J_{a,\beta}=\int_\R k_\beta(t)|1-e^{-at}|^p\dd t.
\end{equation}

\Needspace{14\baselineskip}
\subsection{Logarithmic cutoffs and an upper bound for the best constant}

\begin{lemma}[Slow logarithmic cutoffs]\label{lem:cutoff}
For every $-\lambda_+/p<a<\lambda_-/p$,
\begin{equation}\label{eq:Jupper}
 H_{\alpha_H+qa,\beta}^{p,q}\le J_{a,\beta}^{q/p}.
\end{equation}
More precisely, if $0\ne\zeta\in C_c^\infty(\R)$ is nonnegative and
\[
 u_T(x)=|x|^{-a}\zeta\left(\frac{\log|x|}{T}\right),
\]
then, with $\alpha=\alpha_H+qa$, the Hardy quotient satisfies
\[
 \lim_{T\to\infty}
 \frac{\displaystyle
 \int_{\R^d}|x|^\alpha
 \left(\int_{\R^d}
 \frac{|u_T(x)-u_T(y)|^p|y|^\beta}{|x-y|^{d+sp}}\dd y
 \right)^{q/p}\dd x}
 {\displaystyle
 \int_{\R^d}|u_T(x)|^q|x|^{\alpha+q\beta/p-sq}\dd x}
 =J_{a,\beta}^{q/p}.
\]
\end{lemma}
\begin{proof}
Write $\zeta_T(\tau)=\zeta(\tau/T)$, and set
\[
 \|F\|_{q,p,\beta}:=
 \left(\int_\R
 \left(\int_\R k_\beta(t)|F(\tau,t)|^p\dd t\right)^{q/p}
 \dd\tau\right)^{1/q}.
\]
For $u_T$, the weighted difference in \eqref{eq:cylinderA} becomes
\[
 \zeta_T(\tau)-e^{-at}\zeta_T(\tau+t)
 =M_T(\tau,t)+R_T(\tau,t),
\]
where
\[
 M_T(\tau,t)=(1-e^{-at})\zeta_T(\tau),\qquad
 R_T(\tau,t)=e^{-at}(\zeta_T(\tau)-\zeta_T(\tau+t)).
\]
The transformed cutoff is independent of the angle. Integrating over
the outer angular variable in \eqref{eq:cylinderenergy} gives
\[
 \E_{\alpha,\beta}[u_T]^{1/q}
 =\omega_{d-1}^{1/q}\|M_T+R_T\|_{q,p,\beta}.
\]
The first term satisfies $\|M_T\|_{q,p,\beta}=J_{a,\beta}^{1/p}\|\zeta_T\|_q$.
Since $q/p>1$, Minkowski's integral inequality in $L^{q/p}$ gives
\[
 \|R_T\|_{q,p,\beta}
 \le\left(\int_\R k_\beta(t)e^{-pat}
 \|\zeta_T-\zeta_T(\cdot+t)\|_q^p\dd t\right)^{1/p}.
\]
The fundamental theorem of calculus, followed by Minkowski in
$L^q(\R)$, yields
$\|\zeta_T-\zeta_T(\cdot+t)\|_q\le|t|\|\zeta_T'\|_q$.
Hence
\[
 \|R_T\|_{q,p,\beta}\le M_{a,\beta}\|\zeta_T'\|_q,
 \qquad M_{a,\beta}^p=\int_\R k_\beta(t)e^{-pat}|t|^p\dd t<\infty.
\]
Indeed, the integrand behaves as $|t|^{p-1-sp}$ near zero, as a
polynomial times $e^{-(\lambda_++pa)t}$ at positive infinity, and as a
polynomial times $e^{(\lambda_--pa)t}$ at negative infinity. All are
integrable precisely under the stated strict inequalities on $a$.
Now use the norm triangle inequality in both directions:
\[
 \left|\|M_T+R_T\|_{q,p,\beta}
       -J_{a,\beta}^{1/p}\|\zeta_T\|_q\right|
 \le M_{a,\beta}\|\zeta_T'\|_q.
\]
Since $\|\zeta_T\|_q=T^{1/q}\|\zeta\|_q$ and
$\|\zeta_T'\|_q=T^{1/q-1}\|\zeta'\|_q$, division by
$\|\zeta_T\|_q$ proves convergence to $J_{a,\beta}^{1/p}$.
Raising to the $q$th power gives the quotient limit, because
$\Pp_{\alpha_H+qa,\beta}[u_T]=\omega_{d-1}\|\zeta_T\|_q^q$.
For every finite $T$, the support is a compact annulus and the power
is smooth there, so the sequence belongs to the original test class.
The argument also covers $a=0$, with $M_T=0$.
\end{proof}

We now complete the proof of Theorem~\ref{thm:basic} by treating
the zero-constant weight and the asymptotic behavior in
\eqref{eq:vanishing}. At $a=0$, Lemma~\ref{lem:cutoff} gives
$H_{\alpha_H,\beta}^{p,q}=0$. For the asymptotic estimate,
\begin{equation}\label{eq:Jsmall}
 \frac{J_{a,\beta}}{|a|^p}\longrightarrow
 L_{d,s,p,\beta}:=\int_{\R^d}\frac{|\log|z||^p|z|^\beta}{|e-z|^{d+sp}}\dd z\in(0,\infty).
\end{equation}

To justify the limit, choose $0<\delta<\min\{\lambda_+/p,\lambda_-/p\}$.
For $0<|a|\le\delta$, the mean value theorem gives
\[
 \frac{|1-e^{-at}|^p}{|a|^p}\le |t|^p e^{p\delta|t|}.
\]
Multiplication by $k_\beta(t)$ gives an integrable majorant by
\eqref{eq:k0bounds}; pointwise the quotient tends to $|t|^p$.
This proves \eqref{eq:Jsmall}, including positivity and finiteness
of its limit. Finally $k_\pm(0)=1$, and differentiation under
their bounded $t$-integrals gives
\[
 k_+'(0)=-\frac d{2^d-1}\int_1^2t^{d-1}\log t\dd t<0,
 \qquad
 k_-'(0)=-\frac d{1-2^{-d}}\int_{1/2}^1t^{d-1}\log t\dd t>0.
\]
Thus $1-k_+(a)\asymp a$ for small $a>0$ and
$1-k_-(a)\asymp|a|$ for small $a<0$.
The annular lower bound and the cutoff upper bound consequently
give $H_{\alpha_H+qa,\beta}^{p,q}\asymp|a|^q$ on both sides.
Since $\alpha-\alpha_H=qa$, this is \eqref{eq:vanishing}.

\begin{remark}[Comparison with all measurable finite-energy functions]\label{rem:class}
One may also define a maximal measurable class by requiring both quantities in \eqref{eq:energy}--\eqref{eq:denominator} to be finite. Equality of its infimum with \eqref{eq:Hdef} requires a density theorem and is not assumed here. All sharpness sequences constructed here belong to the stated smooth test class.
\end{remark}

\subsection{Completion of the smooth test class in the difference norm}

Here we allow either exponent order $1<p,q<\infty$, with
$-d<\beta<sp$ and $-d<\alpha<sq+d(q/p-1)$, so the difference
seminorm is finite on $\test$. It is a norm on this test class:
the triangle inequality follows from the mixed $L^q(L^p)$ norm,
and vanishing forces a function to be constant almost everywhere;
a compactly supported such function is zero.
Define the completion of this test class in the difference norm by
\begin{equation}\label{eq:TL-core-completion}
 \mathcal X_{\alpha,\beta}^{s;p,q}
 :=\overline{\test}^{\,\E_{\alpha,\beta}^{1/q}}.
\end{equation}
This is the two-weight Triebel--Lizorkin difference space generated
by the specified smooth test functions. For nonzero Hardy constants it has the
following concrete realization.

\begin{corollary}[Extension with the same optimal constant]
\label{cor:TL-completion}
Assume $H:=H_{\alpha,\beta}^{p,q}>0$ on the test class in
\eqref{eq:Hdef}. Then $\mathcal X_{\alpha,\beta}^{s;p,q}$ embeds
continuously and injectively into
$L^q(\R^d,|x|^{\alpha+q\beta/p-sq}\dd x)$.
Its elements are represented by measurable functions, its norm
is the displayed difference norm, and
\begin{equation}\label{eq:TL-completed-Hardy}
 \E_{\alpha,\beta}[u]\ge
 H\int_{\R^d}|u(x)|^q|x|^{\alpha+q\beta/p-sq}\dd x
 \qquad(u\in\mathcal X_{\alpha,\beta}^{s;p,q}).
\end{equation}
The constant remains optimal. In particular, on every proved sharp
interval with positive homogeneous constant it remains
$J_{d,s,p}^{(\beta)}(a)^{q/p}$ on this space.
\end{corollary}
\begin{proof}
Write $w(x)=|x|^{\alpha+q\beta/p-sq}$. If $(u_j)$ is Cauchy
in the difference norm, the Hardy inequality on the test class applied to
$u_j-u_\ell$ gives
\[
 \|u_j-u_\ell\|_{L^q(w)}
 \le H^{-1/q}\E_{\alpha,\beta}[u_j-u_\ell]^{1/q}.
\]
Thus $u_j\to u$ in $L^q(w)$ for a measurable $u$; a subsequence
converges almost everywhere since $w$ is positive and finite almost
everywhere. For each fixed $j$, apply Fatou first to the inner
integral and then to the outer integral along that subsequence:
\begin{equation}\label{eq:TL-completion-Fatou}
 \E_{\alpha,\beta}[u_j-u]
 \le\liminf_{\ell\to\infty}
          \E_{\alpha,\beta}[u_j-u_\ell].
\end{equation}
The Cauchy property makes the right side tend to zero as $j\to\infty$.
For one fixed $j$ the right side is finite, so the triangle inequality
also gives $\E_{\alpha,\beta}[u]<\infty$. Consequently the abstract
completion is represented by $u$ in the same difference norm.
The norm triangle inequality gives
$\E_{\alpha,\beta}[u_j]^{1/q}\to\E_{\alpha,\beta}[u]^{1/q}$.
If the weighted $L^q$ limit is zero, the same Fatou argument shows
that the completion element is zero, proving injectivity.
Passing to the limit in the Hardy inequality on the test class proves
\eqref{eq:TL-completed-Hardy}.

Since the smooth test class is contained in the completion, an improvement of
$H$ on the latter would improve the constant in \eqref{eq:Hdef}. This proves
optimality without imposing a density theorem for any larger
measurable class. The same argument, with the corresponding distance
weights and kernels, applies to the flat-subspace and regional
half-space completions of the smooth test classes whenever their Hardy
constant is positive. At a zero-constant weight no positive Hardy
embedding is asserted.
\end{proof}

\section{Two-point supporting lines and quantitative remainders}\label{sec:scalar}

The statements in this section are scalar inequalities: they do not
involve a radial measure or an additional power weight. In their
application, all dependence on $\beta$ enters through
$\vartheta=e^{-at}$ and $b=e^{-(d+2\beta-sp-pa)t}$.
Put $m=q-p>0$. For $0<\vartheta<1$ and $b\ge0$, define
\begin{align}
 F_{\vartheta,b}(U,V)&=|U-\vartheta V|^p(U^m+bV^m),\label{eq:Ftwo}\\
 A_{\vartheta,b}&=\frac{(1-\vartheta)^{p-1}}q(q-m\vartheta+pb),\label{eq:Atwo}\\
 B_{\vartheta,b}&=\frac{(1-\vartheta)^{p-1}}q\bigl(b(m-q\vartheta)-p\vartheta\bigr),\label{eq:Btwo}\\
 G_{\vartheta,b}(U,V)&=F_{\vartheta,b}(U,V)-A_{\vartheta,b}U^q-B_{\vartheta,b}V^q.
 \label{eq:Gtwo}
\end{align}
The coefficients $A_{\vartheta,b}$ and $B_{\vartheta,b}$ are chosen
so that $F_{\vartheta,b}$ and $A_{\vartheta,b}U^q+B_{\vartheta,b}V^q$
agree, together with their first derivatives, on the diagonal $U=V$.
For $V>0$, the inequality $G_{\vartheta,b}\ge0$ says that the graph
of $x\mapsto F_{\vartheta,b}(x^{1/q},1)$ lies above its tangent
line at $x=1$; this is the meaning of \emph{support} below. The
remainder is strict when $G_{\vartheta,b}(U,V)>0$ for $U\ne V$. Also,
\begin{equation}\label{eq:supportsum}
 A_{\vartheta,b}+B_{\vartheta,b}=(1-\vartheta)^p(1+b).
\end{equation}
They need not both be nonnegative.

\begin{lemma}[Strict scalar support]\label{lem:scalar-support}
For $n_{p,q}$ in \eqref{eq:n-pq},
\begin{equation}\label{eq:scalar-positive}
 G_{\vartheta,b}(U,V)\ge0\qquad
 \bigl(U,V\ge0,\ 0\le b\le\vartheta^{n_{p,q}}\bigr).
\end{equation}
For fixed $\vartheta\in(0,1)$, the inequality is strict if $U\ne V$. Moreover,
\begin{equation}\label{eq:cdef}
 c(\vartheta,b):=\inf_{\substack{U,V\ge0\\U\ne V}}
 \frac{G_{\vartheta,b}(U,V)}{(U^{q/2}-V^{q/2})^2}>0
 \qquad\text{when }b=0\text{ or }b=\vartheta^{n_{p,q}}.
\end{equation}
Consequently $c(\vartheta,b)>0$ throughout $0\le b\le\vartheta^{n_{p,q}}$.
\end{lemma}

\begin{proof}
\emph{Endpoint interpolation and reduction.}
Fix $\vartheta\in(0,1)$ and put $b_*:=\vartheta^{n_{p,q}}>0$.
We first explain why it suffices to prove the assertions at the two
endpoints $b=0$ and $b=b_*$. Expanding
\eqref{eq:Ftwo}--\eqref{eq:Gtwo} with respect to $b$ gives
\[
 G_{\vartheta,b}(U,V)
 =G_{\vartheta,0}(U,V)+bH_{\vartheta}(U,V),
\]
where
\[
 H_{\vartheta}(U,V)
 =|U-\vartheta V|^pV^m
 -\frac{(1-\vartheta)^{p-1}}q
   \bigl(pU^q+(m-q\vartheta)V^q\bigr)
\]
does not depend on $b$. In particular,
$G_{\vartheta,b_*}-G_{\vartheta,0}=b_*H_{\vartheta}$.
For $0\le b\le b_*$, set $\lambda=b/b_*\in[0,1]$. Substituting
this identity yields the endpoint interpolation formula
\[
\begin{aligned}
 G_{\vartheta,b}(U,V)
 &=G_{\vartheta,0}(U,V)
   +\frac{b}{b_*}\bigl(G_{\vartheta,b_*}(U,V)
                      -G_{\vartheta,0}(U,V)\bigr)\\
 &=\left(1-\frac{b}{b_*}\right)G_{\vartheta,0}(U,V)
   +\frac{b}{b_*}G_{\vartheta,b_*}(U,V).
\end{aligned}
\]
Both coefficients are nonnegative and their sum is one.
For $U\ne V$, divide by $(U^{q/2}-V^{q/2})^2$ and take the
infimum over $U,V\ge0$ to obtain
\[
\begin{aligned}
 c(\vartheta,b)
 &\ge(1-\lambda)c(\vartheta,0)+\lambda c(\vartheta,b_*)\\
 &\ge\min\{c(\vartheta,0),c(\vartheta,b_*)\}.
\end{aligned}
\]
Thus positivity of the two endpoint constants implies a positive
quadratic remainder, and hence strict support for $U\ne V$, throughout
$0\le b\le b_*$. On the diagonal,
$G_{\vartheta,b}(U,U)=0$ by \eqref{eq:supportsum}.
It therefore remains only to prove $c(\vartheta,0)>0$ and
$c(\vartheta,b_*)>0$. We record the common formulas and two support
criteria, then treat $b=0$ and the three exponent ranges for $b=b_*$.

First suppose $V>0$. Positive $q$-homogeneity permits division
by $V^q$, so set $V=1$ and $U=z$. The case $V=0$ is recovered from
the endpoint limit below. Put $x=z^q$ and define
\[
 f(x)=|x^{1/q}-\vartheta|^p(x^{m/q}+b).
\]
Write $\ell(x)=f(1)+f'(1)(x-1)$ for its tangent line at $1$.
Direct differentiation at $x=1$ gives $f'(1)=A_{\vartheta,b}$
and $f(1)-f'(1)=B_{\vartheta,b}$. Hence
$G_{\vartheta,b}(z,1)=f(z^q)-f(1)-f'(1)(z^q-1)$.
For the second derivative, put
$H(z)=(z-\vartheta)^p(z^m+b)$ on $z>\vartheta$. Then
\[
 f''(z^q)=\frac{zH''(z)-(q-1)H'(z)}{q^2z^{2q-1}},
\]
where
\begin{align*}
 H'(z)&=p(z-\vartheta)^{p-1}(z^m+b)
       +mz^{m-1}(z-\vartheta)^p,\\
 H''(z)&=p(p-1)(z-\vartheta)^{p-2}(z^m+b)\\
 &\quad+2pmz^{m-1}(z-\vartheta)^{p-1}
       +m(m-1)z^{m-2}(z-\vartheta)^p.
\end{align*}
Substitution and $p+m=q$ give
\begin{equation}\label{eq:fsecond}
 f''(z^q)=\frac{p(z-\vartheta)^{p-2}}{q^2z^{2q-1}}
 E_{\vartheta,b}(z),
\end{equation}
where
\begin{equation}\label{eq:Escalar}
 E_{\vartheta,b}(z)=
 \vartheta z^{m-1}\bigl((q-1)z-m\vartheta\bigr)
 +b\bigl((q-1)\vartheta-mz\bigr).
\end{equation}

\smallskip\noindent
\emph{Two common support criteria.}
By \eqref{eq:Atwo}, the tangent slope is always positive:
$q-m\vartheta+pb>q-m=p>0$.

\emph{(i) Strict convexity above $\vartheta^q$.}
If $f''(x)>0$ for $x>\vartheta^q$, Taylor's integral formula gives
\[
 f(x)-\ell(x)=
 \begin{cases}
  \displaystyle\int_x^1(t-x)f''(t)\dd t>0,
       &\vartheta^q\le x<1,\\[2mm]
  \displaystyle\int_1^x(x-t)f''(t)\dd t>0,&x>1.
 \end{cases}
\]
At the left endpoint the formula follows by a limit; the only
possible singularity has order $(x-\vartheta^q)^{p-2}$ and is
integrable because $p>1$.
Since $f(\vartheta^q)=0$, the first formula gives
$\ell(\vartheta^q)<0$. Its positive slope then gives
$\ell(x)<0\le f(x)$ for $0\le x\le\vartheta^q$.
Thus this criterion proves strict support on the whole half-line
and also $f''(1)>0$.

\emph{(ii) A criterion for $0<m<1$.}
For later use, we prove that
\[
 0<m<1,\qquad b>0,\qquad A_{\vartheta,b}=f'(1)<1
\]
imply strict tangent support on $[0,\infty)$ and $f''(1)>0$.
Throughout this argument assume all three conditions. Put $k=q-1$,
so $k-m=p-1>0$, and set $z=\vartheta y$.
The second-derivative factor is
\[
 E_{\vartheta,b}(\vartheta y)
 =\vartheta\bigl(\vartheta^m y^{m-1}(ky-m)-b(my-k)\bigr).
\]
It is positive for $1<y\le k/m$: the first term in parentheses
is positive and the second is nonnegative. For $y>k/m$, write
\[
 E_{\vartheta,b}(\vartheta y)
 =\vartheta^{m+1}(my-k)
   \left(R(y)-\frac b{\vartheta^m}\right),
 \qquad
 R(y)=\frac{y^{m-1}(ky-m)}{my-k}.
\]
All factors outside the parentheses are positive. Moreover,
\[
\begin{aligned}
 \frac{R'(y)}{R(y)}
 &=\frac{m-1}{y}+\frac{k}{ky-m}-\frac m{my-k}\\
 &=\frac{m-1}{y}
   +\frac{m^2-k^2}{(ky-m)(my-k)}<0.
\end{aligned}
\]
Here $m-1<0$, $k>m$, and both denominator factors are positive.
The endpoint limits are
\[
 \lim_{y\downarrow k/m}R(y)=+\infty,
 \qquad
 R(y)\sim\frac{k}{m}y^{m-1}\longrightarrow0
 \quad(y\to\infty).
\]
Thus there is a unique $y_*>k/m$ such that
$R(y_*)=b/\vartheta^m$. With $x_*=(\vartheta y_*)^q$,
equation \eqref{eq:fsecond} gives
\[
 f''(x)>0\quad(\vartheta^q<x<x_*),\qquad
 f''(x)<0\quad(x>x_*).
\]
The first derivative can be written explicitly as
\[
 f'(z^q)
 =\left(1-\frac\vartheta z\right)^{p-1}
   \left(1-\frac{m\vartheta}{qz}+\frac{pb}{qz^m}\right)
 \longrightarrow1\quad(z\to\infty).
\]
Hence $f'$ increases strictly up to $x_*$ and then decreases
strictly to $1$. In particular, $f'(x)>1$ for $x\ge x_*$:
a value at most $1$ on this decreasing branch would force a later
value strictly below $1$, and hence a limiting value below $1$.
Since $A_{\vartheta,b}=f'(1)<1$, necessarily $1<x_*$, and therefore
\[
 f'(x)<A_{\vartheta,b}\quad(\vartheta^q<x<1),\qquad
 f'(x)>A_{\vartheta,b}\quad(x>1).
\]
For the second inequality, use the increasing branch for $1<x<x_*$
and $f'(x)>1>A_{\vartheta,b}$ for $x\ge x_*$.
These derivative inequalities give strict tangent support directly:
\[
 f(x)-\ell(x)=
 \begin{cases}
  \displaystyle\int_x^1(A_{\vartheta,b}-f'(t))\dd t>0,
    &\vartheta^q\le x<1,\\[2mm]
  \displaystyle\int_1^x(f'(t)-A_{\vartheta,b})\dd t>0,
    &x>1.
 \end{cases}
\]
The first formula remains valid at $x=\vartheta^q$, since
$f'$ is continuous from the right there and $p>1$.
As $f(\vartheta^q)=0$, it follows that $\ell(\vartheta^q)<0$.
Also $A_{\vartheta,b}>0$, because
$q-m\vartheta+pb>q-m=p>0$. Thus
\[
 \ell(x)=\ell(\vartheta^q)+A_{\vartheta,b}(x-\vartheta^q)<0
 \qquad(0\le x\le\vartheta^q),
\]
whereas $f(x)\ge0$. This proves strict support on the whole
half-line whenever $0<m<1$, $b>0$, and $A_{\vartheta,b}<1$.
It also shows $f''(1)>0$, since $1<x_*$.

\smallskip\noindent
\emph{A common bound for the tangent slope.}
The function
\[
 g(\vartheta)=(1-\vartheta)^{p-1}(1+(p-1)\vartheta)
\]
satisfies $g(0)=1$ and
\[
 g'(\vartheta)=-p(p-1)\vartheta(1-\vartheta)^{p-2}<0.
\]
Hence $g(\vartheta)<1$ for $0<\vartheta<1$.
This bound will be used in Cases~2 and~3.

\smallskip\noindent
\emph{The endpoint $b=0$, for every $1<p<q$.}
For $z>\vartheta$, equation \eqref{eq:Escalar} gives
\[
 E_{\vartheta,0}(z)
 =\vartheta z^{m-1}((q-1)z-m\vartheta)>0,
\]
because $(q-1)z-m\vartheta>(q-1-m)\vartheta=(p-1)\vartheta>0$.
Equations \eqref{eq:fsecond} and criterion~(i) prove strict support
on the whole half-line and $f''(1)>0$. This endpoint is now
settled independently of the exponent cases below.

\smallskip\noindent
\emph{Case 1: $q>2p$, with $b=b_* =\vartheta^m$.}
Here $m=q-p>p>1$. Put $z=\vartheta y$, $y>1$, and $k=q-1$.
Then \eqref{eq:Escalar} yields
\[
\begin{aligned}
 \vartheta^{-m-1}E_{\vartheta,\vartheta^m}(z)
 &=y^{m-1}(ky-m)+k-my\\
 &\ge ky-m+k-my=(p-1)(y+1)>0.
\end{aligned}
\]
The inequality uses $y^{m-1}\ge1$, $ky-m>0$, and $k-m=p-1$.
Criterion~(i) therefore proves strict support on the whole
half-line and $f''(1)>0$, completing Case~1.

\smallskip\noindent
\emph{Case 2: $2\le q\le2p$, with $b=b_* =\vartheta$.}
Here $0<m\le p$. Write $c=(p-m)/q\ge0$.

\emph{(a) The interval $0\le z\le1$.}
The scalar inequality of Frank and Seiringer~\cite[Lemma~2.6, (2.15)]{FS} gives
\begin{equation}\label{eq:picone}
 |z-\vartheta|^p\ge(1-\vartheta)^{p-1}(z^p-\vartheta).
\end{equation}
Since $p>1$ and $0<\vartheta<1$, equality holds if and only if $z=1$. For completeness, we recall a short proof using the triangle inequality and convexity:
\[
\begin{aligned}
 z^p&\le(|z-\vartheta|+\vartheta)^p\\
 &\le(1-\vartheta)\left(\frac{|z-\vartheta|}{1-\vartheta}\right)^p+\vartheta\\
 &=(1-\vartheta)^{1-p}|z-\vartheta|^p+\vartheta.
\end{aligned}
\]
The second inequality uses the convexity of $t\mapsto t^p$ with weights
$1-\vartheta$ and $\vartheta$ and entries
$|z-\vartheta|/(1-\vartheta)$ and $1$.
Equality in both steps requires $z=1$.
Also,
\begin{equation}\label{eq:power-compare}
 z^p-z^m\ge\frac{p-m}{p+m}(z^{p+m}-1),\qquad 0\le z\le1.
\end{equation}
If $p=m$, both sides of \eqref{eq:power-compare} vanish identically.
Suppose $p>m$ and first take $0<z<1$. Put $r=-\log z>0$.
Factoring the same exponential from numerator and denominator gives
\[
\begin{aligned}
 z^m-z^p
 &=2e^{-(p+m)r/2}\sinh((p-m)r/2),\\
 1-z^{p+m}
 &=2e^{-(p+m)r/2}\sinh((p+m)r/2).
\end{aligned}
\]
To compare these quantities, define $\phi(t)=\sinh t/t$ for $t>0$.
Then
\[
 \phi'(t)=\frac{t\cosh t-\sinh t}{t^2}>0,
\]
because $\psi(t)=t\cosh t-\sinh t$ satisfies $\psi(0)=0$ and
$\psi'(t)=t\sinh t>0$. Since $0<p-m<p+m$, it follows that
\[
\begin{aligned}
 \frac{z^m-z^p}{1-z^{p+m}}
 &=\frac{p-m}{p+m}
   \frac{\phi((p-m)r/2)}{\phi((p+m)r/2)}\\
 &\le\frac{p-m}{p+m}.
\end{aligned}
\]
Multiplication by $1-z^{p+m}>0$, followed by a change of sign,
is precisely \eqref{eq:power-compare}. At $z=0$, that inequality
reads $0\ge-(p-m)/(p+m)$; at $z=1$, both sides are zero.
Thus it holds on the entire closed interval.

We now combine the two scalar inequalities. Since $z^m+\vartheta>0$,
multiplying \eqref{eq:picone} by this factor gives
\[
\begin{aligned}
 |z-\vartheta|^p(z^m+\vartheta)
 &\ge(1-\vartheta)^{p-1}(z^p-\vartheta)(z^m+\vartheta)\\
 &=(1-\vartheta)^{p-1}
   \bigl(z^q+\vartheta(z^p-z^m)-\vartheta^2\bigr)\\
 &\ge(1-\vartheta)^{p-1}
   \bigl((1+c\vartheta)z^q-\vartheta(c+\vartheta)\bigr),
\end{aligned}
\]
where the last step uses $p+m=q$ and $c=(p-m)/q$.
For $b=\vartheta$, the coefficients in \eqref{eq:Atwo}--\eqref{eq:Btwo}
satisfy
\[
\begin{aligned}
 A_{\vartheta,\vartheta}
 &=(1-\vartheta)^{p-1}(1+c\vartheta),\\
 B_{\vartheta,\vartheta}
 &=-(1-\vartheta)^{p-1}\vartheta(c+\vartheta).
\end{aligned}
\]
The preceding estimate is therefore $G_{\vartheta,\vartheta}(z,1)\ge0$
for $0\le z\le1$. It is strict for $z\ne1$, since
\eqref{eq:picone} is then strict and its multiplier is positive.

\emph{(b) The interval $z>1$.}
First note that the slope satisfies
\[
 p-1-c=\frac{p(q-2)}q\ge0,\qquad
 0<A_{\vartheta,\vartheta}
 =(1-\vartheta)^{p-1}(1+c\vartheta)\le g(\vartheta)<1,
\]
where $g$ is the common slope bound defined above.
We now distinguish the two possible ranges of $m$.

If $m\ge1$, put $w=z^{m-1}\ge1$ and $k=q-1$.
Then $k-m=p-1>0$, and \eqref{eq:Escalar} becomes
\[
 \frac{E_{\vartheta,\vartheta}(z)}\vartheta
 =z(kw-m)+\vartheta(k-mw).
\]
If $k-mw\ge0$, the right side is at least
$(k-m)z=(p-1)z>0$. If $k-mw<0$, then
$\vartheta(k-mw)\ge k-mw$, so
\[
\begin{aligned}
 \frac{E_{\vartheta,\vartheta}(z)}\vartheta
 &\ge(kz-m)w+k-mz\\
 &\ge kz-m+k-mz=(p-1)(z+1)>0.
\end{aligned}
\]
In the second line we used $kz-m\ge k-m>0$ and $w\ge1$.
It follows from \eqref{eq:fsecond} that $f''(x)>0$ for $x\ge1$.
Writing $\ell(x)=f(1)+f'(1)(x-1)$, we obtain, for $x>1$,
\[
 f(x)-\ell(x)=\int_1^x(x-t)f''(t)\dd t>0.
\]
This proves the desired support on $z>1$ in the subcase $m\ge1$.
If $0<m<1$, apply criterion~(ii) with $b=\vartheta>0$ and
the slope bound $A_{\vartheta,\vartheta}<1$ just established.
It gives strict support on the whole half-line, and in particular
on $z>1$, together with $f''(1)>0$.
Combining (a) and (b) completes Case~2.

\smallskip\noindent
\emph{Case 3: $1<p<q<2$, with $b=b_* =\vartheta^{n_{p,q}}$.}
In this case $0<m=q-p<1$. We will apply criterion~(ii), so
it remains to verify $A_{\vartheta,b}<1$. Put
\[
 \vartheta_0=1-2^{-1/(p-1)},\qquad
 n=n_{p,q},\qquad b=\vartheta^n.
\]
The definition \eqref{eq:n-pq} gives $n>1$ and
$\vartheta_0^{n-1}=q-1$. If $\vartheta\le\vartheta_0$, then
\[
 b=\vartheta\vartheta^{n-1}
 \le\vartheta\vartheta_0^{n-1}=(q-1)\vartheta.
\]
Substitution into \eqref{eq:Atwo}, using
$p(q-1)-m=q(p-1)$, gives
\[
 A_{\vartheta,b}
 \le\frac{(1-\vartheta)^{p-1}}q
      \bigl(q+(p(q-1)-m)\vartheta\bigr)
 =g(\vartheta)<1.
\]
If $\vartheta\ge\vartheta_0$, then
$(1-\vartheta)^{p-1}\le1/2$ and $b=\vartheta^n\le1$, whence
\[
 A_{\vartheta,b}
 \le\frac{q-m\vartheta+pb}{2q}
 \le\frac{q+p}{2q}<1.
\]
Thus $A_{\vartheta,b}<1$ for every $0<\vartheta<1$.
Since also $0<m<1$ and $b>0$, criterion~(ii) proves strict
support on the whole half-line and $f''(1)>0$. This completes
Case~3.

\smallskip\noindent
\emph{Positive quadratic remainder at the two endpoints.}
We next prove the quantitative assertion \eqref{eq:cdef} at
$b=0$ and $b=\vartheta^{n_{p,q}}$. For $z\ne1$, define
\[
 Q_{\vartheta,b}(z)
 =\frac{G_{\vartheta,b}(z,1)}{(z^{q/2}-1)^2}.
\]
Since $f$ is smooth near $x=1$ and is tangent to $\ell$ there,
Taylor's theorem gives
\[
\begin{aligned}
 G_{\vartheta,b}(z,1)
 &=\tfrac12f''(1)(z^q-1)^2+o((z^q-1)^2),\\
 (z^{q/2}-1)^2&=\tfrac{q^2}{4}(z-1)^2+o((z-1)^2).
\end{aligned}
\]
Using $z^q-1=q(z-1)+o(z-1)$ and \eqref{eq:fsecond}, the
continuous extension at $z=1$ is
\begin{equation}\label{eq:diagonalratio}
 Q_{\vartheta,b}(1)=2f''(1)
 =\frac{2p}{q^2}(1-\vartheta)^{p-2}E_{\vartheta,b}(1)>0.
\end{equation}
The positivity of $f''(1)$ follows from strict positivity of
\eqref{eq:Escalar} in the convex cases above, and from $1<x_*$
in the case $0<m<1$. In particular, when $b=\vartheta$,
direct substitution gives
\[
 E_{\vartheta,\vartheta}(1)
 =\vartheta\bigl((q-1)-m\vartheta+(q-1)\vartheta-m\bigr)
 =(p-1)\vartheta(1+\vartheta)>0.
\]
At the two ends of the half-line,
\[
 Q_{\vartheta,b}(0)=\vartheta^pb-B_{\vartheta,b}>0,
 \qquad
 \lim_{z\to\infty}Q_{\vartheta,b}(z)=1-A_{\vartheta,b}.
\]
The first inequality is strict support at $z=0$. For the second
limit, divide numerator and denominator by $z^q$ and use
$F_{\vartheta,b}(z,1)/z^q\to1$.
The limiting value is positive in every case considered. We have
already proved this for $b=\vartheta$, $2\le q\le2p$, and for
$b=\vartheta^{n_{p,q}}$, $q<2$. In the other cases,
\[
\begin{aligned}
 0<A_{\vartheta,0}
 &=(1-\vartheta)^{p-1}\left(1-\frac mq\vartheta\right)<1,\\
 A_{\vartheta,\vartheta^m}
 &\le A_{\vartheta,\vartheta}
  =(1-\vartheta)^{p-1}
    \left(1+\frac{p-m}{q}\vartheta\right)<1
  \qquad(q>2p).
\end{aligned}
\]
In the last line, $m>p>1$ implies $\vartheta^m\le\vartheta$
and $p-m<0$.

\smallskip\noindent
\emph{Compactness and a uniform positive lower bound.}
Keep $\vartheta$ and one of the two endpoint values of $b$ fixed.
Introduce the change of variables
\[
 \xi=\frac{z}{1+z},\qquad
 z=\frac{\xi}{1-\xi}\quad(0\le\xi<1).
\]
The first map is continuous and strictly increasing on $[0,\infty)$,
since its derivative is $(1+z)^{-2}>0$, and its inverse is the second
map. The values $z=0$ and $z=1$ correspond to $\xi=0$ and
$\xi=1/2$, respectively, while $z\to\infty$ corresponds to
$\xi\to1$. Thus adjoining the endpoint $z=\infty$ replaces the
extended half-line by the compact interval $[0,1]$.

Using the value of $Q_{\vartheta,b}(1)$ defined in
\eqref{eq:diagonalratio}, set
\[
 \widetilde Q_{\vartheta,b}(\xi)
 =\begin{cases}
 Q_{\vartheta,b}\!\left(\dfrac{\xi}{1-\xi}\right),
       &0\le\xi<1,\\[2mm]
 1-A_{\vartheta,b},&\xi=1.
 \end{cases}
\]
Away from $\xi=1/2,1$, this function is continuous because
$G_{\vartheta,b}(z,1)$ is continuous and the denominator
$(z^{q/2}-1)^2$ is nonzero. At $\xi=1/2$, continuity is exactly
the Taylor limit in \eqref{eq:diagonalratio}; at $\xi=1$, it is
the limit $Q_{\vartheta,b}(z)\to1-A_{\vartheta,b}$ proved above.
Hence $\widetilde Q_{\vartheta,b}$ is continuous on all of $[0,1]$.
Strict support gives positivity whenever $0\le\xi<1$ and
$\xi\ne1/2$. The two remaining values are also positive:
\[
 \widetilde Q_{\vartheta,b}(1/2)=2f''(1)>0,
 \qquad
 \widetilde Q_{\vartheta,b}(1)=1-A_{\vartheta,b}>0.
\]
By the extreme value theorem, there is $\xi_*\in[0,1]$ such that
\[
 \kappa_{\vartheta,b}
 :=\min_{\xi\in[0,1]}\widetilde Q_{\vartheta,b}(\xi)
 =\widetilde Q_{\vartheta,b}(\xi_*)>0.
\]
The strict inequality follows because the minimum is attained and
the function is positive at every point, including both endpoints.

To transfer this bound to all pairs, let $V>0$ and $U\ne V$.
With $z=U/V$, positive $q$-homogeneity gives
\[
 \frac{G_{\vartheta,b}(U,V)}{(U^{q/2}-V^{q/2})^2}
 =\frac{V^qG_{\vartheta,b}(z,1)}
        {V^q(z^{q/2}-1)^2}
 =\widetilde Q_{\vartheta,b}\!\left(\frac{z}{1+z}\right)
 \ge\kappa_{\vartheta,b}.
\]
On the axis $V=0$, $U>0$, the same bound holds since
\[
 \frac{G_{\vartheta,b}(U,0)}{U^q}
 =1-A_{\vartheta,b}
 =\widetilde Q_{\vartheta,b}(1)\ge\kappa_{\vartheta,b}.
\]
Conversely, the pairs $(U,V)=(z,1)$, $z\ge0$, $z\ne1$,
represent the dense subset $[0,1)\setminus\{1/2\}$ of $[0,1]$.
Continuity therefore implies
\[
 c(\vartheta,b)
 =\inf_{\xi\in[0,1]\setminus\{1/2\}}
      \widetilde Q_{\vartheta,b}(\xi)
 =\min_{\xi\in[0,1]}\widetilde Q_{\vartheta,b}(\xi)
 =\kappa_{\vartheta,b}>0.
\]
On the diagonal, $G_{\vartheta,b}(U,U)=0$ by
\eqref{eq:supportsum}, including the zero pair. Consequently
\[
 G_{\vartheta,b}(U,V)
 \ge\kappa_{\vartheta,b}(U^{q/2}-V^{q/2})^2
 \qquad\text{for all }U,V\ge0.
\]
This lower bound is uniform in $U,V$, with $\vartheta,b$ fixed.
\end{proof}

The next lemma addresses the singular region of small logarithmic displacements. Pointwise strictness at a fixed displacement would not suffice there.

\begin{lemma}[Uniform remainder near $\vartheta=1$]\label{lem:short}
Let $r_p=(p-2)_+$. There are $h_0,c_0>0$, depending only on $p,q$, such that
\begin{equation}\label{eq:short-bound}
 G_{\vartheta,b}(U,V)\ge
 c_0(1-\vartheta)^{r_p}(U^{q/2}-V^{q/2})^2
\end{equation}
whenever $1-h_0<\vartheta<1$, $1/2\le b\le2$, and $U,V\ge0$.
\end{lemma}
\begin{proof}
First consider a neighborhood of the diagonal with $V>0$.
By $q$-homogeneity we may set $V=1$, $U=1+z$.
Write $h=1-\vartheta$ and $m=q-p$.
For $|z|\le1/2$, Taylor's theorem on $[1/2,3/2]$ gives
\[
 (1+z)^m=1+mz+R_m(z),\qquad
 (1+z)^q=1+qz+R_q(z),\qquad
 |R_m(z)|+|R_q(z)|\le C z^2.
\]
Equations \eqref{eq:Atwo} and \eqref{eq:supportsum} imply
\[
 qA_{\vartheta,b}=p(1+b)h^{p-1}+mh^p,
 \qquad A_{\vartheta,b}+B_{\vartheta,b}=(1+b)h^p.
\]
Substituting these identities and the Taylor expansions into $G$ gives
\begin{align*}
 G_{\vartheta,b}(1+z,1)
 ={}&(1+b)\bigl(|h+z|^p-h^p-ph^{p-1}z\bigr)\\
 &+mz\bigl(|h+z|^p-h^p\bigr)
   +R_m(z)|h+z|^p-A_{\vartheta,b}R_q(z).
\end{align*}
This decomposition is exact. The mean value theorem for
$t\mapsto|t|^p$ yields
\[
 \bigl||h+z|^p-h^p\bigr|
 \le p|z|(h+|z|)^{p-1}.
\]
Also $|A_{\vartheta,b}|\le C h^{p-1}$ when $0<h\le1$ and
$b\in[1/2,2]$. If $h+|z|\le1$, the last three terms of the
decomposition are therefore bounded respectively by
\[
 C z^2(h+|z|)^{p-1},\qquad
 C z^2(h+|z|)^p,\qquad
 C z^2h^{p-1}.
\]
Each is at most $Cz^2(h+|z|)^{p-1}$. Consequently
\begin{equation}\label{eq:shortexpand}
 G_{\vartheta,b}(1+z,1)
 =(1+b)\bigl(|h+z|^p-h^p-ph^{p-1}z\bigr)
 +O_{p,q}\bigl(z^2(h+|z|)^{p-1}\bigr),
\end{equation}
uniformly for $b\in[1/2,2]$.

The expression in parentheses is the remainder after subtracting
the value and tangent line of $t\mapsto|t|^p$ at $t=h$.
This Bregman difference has the integral representation
\[
 |h+z|^p-h^p-ph^{p-1}z
 =p(p-1)z^2\int_0^1(1-r)|h+rz|^{p-2}\dd r.
\]
For $1<p<2$, any zero of $h+rz$ gives an integrable singularity,
since $p-2>-1$; thus the identity follows from the absolute
continuity of the first derivative. In this range,
$|h+rz|\le h+|z|$ gives the direct lower bound
\[
 \int_0^1(1-r)|h+rz|^{p-2}\dd r
 \ge\tfrac12(h+|z|)^{p-2}.
\]
For $p=2$ the Bregman difference is exactly $z^2$.
For $p>2$ and $z\ne0$, put
$a=h/(h+|z|)$ and $v=z/(h+|z|)$. Then
$a\ge0$ and $a+|v|=1$. The function
\[
 (a,v)\longmapsto\int_0^1(1-r)|a+rv|^{p-2}\dd r
\]
is continuous on this compact set and is positive everywhere:
an affine function $a+rv$ could vanish identically only if
$a=v=0$, which is excluded. Its minimum is therefore positive.
These three cases prove
\begin{equation}\label{eq:bregman}
 |h+z|^p-h^p-ph^{p-1}z
 \ge c_p z^2(h+|z|)^{p-2}.
\end{equation}
The assertion is also immediate at $z=0$.

Combining \eqref{eq:shortexpand}--\eqref{eq:bregman} gives
\[
 G_{\vartheta,b}(1+z,1)
 \ge z^2(h+|z|)^{p-2}
       \bigl(\tfrac32c_p-C(h+|z|)\bigr),
\]
because $1+b\ge3/2$. Choose $0<\rho\le1/2$, depending only
on $p,q$, such that $C\rho\le3c_p/4$.
For $0<h\le\rho/2$ and $|z|\le\rho/2$, this yields
\[
 G_{\vartheta,b}(1+z,1)
 \ge\tfrac34c_p z^2(h+|z|)^{p-2}.
\]
On this neighborhood the derivative of $u\mapsto u^{q/2}$ is
bounded above and below by positive constants, so
$|(1+z)^{q/2}-1|\asymp_{q}|z|$.
If $p\ge2$, then $(h+|z|)^{p-2}\ge h^{p-2}$;
if $p<2$, then $(h+|z|)^{p-2}\ge1$, since $h+|z|\le1$.
This proves \eqref{eq:short-bound} near the diagonal.

For the remaining pairs, use homogeneity to normalize
$U^q+V^q=1$. Remove the relative neighborhood
$V>0$, $|U/V-1|<\rho/2$ of its diagonal point.
The complement is compact, includes both axes, and has
$(U^{q/2}-V^{q/2})^2>0$.
At $\vartheta=1$, the coefficients $A_{\vartheta,b}$ and
$B_{\vartheta,b}$ vanish, and hence
\[
 G_{1,b}(U,V)=|U-V|^p(U^m+bV^m)>0
\]
on this complement, uniformly for $b\in[1/2,2]$.
Its ratio to $(U^{q/2}-V^{q/2})^2$ has a positive minimum
$\gamma$ on the compact product set. Joint continuity, including
at $\vartheta=1$, makes that ratio at least $\gamma/2$ when
$1-h_0<\vartheta<1$, for some $h_0\le\rho/2$ depending only
on $p,q$. Since $h^{(p-2)_+}\le1$, this also gives
\eqref{eq:short-bound} on the complement. Taking the smaller
constant from the two regions completes the proof; the zero
pair follows directly from homogeneity.
\end{proof}

\begin{lemma}[Continuity and a bound for the negative part]\label{lem:continuity}
The function $c(\vartheta,b)$ defined by \eqref{eq:cdef} is finite and continuous on $(0,1)\times[0,\infty)$. For each $\vartheta_*<1$ there is $C=C(p,q,\vartheta_*)$ such that
\begin{equation}\label{eq:cnegative}
 [c(\vartheta,b)]_-\le C(1+b)
 \qquad(0\le\vartheta\le\vartheta_*,\ b\ge0).
\end{equation}
\end{lemma}
\begin{proof}
By homogeneity it is enough to minimize over
$\mathcal K=\{(U,V)\ge0:U^q+V^q=1\}$.
This is compact. The denominator vanishes at its single diagonal
point. Near that point $U,V$ are positive, and
$U-\vartheta V$ stays nonzero when $\vartheta$ belongs to a
compact subinterval of $(0,1)$. Taylor's theorem and tangency
give the continuous extension \eqref{eq:diagonalratio}.
At the axes the denominator is nonzero, so no extension is needed.
The resulting ratio is jointly continuous on $\mathcal K$ and
on every compact parameter set. Uniform continuity implies
\[
 |\min_{\mathcal K}R_{\vartheta,b}
   -\min_{\mathcal K}R_{\vartheta',b'}|
 \le\sup_{\mathcal K}|R_{\vartheta,b}-R_{\vartheta',b'}|\longrightarrow0,
\]
which proves finiteness and continuity of $c$.

The same extension is valid on $0\le\vartheta\le\vartheta_*<1$.
The ratio is affine in $b$, even at its extended diagonal point:
$R_{\vartheta,b}=R_{\vartheta,0}+bR^{(1)}_\vartheta$.
Both coefficient functions are bounded on this compact set.
Thus $R_{\vartheta,b}\ge-C_0-bC_1$, uniformly over
$(U,V)\in\mathcal K$. Taking the minimum proves
$[c(\vartheta,b)]_-\le C(1+b)$ for every $b\ge0$.
\end{proof}

\subsection{The exact scalar criterion in the low-exponent range}

The elementary exponent in Lemma~\ref{lem:scalar-support} can be sharpened when $1<p<q<2$. Here it suffices to compare the slope $A_{\vartheta,b}=f'(1)$ of
the tangent line used in the proof of Lemma~\ref{lem:scalar-support}
with $\lim_{x\to\infty}f'(x)=1$. The condition $A_{\vartheta,b}\le1$
gives the exact criterion below, even though $f$ is not globally convex.

\begin{proposition}[Exact low-exponent support criterion]\label{prop:low-exact}
Assume $1<p<q<2$, $0<\vartheta<1$, and $b\ge0$. With the notation in \eqref{eq:Ftwo}--\eqref{eq:Gtwo}, define
\begin{equation}\label{eq:low-C}
 C_{p,q}(\vartheta)=
 \frac{q(1-\vartheta)^{1-p}-q+(q-p)\vartheta}{p}.
\end{equation}
Then
\begin{equation}\label{eq:low-iff}
 G_{\vartheta,b}(U,V)\ge0\quad\text{for all }U,V\ge0
 \quad\Longleftrightarrow\quad
 A_{\vartheta,b}\le1
 \quad\Longleftrightarrow\quad
 b\le C_{p,q}(\vartheta).
\end{equation}
If $b<C_{p,q}(\vartheta)$, then $c(\vartheta,b)>0$. This lower bound is uniform on compact subsets of the strict parameter region.
\end{proposition}
\begin{proof}
For $V=1$ and $U\to\infty$, dividing the definition of $G$
by $U^q$ gives
\[
 \frac{G_{\vartheta,b}(U,1)}{U^q}
 =\left|1-\frac\vartheta U\right|^p(1+bU^{-m})
   -A_{\vartheta,b}-B_{\vartheta,b}U^{-q}
 \longrightarrow1-A_{\vartheta,b},
\]
where $m=q-p\in(0,1)$. Thus support requires
$A_{\vartheta,b}\le1$. Solving this inequality using
\eqref{eq:Atwo} gives
\[
\begin{aligned}
 A_{\vartheta,b}\le1
 &\quad\Longleftrightarrow\quad
 q-m\vartheta+pb\le q(1-\vartheta)^{1-p}\\
 &\quad\Longleftrightarrow\quad
 b\le\frac{q((1-\vartheta)^{1-p}-1)+m\vartheta}{p}
      =C_{p,q}(\vartheta).
\end{aligned}
\]
In particular, $C_{p,q}(\vartheta)>0$ for $0<\vartheta<1$,
because both terms in its numerator are positive.

For sufficiency, first suppose $0<b<C_{p,q}(\vartheta)$.
Then $A_{\vartheta,b}<1$, and the general argument for
$0<m<1$, $b>0$ in Lemma~\ref{lem:scalar-support} applies
without any restriction of the form $b\le\vartheta^{n_{p,q}}$.
Specifically, the unique solution of
\[
 \frac{y^{m-1}((q-1)y-m)}{my-(q-1)}
 =\frac b{\vartheta^m},\qquad y>\frac{q-1}{m},
\]
determines the change from $f''>0$ to $f''<0$.
Since the decreasing branch of $f'$ lies above its limit $1$,
the inequality $f'(1)=A_{\vartheta,b}<1$ places $1$ strictly
on the increasing branch. The integral formulas in that proof
then give $f(x)>f(1)+f'(1)(x-1)$ for every $x\ge0$, $x\ne1$,
and also $f''(1)>0$. For $b=0$ the strict-convexity argument
in the same lemma gives these conclusions.
At $b=C_{p,q}(\vartheta)$, take
$b_j=(1-1/j)C_{p,q}(\vartheta)$, $j\ge2$.
For each fixed $U,V\ge0$, continuity in $b$ gives
\[
 G_{\vartheta,C_{p,q}(\vartheta)}(U,V)
 =\lim_{j\to\infty}G_{\vartheta,b_j}(U,V)\ge0.
\]
This proves all equivalences in \eqref{eq:low-iff}.

If $b<C_{p,q}(\vartheta)$, the extended ratio in
\eqref{eq:cdef} is positive at $z=1$ by $f''(1)>0$,
at $z=0$ by strict support, and at $z=\infty$ because
$1-A_{\vartheta,b}>0$. It is continuous and strictly positive
on the compactified half-line, so its minimum is positive.
Finally Lemma~\ref{lem:continuity} makes $c(\vartheta,b)$
continuous. Its positive minimum on any compact subset of
$\{(\vartheta,b):0<\vartheta<1,\ 0\le b<C_{p,q}(\vartheta)\}$
gives the stated uniform bound.
\end{proof}

\begin{proposition}[Optimal power exponent for the scalar inequality]\label{prop:nsc}
For $1<p<q<2$, the number $\nsc$ in \eqref{eq:nsc-main} is finite, is attained at a point of $(0,1)$, and satisfies \eqref{eq:nsc-improvement-main}. For a real exponent $n$, set $b=\vartheta^n$. Then
\begin{equation}\label{eq:nsc-optimal}
 G_{\vartheta,\vartheta^n}(U,V)\ge0
 \ \text{for all }0<\vartheta<1,\ U,V\ge0
 \quad\Longleftrightarrow\quad n\ge\nsc.
\end{equation}
\end{proposition}
\begin{proof}
The positivity of $C_{p,q}$ follows from the numerator formula
in the proof of Proposition~\ref{prop:low-exact}. Taylor expansion gives
\[
 (1-\vartheta)^{1-p}
 =1+(p-1)\vartheta+\frac{p(p-1)}2\vartheta^2
   +O(\vartheta^3).
\]
Substitute this into \eqref{eq:low-C} and use
$q(p-1)+q-p=p(q-1)$ to obtain
\begin{align}
 C_{p,q}(\vartheta)
  &=(q-1)\vartheta+\frac{q(p-1)}2\vartheta^2+O(\vartheta^3)
  &&(\vartheta\downarrow0),\label{eq:low-C-zero}\\
 C_{p,q}(\vartheta)
  &\sim\frac qp(1-\vartheta)^{1-p}
  &&(\vartheta\uparrow1).\label{eq:low-C-one}
\end{align}
The second relation holds because $(1-\vartheta)^{1-p}\to\infty$,
whereas the remaining terms in \eqref{eq:low-C} stay bounded.
For
\[
 L(\vartheta)=\frac{\log C_{p,q}(\vartheta)}{\log\vartheta},
\]
the expansion at zero yields
\[
 L(\vartheta)
 =1+\frac{\log(q-1)+O(\vartheta)}{\log\vartheta}
 \longrightarrow1.
\]
Since $q-1<1$, the numerator of the fraction is negative
for small $\vartheta>0$; its denominator is also negative.
Hence $L(\vartheta)>1$ there. At the other endpoint,
\[
 \log C_{p,q}(\vartheta)
 =\log(q/p)+(1-p)\log(1-\vartheta)+o(1)
 \longrightarrow+\infty,
\]
while $\log\vartheta\to0^-$. Thus $L(\vartheta)\to-\infty$
as $\vartheta\uparrow1$.
Choose an interior point $\vartheta_1$ with $L(\vartheta_1)>1$.
The endpoint limits allow both tails to be removed while
all their values remain below $L(\vartheta_1)$.
On the remaining compact interval $L$ is continuous and
attains its maximum. This proves interior attainment,
finiteness, and $\nsc>1$.

Let $n_0$ be the first line of \eqref{eq:n-pq}, and put
$\vartheta_0=1-2^{-1/(p-1)}$. The identity
$\vartheta_0^{n_0-1}=q-1$ and the two estimates proved in
Lemma~\ref{lem:scalar-support} give
\[
 A_{\vartheta,\vartheta^{n_0}}
 \le
 \begin{cases}
  (1-\vartheta)^{p-1}(1+(p-1)\vartheta)<1,
    &0<\vartheta\le\vartheta_0,\\[1mm]
  (p+q)/(2q)<1,&\vartheta_0\le\vartheta<1.
 \end{cases}
\]
By Proposition~\ref{prop:low-exact}, this is equivalent to
$\vartheta^{n_0}<C_{p,q}(\vartheta)$. Taking logarithms
and dividing by $\log\vartheta<0$ reverses the inequality:
$L(\vartheta)<n_0$ for every $\vartheta\in(0,1)$.
Evaluation at an interior maximizer now gives $\nsc<n_0$,
and hence \eqref{eq:nsc-improvement-main}.
Finally, for every real $n$ and $0<\vartheta<1$,
\[
 \vartheta^n\le C_{p,q}(\vartheta)
 \quad\Longleftrightarrow\quad
 n\log\vartheta\le\log C_{p,q}(\vartheta)
 \quad\Longleftrightarrow\quad n\ge L(\vartheta).
\]
Requiring this at every $\vartheta$ is equivalent to
$n\ge\max_{0<\vartheta<1}L(\vartheta)=\nsc$.
Proposition~\ref{prop:low-exact} therefore proves
\eqref{eq:nsc-optimal}.
\end{proof}

\begin{remark}[The optimized scalar endpoint and its vanishing remainder]\label{rem:nsc-limit}
At a maximizing value of $\vartheta$ in \eqref{eq:nsc-main}, one has $\vartheta^{\nsc}=C_{p,q}(\vartheta)$ and $A_{\vartheta,\vartheta^{\nsc}}=1$. The support remains valid, but the quantitative coefficient
$c(\vartheta,\vartheta^{\nsc})$ is zero: support makes it
nonnegative, while sending $U/V\to\infty$ makes its defining
ratio tend to $1-A_{\vartheta,\vartheta^{\nsc}}=0$. Thus the scalar inequality remains valid at its optimized endpoint,
but its quadratic remainder coefficient vanishes at a maximizing
displacement. Theorem~\ref{thm:main}(iii) proves sharpness up to this
endpoint. Extending the interval further would require a new estimate
for the integrated remainder.
\end{remark}

\section{A reference energy for all logarithmic displacements}\label{sec:reference}

For $f\in L^2(\Cyl)$ define
\begin{align}
 \mathcal D_{\beta,t}[f]&=\int_\R\int_{\Sph^{d-1}}\int_{\Sph^{d-1}}
 K_{\beta,t}(\omega,\sigma)|f(\tau,\omega)-f(\tau+t,\sigma)|^2
 \dd\sigma\dd\omega\dd\tau,\label{eq:Dt}\\
 \eref[f]&=\frac1{\omega_{d-1}}\int_1^2\int_\R
 \int_{\Sph^{d-1}}\int_{\Sph^{d-1}}
 |f(\tau,\omega)-f(\tau+t,\sigma)|^2
 \dd\sigma\dd\omega\dd\tau\dd t.\label{eq:Eref}
\end{align}
Write $S=\omega_{d-1}$. Norms of functions of $\tau$ alone are
taken in $L^2(\R)$, and norms of functions on $\Cyl$ are taken
with respect to $\dd\tau\dd\omega$.
The reference energy averages squared differences over $1\le t\le2$.
Using $|a-b|^2\le2|a|^2+2|b|^2$ and translation invariance gives
\[
 \eref[f]\le\frac2S\int_1^2
      \bigl(S\|f\|_2^2+S\|f(\cdot+t,\cdot)\|_2^2\bigr)\dd t
 =4\|f\|_2^2.
\]
For every fixed $t>0$, integrating $K_{\beta,t}$ in either angular
variable gives $k_\beta(t)<\infty$. Fubini's theorem and the same
estimate give
\[
 \mathcal D_{\beta,t}[f]
 \le2k_\beta(t)\bigl(\|f\|_2^2+
                    \|f(\cdot+t,\cdot)\|_2^2\bigr)
 =4k_\beta(t)\|f\|_2^2.
\]
Thus all fixed-displacement quadratic forms used below are finite.

\begin{lemma}[Control of arbitrary displacements]\label{lem:reference}
There is a constant $C=C(d,s,p)$, independent of $\beta$, such that
\begin{equation}\label{eq:reference-control}
 \mathcal D_{\beta,t}[f]\le C k_\beta(t)(1+t^2)\eref[f]
 \qquad(t>0,\ f\in L^2(\Cyl)).
\end{equation}
Also $\eref[f]=0$ for an $L^2$ function implies $f=0$.
\end{lemma}
\begin{proof}
The kernel in \eqref{eq:cylinderchange} depends on the angles only
through $\omega\cdot\sigma$, since
\[
 |\omega-e^t\sigma|^2=1+e^{2t}-2e^t\omega\cdot\sigma.
\]
It is symmetric in $\omega,\sigma$ and invariant under their
simultaneous rotation. Invariance of surface measure therefore gives
\[
 \int_{\Sph^{d-1}}K_{\beta,t}(\omega,\sigma)\dd\sigma
 =\int_{\Sph^{d-1}}K_{\beta,t}(\sigma,\omega)\dd\sigma
 =k_\beta(t),
\]
independently of $\omega$. These masses are finite for $t>0$,
because $|\omega-e^t\sigma|\ge e^t-1>0$.
Moreover,
\[
 K_{\beta,t}=e^{\beta t}K_{0,t},\qquad
 k_\beta(t)=e^{\beta t}k_0(t).
\]
Hence the normalized kernel $K_{\beta,t}/k_\beta(t)$ does not
depend on $\beta$.

Decompose $f$ into its spherical mean and fluctuation:
\[
 \bar f(\tau)=\frac1S\int_{\Sph^{d-1}}f(\tau,\omega)\dd\omega,
 \qquad g(\tau,\omega)=f(\tau,\omega)-\bar f(\tau).
\]
Cauchy--Schwarz gives
$S|\bar f(\tau)|^2\le\int|f(\tau,\omega)|^2\dd\omega$.
Thus $\bar f\in L^2(\R)$ and $g\in L^2(\Cyl)$, with
\[
 \int_{\Sph^{d-1}}g(\tau,\omega)\dd\omega=0,
 \qquad
 \|f\|_2^2=S\|\bar f\|_2^2+\|g\|_2^2.
\]
For a displacement $r$, put
$a_r(\tau)=\bar f(\tau)-\bar f(\tau+r)$.
In the unweighted angular integral, the terms involving $a_r$
and one factor of $g$ vanish by the zero-mean property.
The cross term between the two fluctuations also vanishes, since
\[
 \int\!\!\int g(\tau,\omega)
          \overline{g(\tau+r,\sigma)}\dd\sigma\dd\omega
 =\left(\int g(\tau,\omega)\dd\omega\right)
  \overline{\left(\int g(\tau+r,\sigma)\dd\sigma\right)}=0.
\]
Expanding the square consequently gives, for almost every $\tau$,
\[
\begin{aligned}
 &\int\!\!\int
 |a_r(\tau)+g(\tau,\omega)-g(\tau+r,\sigma)|^2
 \dd\sigma\dd\omega\\
 &\quad=S^2|a_r(\tau)|^2
       +S\int|g(\tau,\omega)|^2\dd\omega
       +S\int|g(\tau+r,\sigma)|^2\dd\sigma.
\end{aligned}
\]
Integrate in $\tau$, divide by $S$, and then integrate over
$1\le r\le2$. Translation invariance in $\tau$ yields
\begin{equation}\label{eq:Eref-decomp}
 \eref[f]=2\|g\|_2^2+
 S\int_1^2\|\bar f(\cdot+r)-\bar f\|_2^2\dd r.
\end{equation}

For the weighted angular integral the fluctuation--fluctuation
term need not vanish, so retain it as $\mathcal D_{\beta,t}[g]$.
The mixed term between $a_t$ and the fluctuation vanishes by
Fubini's theorem:
\[
\begin{aligned}
 &\int\!\!\int K_{\beta,t}(\omega,\sigma)
        \bigl(g(\tau,\omega)-g(\tau+t,\sigma)\bigr)
          \dd\sigma\dd\omega\\
 &\quad=k_\beta(t)
   \left(\int g(\tau,\omega)\dd\omega
         -\int g(\tau+t,\sigma)\dd\sigma\right)=0.
\end{aligned}
\]
The mean contribution carries the total angular mass $Sk_\beta(t)$,
the integral of $K_{\beta,t}$ over both angular variables.
Hence the exact decomposition and its bound are
\[
\begin{aligned}
 \mathcal D_{\beta,t}[f]
 &=Sk_\beta(t)\|\bar f(\cdot+t)-\bar f\|_2^2
      +\mathcal D_{\beta,t}[g],\\
 \mathcal D_{\beta,t}[g]
 &\le2k_\beta(t)\int_\R
     \left(\int|g(\tau,\omega)|^2\dd\omega
          +\int|g(\tau+t,\sigma)|^2\dd\sigma\right)\dd\tau\\
 &=4k_\beta(t)\|g\|_2^2.
\end{aligned}
\]

We next establish the scalar estimate for translations.
For $F\in L^2(\R)$ define $T_hF(\tau)=F(\tau+h)$ and
\[
 P_F(h)=\|T_hF-F\|_2^2,\qquad
 Q_F=\int_1^2P_F(r)\dd r.
\]
The operators $T_h$ are isometries and satisfy $T_hT_r=T_{h+r}$.
For $|h|\le1/2$, choose the interval of length $1/2$
\[
 I_h=\begin{cases}
       [1,3/2],&0\le h\le1/2,\\
       [3/2,2],&-1/2\le h<0.
      \end{cases}
\]
Both $I_h$ and $I_h+h$ lie in $[1,2]$.
For each $r\in I_h$, translation invariance and the squared
triangle inequality give
\[
\begin{aligned}
 P_F(h)&=\|T_{r+h}F-T_rF\|_2^2\\
       &=\|(T_{r+h}F-F)-(T_rF-F)\|_2^2\\
       &\le2P_F(r+h)+2P_F(r).
\end{aligned}
\]
Averaging over $I_h$ now shows explicitly that
\[
\begin{aligned}
 P_F(h)&\le\frac2{|I_h|}
       \left(\int_{I_h+h}P_F(u)\dd u
             +\int_{I_h}P_F(u)\dd u\right)\\
       &\le4(Q_F+Q_F)=8Q_F.
\end{aligned}
\]
For $t>0$, set $N=\lceil2t\rceil\ge1$ and $h=t/N$.
Then $0<h\le1/2$ and $N\le2t+1\le2(1+t)$.
The telescoping identity in $L^2(\R)$ is
\[
 T_tF-F=\sum_{j=0}^{N-1}T_{jh}(T_hF-F).
\]
Cauchy--Schwarz for this finite sum, followed by the isometry
of $T_{jh}$, gives
\[
 P_F(t)\le N\sum_{j=0}^{N-1}\|T_{jh}(T_hF-F)\|_2^2
          =N^2P_F(h).
\]
Combining this with the small-shift estimate yields
\begin{equation}\label{eq:reference-scalar-chain}
 P_F(t)\le N^2P_F(h)
 \le8N^2Q_F\le32(1+t)^2Q_F.
\end{equation}

Apply \eqref{eq:reference-scalar-chain} to $F=\bar f$ and put
$M_t=32(1+t)^2$. Since $M_t\ge2$, the two parts of the
weighted decomposition can be combined using their respective
terms in \eqref{eq:Eref-decomp}:
\[
\begin{aligned}
 \frac{\mathcal D_{\beta,t}[f]}{k_\beta(t)}
 &\le4\|g\|_2^2+M_t S Q_{\bar f}\\
 &\le M_t\bigl(2\|g\|_2^2+S Q_{\bar f}\bigr)
   =M_t\eref[f]\\
 &\le64(1+t^2)\eref[f].
\end{aligned}
\]
The last step uses $(1+t)^2\le2(1+t^2)$.
This proves \eqref{eq:reference-control}, in fact with $C=64$,
independently of $\beta$.

Finally suppose $\eref[f]=0$. Its nonnegative decomposition
\eqref{eq:Eref-decomp} gives $g=0$ and $Q_{\bar f}=0$.
The small-shift estimate then gives
$P_{\bar f}(h)=0$ for every $|h|\le1/2$.
In particular, $T_{1/2}\bar f=\bar f$ in $L^2(\R)$.
Writing $a=1/2$ and using this equality for all integer translates,
\[
 \|\bar f\|_2^2
 =\sum_{j\in\mathbb Z}\int_0^a|\bar f(\tau+ja)|^2\dd\tau
 =\sum_{j\in\mathbb Z}\int_0^a|\bar f(\tau)|^2\dd\tau.
\]
The sum is finite only if its repeated nonnegative term is zero.
Thus $\bar f=0$ almost everywhere, and $f=\bar f+g=0$.
All identities above are identities of $L^2$ functions or of
absolutely integrable products; no differentiability of $f$ is required.
\end{proof}

\begin{remark}
The reference energy controls the angular fluctuation $g$ in
$L^2$ by \eqref{eq:Eref-decomp}. Its control of the spherical
mean is through translation differences. To quantify the distinction,
let $f_T(\tau,\omega)=\zeta(\tau/T)$, where
$0\ne\zeta\in C_c^\infty(\R)$ and $T\ge1$.
Then $g_T=0$ and
\[
 \|f_T\|_2^2=ST\|\zeta\|_2^2,\qquad
 \eref[f_T]=ST\int_1^2
    \|\zeta(\cdot+r/T)-\zeta\|_2^2\dd r.
\]
The fundamental theorem of calculus and translation invariance give
$\|\zeta(\cdot+a)-\zeta\|_2\le|a|\|\zeta'\|_2$.
Consequently
\[
 \eref[f_T]\le\frac{S}{T}\|\zeta'\|_2^2\int_1^2r^2\dd r
 =\frac{7S}{3T}\|\zeta'\|_2^2,
 \qquad
 \frac{\eref[f_T]}{\|f_T\|_2^2}
 \le\frac7{3T^2}\frac{\|\zeta'\|_2^2}{\|\zeta\|_2^2}
 \xrightarrow[T\to+\infty]{}0.
\]
Hence no positive constant can give
$\eref[f]\ge c\|f\|_2^2$ for every $f\in L^2(\Cyl)$.
The same family explains the growth in the displacement estimate:
\[
 \frac{\mathcal D_{\beta,T}[f_T]}{k_\beta(T)}
 =ST\|\zeta(\cdot+1)-\zeta\|_2^2.
\]
The norm on the right is positive, since a nonzero compactly
supported function cannot have period one. Together with the
bound $\eref[f_T]=O(T^{-1})$, this shows that a bound of the
form \eqref{eq:reference-control} requires a displacement factor
of at least quadratic growth.
\end{remark}

\begin{lemma}[An elementary translation-chain estimate]\label{lem:chain}
Set
\[
 P_f(h)=\frac1{\omega_{d-1}}
 \int_\R\int_{\Sph^{d-1}}\int_{\Sph^{d-1}}
 |f(\tau,\omega)-f(\tau+h,\sigma)|^2
 \dd\sigma\dd\omega\dd\tau.
\]
Then, for every $h\ge1$ and $f\in L^2(\Cyl)$,
\begin{equation}\label{eq:chain}
 P_f(h)\le32(1+h)^2\eref[f].
\end{equation}
\end{lemma}
\begin{proof}
Write $f=\bar f+g$ as in Lemma~\ref{lem:reference}.
The unweighted angular expansion in that proof gives
\[
 P_f(h)=S P_{\bar f}(h)+2\|g\|_2^2,\qquad
 \eref[f]=S Q_{\bar f}+2\|g\|_2^2.
\]
Set $M_h=32(1+h)^2\ge1$. Applying
\eqref{eq:reference-scalar-chain} to $\bar f$ yields
\[
\begin{aligned}
 P_f(h)&\le M_h S Q_{\bar f}+2\|g\|_2^2\\
       &\le M_h\bigl(S Q_{\bar f}+2\|g\|_2^2\bigr)
        =32(1+h)^2\eref[f].
\end{aligned}
\]
This is \eqref{eq:chain}. Grouping the mean and fluctuation
terms before using \eqref{eq:Eref-decomp} preserves the stated
constant $32$.
\end{proof}

\section{Proof of the sharp-interval and strict-extension theorems}\label{sec:proofs}

This section proves Theorems~\ref{thm:extension} and~\ref{thm:main}.
Lemma~\ref{lem:cutoff} already gives the upper bound by the
homogeneous candidate $J_{a,\beta}^{q/p}$. The task in this section
is therefore to prove the matching Hardy lower bound on the
asserted parameter intervals. A common reduction expresses this
task in terms of an integrated scalar remainder.

We first derive that reduction, including the changes of variables
and limiting operations. We then prove the intervals for $D>0$,
treat $D<0$ by reflection, and handle $D=0$ directly. An independent
convexity proof for $q>2p$ then supplies a direct extension width. The final
subsection collects the positive remainder estimates and their
consequences for equality cases and near-extremizers.

\subsection{A common reduction to an integrated remainder}\label{sec:proof-reduction}

We first establish the reduction used in the proofs of both
Theorems~\ref{thm:extension} and~\ref{thm:main}.

Fix an admissible exponent $a$ in \eqref{eq:a-range} and a nonzero
smooth test $u$. The zero test is immediate. Replacing $u$ by $|u|$
preserves the denominator and decreases the energy, since
$||u(x)|-|u(y)||\le|u(x)-u(y)|$.
Thus it suffices to work with
$v(\tau,\omega)=e^{a\tau}|u(e^\tau\omega)|\ge0$.
This function is Lipschitz and compactly supported on the cylinder;
the argument uses no differentiation at its zeros. Set
\begin{equation}\label{eq:IN}
 N[v]=\int_\Cyl v^q,\qquad I_{a,\beta}[v]=\int_\Cyl v^{q-p}\Aa_{a,\beta}[v].
\end{equation}
Here $0<N[v]<\infty$. With $\rho=q/p$ and its conjugate exponent
$\rho'=q/(q-p)$, one has $(q-p)\rho'=q$.
H\"older's inequality therefore gives
\begin{equation}\label{eq:outerHolder}
 I_{a,\beta}[v]\le
 \left(\int_\Cyl\Aa_{a,\beta}[v]^\rho\right)^{1/\rho}
 N[v]^{1-1/\rho}.
\end{equation}
The integral on the right is finite by the admissibility of $a$
and the cylindrical representation. In particular, $I_{a,\beta}[v]$
is finite. Rearranging \eqref{eq:outerHolder}, if
$I_{a,\beta}[v]\ge J_{a,\beta}N[v]+R[v]$ with $R[v]\ge0$, then
\[
 \E_{\alpha,\beta}[u]
 \ge\frac{I_{a,\beta}[v]^\rho}{N[v]^{\rho-1}}
 \ge N[v]\left(J_{a,\beta}+\frac{R[v]}{N[v]}\right)^\rho.
\]
Consequently $I_{a,\beta}[v]\ge J_{a,\beta}N[v]$ is sufficient
for the sharp Hardy lower bound. If $a\ne0$, then $J_{a,\beta}>0$,
and convexity also gives the useful implication
\[
 \E_{\alpha,\beta}[u]-J_{a,\beta}^\rho N[v]
 \ge\rho J_{a,\beta}^{\rho-1}R[v].
\]

To derive the identity, let $a>0$. No restriction on the sign of $D$
is needed at this stage. For $t>0$ define
\begin{equation}\label{eq:theta-b-a}
 \vartheta_a(t)=e^{-at},\qquad b_a(t)=e^{-(D-pa)t}.
\end{equation}
Splitting the displacement integral into its positive and negative
parts and rewriting both contributions as integrals over $t>0$, we
obtain
\begin{equation}\label{eq:pairidentity}
 I_{a,\beta}[v]-J_{a,\beta}N[v]
 =\int_0^\infty\int_\R\int_{\Sph^{d-1}}\int_{\Sph^{d-1}}
 K_{\beta,t}(\omega,\sigma)G_{\vartheta_a(t),b_a(t)}(U,V)
 \dd\sigma\dd\omega\dd\tau\dd t,
\end{equation}
where $U=v(\tau,\omega)$ and $V=v(\tau+t,\sigma)$.
We first prove the identity with a common cutoff. Let $m=q-p>0$
and $0<\epsilon<T$. In the calculation below, $\int\dd\mu$ denotes
integration over $\R\times\Sph^{d-1}\times\Sph^{d-1}$ with
$\dd\mu=\dd\sigma\dd\omega\dd\tau$. Define
\[
\begin{aligned}
 I_{a,\beta,\epsilon,T}[v]
 &=\int_{\epsilon<|r|<T}\int
 K_{\beta,r}(\omega,\sigma)v(\tau,\omega)^m\\
 &\qquad\qquad\times
 |v(\tau,\omega)-e^{-ar}v(\tau+r,\sigma)|^p\dd\mu\dd r,\\
 J_{a,\beta,\epsilon,T}
 &=\int_{\epsilon<|r|<T}k_\beta(r)|1-e^{-ar}|^p\dd r.
\end{aligned}
\]
Every integral at these cutoffs is finite. Split the first integral
as $I_{a,\beta,\epsilon,T}[v]=I^+_{\epsilon,T}[v]+I^-_{\epsilon,T}[v]$
according to the sign of $r$. For $t>0$, abbreviate
$\vartheta=\vartheta_a(t)$ and $b=b_a(t)$; these parameters still
depend on $t$.

\emph{The positive and negative displacements.}
In the positive part, set $r=t$. With $U=v(\tau,\omega)$ and
$V=v(\tau+t,\sigma)$, this gives
\[
 I^+_{\epsilon,T}[v]
 =\int_\epsilon^T\int K_{\beta,t}(\omega,\sigma)
 U^m|U-\vartheta V|^p\dd\mu\dd t.
\]
In the negative part, first set $r=-t$ and reverse the integration
limits. Thus
\[
\begin{aligned}
 I^-_{\epsilon,T}[v]
 &=\int_\epsilon^T\int
 K_{\beta,-t}(\omega,\sigma)v(\tau,\omega)^m\\
 &\qquad\qquad\times
 |v(\tau,\omega)-e^{at}v(\tau-t,\sigma)|^p\dd\mu\dd t.
\end{aligned}
\]
For each fixed $t$, now make the change of variables
\[
 (\tau_{\mathrm{old}},\omega_{\mathrm{old}},\sigma_{\mathrm{old}})
 =(\tau+t,\sigma,\omega).
\]
Translation in $\tau$ and exchange of the two angular variables
preserve the domain and the product measure. The two function
values become, respectively,
\[
 v(\tau_{\mathrm{old}},\omega_{\mathrm{old}})
 =v(\tau+t,\sigma)=V,\qquad
 v(\tau_{\mathrm{old}}-t,\sigma_{\mathrm{old}})
 =v(\tau,\omega)=U.
\]
Consequently the transformed integrand is
$K_{\beta,-t}(\sigma,\omega)V^m|V-e^{at}U|^p$.
The kernel reversal identity \eqref{eq:reversal} and homogeneity give
\[
 K_{\beta,-t}(\sigma,\omega)
 =e^{-Dt}K_{\beta,t}(\omega,\sigma),\qquad
 |V-e^{at}U|^p=e^{pat}|U-e^{-at}V|^p.
\]
The two exponential factors combine as
$e^{-Dt}e^{pat}=e^{-(D-pa)t}=b$. Hence
\[
 I^-_{\epsilon,T}[v]
 =\int_\epsilon^T\int K_{\beta,t}(\omega,\sigma)
 bV^m|U-\vartheta V|^p\dd\mu\dd t.
\]
Adding the two parts and using \eqref{eq:Ftwo} yields
\[
\begin{aligned}
 I_{a,\beta,\epsilon,T}[v]
 &=\int_\epsilon^T\int K_{\beta,t}(\omega,\sigma)
 |U-\vartheta V|^p(U^m+bV^m)\dd\mu\dd t\\
 &=\int_\epsilon^T\int K_{\beta,t}(\omega,\sigma)
 F_{\vartheta,b}(U,V)\dd\mu\dd t.
\end{aligned}
\]
There is no factor $1/2$: the original integral is the sum of its
positive- and negative-displacement parts, and both contributions
are included in $F_{\vartheta,b}$.

\emph{Representing and subtracting $J N[v]$ within the same integral.}
Integrating \eqref{eq:reversal} in an angular variable gives
$k_\beta(-t)=e^{-Dt}k_\beta(t)$. Pairing the two displacement
signs in $J$ therefore gives
\[
\begin{aligned}
 J_{a,\beta,\epsilon,T}
 &=\int_\epsilon^T
 \bigl(k_\beta(t)|1-e^{-at}|^p
       +k_\beta(-t)|1-e^{at}|^p\bigr)\dd t\\
 &=\int_\epsilon^T k_\beta(t)
 \bigl((1-\vartheta)^p+e^{-Dt}e^{pat}(1-\vartheta)^p\bigr)\dd t\\
 &=\int_\epsilon^T k_\beta(t)(1-\vartheta)^p(1+b)\dd t.
\end{aligned}
\]
For each fixed $t>0$, integrating the kernel in either angular
variable gives $k_\beta(t)$, and Fubini's theorem gives
\[
\begin{aligned}
 \int K_{\beta,t}(\omega,\sigma)U^q\dd\mu
 &=k_\beta(t)\int_\R\int_{\Sph^{d-1}}
 v(\tau,\omega)^q\dd\omega\dd\tau
 =k_\beta(t)N[v],\\
 \int K_{\beta,t}(\omega,\sigma)V^q\dd\mu
 &=k_\beta(t)\int_\R\int_{\Sph^{d-1}}
 v(\tau+t,\sigma)^q\dd\sigma\dd\tau
 =k_\beta(t)N[v].
\end{aligned}
\]
The last equality in the second line uses the translation
$\tau\mapsto\tau+t$.

To represent $J_{a,\beta,\epsilon,T}N[v]$ using the same kernel
and integration variables as in the formula for $I_{a,\beta,\epsilon,T}[v]$,
we use the coefficients $A_{\vartheta,b},B_{\vartheta,b}$ from
\eqref{eq:Atwo}--\eqref{eq:Btwo}. Their choice makes
$A_{\vartheta,b}U^q+B_{\vartheta,b}V^q$ agree with
$F_{\vartheta,b}$ in value and first derivatives on the positive
diagonal $U=V>0$, as required for the scalar remainder studied
in Section~\ref{sec:scalar}. Indeed, matching the two first
derivatives at $(1,1)$ gives
\[
 A_{\vartheta,b}=\frac1q\,\partial_U F_{\vartheta,b}(1,1),
 \qquad
 B_{\vartheta,b}=\frac1q\,\partial_V F_{\vartheta,b}(1,1).
\]
Direct differentiation of \eqref{eq:Ftwo} gives exactly
\eqref{eq:Atwo}--\eqref{eq:Btwo}. Since $F_{\vartheta,b}$ is
homogeneous of degree $q$, Euler's identity also gives
\[
 F_{\vartheta,b}(1,1)
 =\frac{\partial_U F_{\vartheta,b}(1,1)
        +\partial_V F_{\vartheta,b}(1,1)}q
 =A_{\vartheta,b}+B_{\vartheta,b}.
\]
Thus the definition \eqref{eq:Gtwo} implies the contact conditions
\[
 G_{\vartheta,b}(1,1)=0,\qquad
 \partial_U G_{\vartheta,b}(1,1)
 =\partial_V G_{\vartheta,b}(1,1)=0.
\]
Homogeneity extends these conditions to every $(c,c)$ with $c>0$.
These contact conditions explain why we subtract
$A_{\vartheta,b}U^q+B_{\vartheta,b}V^q$ from $F_{\vartheta,b}$
and use $G_{\vartheta,b}$ as the remainder in the scalar estimates.

Since $A_{\vartheta,b}$ and $B_{\vartheta,b}$
depend only on $t$, the identity
$A_{\vartheta,b}+B_{\vartheta,b}=(1-\vartheta)^p(1+b)$
from \eqref{eq:supportsum} now implies
\[
\begin{aligned}
 &\int_\epsilon^T\int K_{\beta,t}(\omega,\sigma)
   \bigl(A_{\vartheta,b}U^q+B_{\vartheta,b}V^q\bigr)\dd\mu\dd t\\
 &\qquad=N[v]\int_\epsilon^T k_\beta(t)
   (A_{\vartheta,b}+B_{\vartheta,b})\dd t
 =J_{a,\beta,\epsilon,T}N[v].
\end{aligned}
\]
Subtracting this equality from the paired formula for $I$, and
using the definition \eqref{eq:Gtwo}, proves
\[
\begin{aligned}
 &I_{a,\beta,\epsilon,T}[v]-J_{a,\beta,\epsilon,T}N[v]\\
 &\quad=\int_\epsilon^T\int K_{\beta,t}(\omega,\sigma)
 \bigl(F_{\vartheta,b}(U,V)-A_{\vartheta,b}U^q
                         -B_{\vartheta,b}V^q\bigr)\dd\mu\dd t\\
 &\quad=\int_\epsilon^T\int K_{\beta,t}(\omega,\sigma)
 G_{\vartheta,b}(U,V)\dd\mu\dd t.
\end{aligned}
\]
This is the truncated form of \eqref{eq:pairidentity}.
For each fixed $0<\epsilon<T$, the integrals of
$K_{\beta,t}A_{\vartheta,b}U^q$ and
$K_{\beta,t}B_{\vartheta,b}V^q$ are finite, so the subtraction
is legitimate. We keep the common cutoffs until these terms have
been combined with $K_{\beta,t}F_{\vartheta,b}$ into
$K_{\beta,t}G_{\vartheta,b}$. We do not assume that the first two
integrals converge separately when the lower cutoff is removed
by letting $\epsilon\downarrow0$.

We justify the signed integral when the cutoffs are removed.
Put $f=v^{q/2}\in L^2(\Cyl)$, so $\|f\|_2^2=N[v]$.
As $t\downarrow0$, the pair $(\vartheta_a(t),b_a(t))$ tends to
$(1,1)$. Lemma~\ref{lem:short} therefore makes $G$ nonnegative
for all sufficiently small $t>0$.
On the remaining range $t\ge t_0>0$, the parameter
$\vartheta_a(t)\le e^{-at_0}<1$, and Lemma~\ref{lem:continuity}
implies
\[
 [G_{\vartheta_a(t),b_a(t)}(U,V)]_-
 \le C(1+b_a(t))(U^{q/2}-V^{q/2})^2.
\]
Its integral in $\tau,\omega,\sigma$ is at most
$C(1+b_a(t))\mathcal D_{\beta,t}[f]$.
For $t\ge1$, Lemma~\ref{lem:reference} and \eqref{eq:k0bounds}
bound this by
\[
 C(1+t^2)
 \bigl(e^{-\lambda_+t}+e^{-(\lambda_--pa)t}\bigr)\eref[f].
\]
Here $\lambda_++D=\lambda_-$ and admissibility gives
$\lambda_--pa>0$. This is integrable in $t$; on every compact
subinterval of $(0,\infty)$ the coefficients and $k_\beta$
are bounded. Thus the negative part of the full paired integrand
has finite integral.
For nested cutoffs, the integrals of its positive and negative
parts both increase. The latter have a finite limit, and the
truncated identity bounds the former by
$I_{a,\beta}[v]$ plus that finite negative-part integral.
Both parts are consequently integrable. Tonelli and monotone
convergence for $I$ and $J$, and then for these two parts, prove
\eqref{eq:pairidentity}. This also explains why the same cutoff
must be used in the two terms on its left.

With the notation
\begin{equation}\label{eq:fcentral}
 f=v^{q/2},\qquad c_{a,\beta}(t)=c(\vartheta_a(t),b_a(t)),
\end{equation}
the defining scalar inequality for $c$ gives
\begin{equation}\label{eq:pairlower}
 I_{a,\beta}[v]-J_{a,\beta}N[v]\ge\int_0^\infty c_{a,\beta}(t)\mathcal D_{\beta,t}[f]\dd t.
\end{equation}
The negative part on the right is integrable by the preceding
bound; its positive part is bounded by the positive paired
integrand. Thus this comparison is also justified as a signed
integral. It reduces the nonlinear remainder to a quadratic
energy of $v^{q/2}$.

All subsequent quantitative estimates use the same fixed interval
$1\le t\le2$ as the reference energy in \eqref{eq:Eref}. On this
interval, $|\omega-e^t\sigma|\le1+e^2$ and
$e^{(d+\beta)t}\ge e^{d+\beta}$. Since $d+\beta>0$,
\[
 K_{\beta,t}(\omega,\sigma)\ge
 k_*:=\frac{e^{d+\beta}}{(1+e^2)^{d+sp}}>0
 \qquad(1\le t\le2).
\]
Writing $S=\omega_{d-1}$, the definition of $\eref$ gives
$\int_1^2\mathcal D_{\beta,t}[f]\dd t\ge Sk_*\eref[f]$.
This kernel estimate is independent of $a$ and of the sign of $D$;
it will be used in each of the cases below.

\subsection{Sharpness when \texorpdfstring{$D>0$}{D>0}}\label{sec:low-new}

Throughout this subsection $D>0$. We prove
Theorem~\ref{thm:extension}\textup{(i)} in this case and all three
parts of Theorem~\ref{thm:main}. We first obtain the initial
interval from scalar support and then enlarge it by the integrated
remainder estimate. In the low-exponent range, we also obtain an
explicit enlargement of the initial interval from the sharper scalar
criterion. We finish by
converting these intervals to the original weight $\alpha$.

\emph{The initial interval.}
Put $a_0=D/(p+n_{p,q})$.
Since $n_{p,q}>0$ and $D<\lambda_-$, one has
$0<a_0<D/p<\lambda_-/p$.
For $0<a\le a_0$, division by $a>0$ gives
$D/a-p\ge n_{p,q}$. Because $0<\vartheta_a(t)<1$,
\[
 b_a(t)=\vartheta_a(t)^{D/a-p}
 \le\vartheta_a(t)^{n_{p,q}}.
\]
Lemma~\ref{lem:scalar-support} makes every paired remainder
nonnegative. Equations \eqref{eq:outerHolder}--\eqref{eq:pairlower}
therefore give $\E_{\alpha_H+qa,\beta}[u]\ge J_{a,\beta}^\rho N[v]$.
The test sequence in Lemma~\ref{lem:cutoff} gives the reverse
bound for the infimum, and hence
\begin{equation}\label{eq:seedsharp}
 H_{\alpha_H+qa,\beta}^{p,q}=J_{a,\beta}^\rho\qquad(0\le a\le a_0).
\end{equation}
At $a=0$, nonnegativity of the energy and the same cutoff upper
bound give $0\le H_{\alpha_H,\beta}^{p,q}\le J_{0,\beta}^\rho=0$.

\emph{Extension by the integrated remainder.}
We next extend the interval past $a_0$ by estimating the integrated
remainder in \eqref{eq:pairlower}. This argument applies to all
$p<q$; in particular, it supplies the strict extension needed when
$q\ge2$. The factor $\mathcal D_{\beta,t}[f]$ is nonnegative, but
$c_{a,\beta}(t)$ may become negative when $a>a_0$. We retain a
uniformly positive contribution from $1\le t\le2$ and bound the
absolute value of the integral over the set where $c_{a,\beta}(t)<0$.
The comparison will leave a positive total remainder.

We first ensure that sufficiently small logarithmic shifts have
nonnegative remainder uniformly in $a$. Choose a compact interval
$\mathcal U_0$ containing $a_0$ in its interior, with
$0<a_-\le a\le a_+<D/p$ on $\mathcal U_0$.
This also ensures admissibility. Put
$L=\max_{a\in\mathcal U_0}(D-pa)>0$.
Choose $t_0\in(0,1)$ so small that
\[
 e^{-a_+t_0}>1-h_0,\qquad e^{-Lt_0}>1/2,
\]
where $h_0$ is from Lemma~\ref{lem:short}.
For $a\in\mathcal U_0$ and $0<t<t_0$ this gives
$1-h_0<\vartheta_a(t)<1$ and $1/2<b_a(t)\le1$.
The lemma thus proves
\begin{equation}\label{eq:shortpositive}
 c_{a,\beta}(t)\ge0\qquad (0<t<t_0,\ a\in\mathcal U_0).
\end{equation}

We next obtain a positive lower bound on the fixed interval
$1\le t\le2$. At $a=a_0$, the identity $D-pa_0=n_{p,q}a_0$ gives
$b_{a_0}(t)=\vartheta_{a_0}(t)^{n_{p,q}}$.
Lemma~\ref{lem:scalar-support} therefore gives $c_{a_0,\beta}(t)>0$.
By Lemma~\ref{lem:continuity}, this coefficient is continuous in
$t$, so its minimum on the compact interval $[1,2]$ is positive.
Uniform continuity in $(a,t)$ then allows us to shrink $\mathcal U_0$
to a compact neighborhood $\mathcal U$ of $a_0$ such that
$c_{a,\beta}(t)\ge\gamma>0$ for $a\in\mathcal U$, $1\le t\le2$.
Using the kernel lower bound $k_*$ above yields
\begin{equation}\label{eq:strippositive}
 \int_1^2c_{a,\beta}(t)\mathcal D_{\beta,t}[f]\dd t\ge c_*\eref[f]
 \qquad(a\in\mathcal U),\qquad c_*=\gamma Sk_*>0.
\end{equation}

It remains to estimate the part where $c_{a,\beta}(t)<0$.
By \eqref{eq:shortpositive}, this can occur only for $t\ge t_0$.
Define the total possible negative cost by
\begin{equation}\label{eq:negativecost}
 \mathcal N_\beta(a)=\int_{t_0}^\infty[c_{a,\beta}(t)]_-k_\beta(t)(1+t^2)\dd t.
\end{equation}
For $t\ge t_0$ and $a\in\mathcal U$, one has
$\vartheta_a(t)\le e^{-a_-t_0}<1$ and $b_a(t)\le1$.
Lemma~\ref{lem:continuity} therefore gives
$[c_{a,\beta}(t)]_-\le2C$ with $C$ independent of $a,t$.
The majorant $2Ck_\beta(t)(1+t^2)$ is integrable on
$[t_0,\infty)$: it is bounded on compact intervals away from
zero and decays exponentially at infinity because $\lambda_+>0$.
For each fixed $t>0$, continuity gives
$[c_{a,\beta}(t)]_-\to[c_{a_0,\beta}(t)]_-=0$.
Dominated convergence proves
\begin{equation}\label{eq:Nsmall}
 \mathcal N_\beta(a)\longrightarrow0\qquad(a\to a_0).
\end{equation}
In \eqref{eq:pairlower}, keep the integral over $1\le t\le2$,
where $c_{a,\beta}(t)\ge\gamma>0$. Outside this interval, discard
the nonnegative contributions and estimate the absolute value
of the negative part using Lemma~\ref{lem:reference}. Together
with \eqref{eq:strippositive} and the definition
\eqref{eq:negativecost}, this gives
\[
 I_{a,\beta}[v]-J_{a,\beta}N[v]
 \ge\bigl(c_*-C_{\mathrm{ref}}\mathcal N_\beta(a)\bigr)\eref[f],
 \qquad C_{\mathrm{ref}}=64.
\]
Choose $\eta>0$ with $[a_0,a_0+\eta]\subset\mathcal U$ and
$C_{\mathrm{ref}}\mathcal N_\beta(a)\le c_*/2$ throughout that
interval. Then
\begin{equation}\label{eq:globalremainder}
 I_{a,\beta}[v]-J_{a,\beta}N[v]\ge c\eref[v^{q/2}]
 \qquad(a_0\le a\le a_0+\eta),\qquad c=c_*/2>0.
\end{equation}
All choices depend only on the fixed parameters, not on $v$.
Combining \eqref{eq:outerHolder} with \eqref{eq:globalremainder} gives
\[
 \E_{\alpha,\beta}[u]\ge N[v]
 \left(J_{a,\beta}+\frac{c_*\eref[v^{q/2}]}{2N[v]}\right)^{q/p}
 \ge J_{a,\beta}^{q/p}N[v].
\]
Lemma~\ref{lem:cutoff} gives the matching upper bound. Together
with \eqref{eq:seedsharp}, this proves
Theorem~\ref{thm:extension}(i) for $D>0$. A sufficient width is any
$\eta$ satisfying the neighborhood, positivity, and negative-part
estimates just established.

\emph{The optimized low-exponent interval.}
The common argument proves a positive extension width for every
$p<q$. For $1<p<q<2$, the sharper scalar criterion gives an
explicit enlargement of the initial interval.
If $0<a\le D/(p+\nsc)$, then
$D/a-p\ge\nsc$. Propositions~\ref{prop:low-exact} and~\ref{prop:nsc}
give, for every $t>0$,
\[
 b_a(t)=\vartheta_a(t)^{D/a-p}
 \le\vartheta_a(t)^{\nsc}
 \le C_{p,q}(\vartheta_a(t)).
\]
Thus \eqref{eq:pairidentity} has nonnegative integrand, including
at the optimized endpoint. Combining this nonnegativity with
\eqref{eq:outerHolder} gives the lower bound
$H_{\alpha_H+qa,\beta}^{p,q}\ge J_{a,\beta}^{q/p}$,
while Lemma~\ref{lem:cutoff} gives the matching upper bound. Hence
\begin{equation}\label{eq:low-improved-sharp}
 H_{\alpha_H+qa,\beta}^{p,q}=J_{a,\beta}^{q/p}
 \qquad\left(0\le a\le\frac{D}{p+\nsc}\right).
\end{equation}
The value at zero was proved above. Since $\alpha=\alpha_H+qa$,
the upper endpoint is
$sq-d-q\beta/p+q(d+2\beta-sp)/(p+\nsc)$, which is
\eqref{eq:main-low-alpha}. All these parameters are admissible,
because $D/(p+\nsc)<D/p<\lambda_-/p$.

\emph{Conclusion in the original weight.}
For $q>2p$, $n_{p,q}=q-p$ and $a_0=D/q$. The identity
\[
 \alpha_H+qa_0=\alpha_H+D=s(q-p)+(2-q/p)\beta
\]
shows that \eqref{eq:main-high-alpha} follows from the common
extension with $\varepsilon=q\eta$.
For $2\le q\le2p$, the common absorption proof has
$n_{p,q}=1$, and its starting endpoint is
\[
 \alpha_H+\frac{qD}{p+1}
 =\frac{sq+(q-p-1)d+q(p-1)\beta/p}{p+1}.
\]
Converting its width by $\varepsilon=q\eta$ proves
\eqref{eq:main-middle-alpha}. The low-exponent assertion is
\eqref{eq:low-improved-sharp}. The admissibility checks above
apply to each interval.

\subsection{Sharpness when \texorpdfstring{$D\le0$}{D<=0}}\label{sec:proof-other-signs}

It remains to prove Theorem~\ref{thm:extension}\textup{(i)} for
$D<0$ and part~\textup{(ii)} for $D=0$.
For $D<0$ we verify that reflection preserves the ingredients of
the positive-$D$ proof. For $D=0$ the starting point is $a=0$,
and we compare the positive integral and its possibly negative
tail directly.

\emph{The case $D<0$: reflection.}
Let $D=-\widetilde D<0$ and $a=-\widetilde a<0$.
Define $\widetilde v(\tau,\omega)=v(-\tau,\omega)$ and
$\widetilde K_{\beta,t}(\omega,\sigma)=K_{\beta,-t}(\omega,\sigma)$.
In the full inner energy at $(-\tau,\omega)$, substitute the old
displacement $r=-t$. Since $a=-\widetilde a$, this gives
\[
\begin{aligned}
 \Aa_{a,\beta}[v](-\tau,\omega)
 &=\int_\R\int_{\Sph^{d-1}}K_{\beta,-t}(\omega,\sigma)
   |v(-\tau,\omega)-e^{-\widetilde a t}v(-\tau-t,\sigma)|^p
   \dd\sigma\dd t\\
 &=\int_\R\int_{\Sph^{d-1}}\widetilde K_{\beta,t}(\omega,\sigma)
   |\widetilde v(\tau,\omega)
     -e^{-\widetilde a t}\widetilde v(\tau+t,\sigma)|^p
   \dd\sigma\dd t.
\end{aligned}
\]
Changing $\tau$ to $-\tau$ in the outer integral preserves its
value and $N[v]$. The reflected kernel has
$\widetilde k_\beta(t)=k_\beta(-t)$, and the same substitution gives
\[
 \int_\R\widetilde k_\beta(t)|1-e^{-\widetilde a t}|^p\dd t
 =\int_\R k_\beta(r)|1-e^{\widetilde a r}|^p\dd r
 =J_{-\widetilde a,\beta}.
\]
Thus the reflected candidate is exactly the original candidate.

By symmetry in the angles and \eqref{eq:reversal},
\[
 \widetilde K_{\beta,t}=e^{-Dt}K_{\beta,t},\qquad
 \widetilde K_{\beta,-t}(\sigma,\omega)
 =e^{Dt}\widetilde K_{\beta,t}(\omega,\sigma)
 =e^{-\widetilde Dt}\widetilde K_{\beta,t}(\omega,\sigma).
\]
The reflected pairing therefore has positive parameters
\[
 \widetilde\vartheta=e^{-\widetilde a t},\qquad
 \widetilde b=e^{-(\widetilde D-p\widetilde a)t},\qquad t>0.
\]
In particular, all scalar lemmas are applied with
$0<\widetilde\vartheta<1$. The equivalent algebraic reversal
of an unsplit pair is
\[
 |U-\vartheta V|^p(U^{q-p}+bV^{q-p})
 =\vartheta^p b\,|V-\vartheta^{-1}U|^p
                    (V^{q-p}+b^{-1}U^{q-p}).
\]
Indeed, the difference contributes $\vartheta^p$ and the
second factor contributes $b$; for the original paired
parameters their product is $\vartheta^pb=e^{-Dt}$.

The kernel properties in the absorption proof are preserved.
For any function $h$ on the cylinder,
\[
 \widetilde{\mathcal D}_{\beta,t}[h]=e^{-Dt}\mathcal D_{\beta,t}[h],
 \qquad \widetilde k_\beta(t)=e^{-Dt}k_\beta(t).
\]
The reference estimate is unchanged because its two sides acquire
the same factor. The reflected kernel has a positive
minimum on $[1,2]$ times the two spheres. Its tails satisfy
\[
 \widetilde k_\beta(t)\asymp e^{-\lambda_-t}\quad(t\to+\infty),
 \qquad
 \widetilde k_\beta(t)\asymp e^{\lambda_+t}\quad(t\to-\infty).
\]
For example, the first exponent is
$\lambda_++D=\lambda_-$ after multiplication by $e^{-Dt}$.
Both decay rates are positive. Also, for
$\widetilde f(\tau,\omega)=f(-\tau,\omega)$, translation and
exchange of the two angles give
$P_{\widetilde f}(t)=P_f(-t)=P_f(t)$, and hence
$\eref[\widetilde f]=\eref[f]$.

The absorption proof therefore applies with
$(D,a,k_\beta)$ replaced by
$(\widetilde D,\widetilde a,\widetilde k_\beta)$.
Its initial endpoint is
$\widetilde a_0=\widetilde D/(p+n_{p,q})$.
Since $\widetilde D=\lambda_+-\lambda_-<\lambda_+$, choose
the extension neighborhood inside
$0<\widetilde a<\widetilde D/p<\lambda_+/p$.
This ensures both $\widetilde D-p\widetilde a>0$ and the
original admissibility $a=-\widetilde a>-\lambda_+/p$.
Reflecting back gives sharpness on
$-(\widetilde a_0+\eta)\le a\le0$, as asserted in
Theorem~\ref{thm:extension}(i) for $D<0$.

\emph{The case $D=0$: a direct comparison.}
Suppose $D=0$, or $d+2\beta=sp$. Then
$\lambda:=\lambda_+=\lambda_->0$. For $a>0$, the paired
parameters are $\vartheta_a(t)=e^{-at}$ and $b_a(t)=e^{pat}$.
Write $h_0,c_{\mathrm{sh}}>0$ for the constants in
Lemma~\ref{lem:short}, and choose a displacement constant $c_0>0$
such that
\[
 e^{-c_0}>1-h_0,\qquad e^{pc_0}<2.
\]
For $0<t<c_0/a$ one then has
$1-h_0<\vartheta_a(t)<1$ and $1\le b_a(t)<2$.
Thus the scalar remainder is nonnegative throughout this range.
Choose $a_*>0$ with
$a_*<\min\{\lambda/(2p),c_0/2\}$, so $[1,2]$ is inside it
for every $0<a\le a_*$.
On that strip,
\[
 1-e^{-at}=at\int_0^1e^{-rat}\dd r
 \ge at e^{-at}\ge a e^{-2a_*}.
\]
Set $r_p=(p-2)_+$. Lemma~\ref{lem:short} consequently gives
$c_{a,\beta}(t)\ge c_{\mathrm{sh}}e^{-2a_*r_p}a^{r_p}$ there.
Combining this with $K_{\beta,t}\ge k_*$ on $[1,2]$ yields
\begin{equation}\label{eq:balancedpositive}
 \int_1^2c_{a,\beta}(t)\mathcal D_{\beta,t}[f]\dd t
 \ge c_+a^{r_p}\eref[f],\qquad
 c_+=c_{\mathrm{sh}}e^{-2a_*r_p}Sk_*>0.
\end{equation}

Only $t\ge T_a:=c_0/a$ can contribute a negative remainder.
There $\vartheta_a(t)\le e^{-c_0}<1$, so
Lemma~\ref{lem:continuity} gives
$[c_{a,\beta}(t)]_-\le C(1+e^{pat})$.
Moreover $T_a>2$, $\lambda-pa\ge\lambda/2$, and the kernel
tail is bounded by $Ce^{-\lambda t}$. Hence
\[
\begin{aligned}
 &\int_{T_a}^\infty[c_{a,\beta}(t)]_-k_\beta(t)(1+t^2)\dd t\\
 &\quad\le C\int_{T_a}^\infty(1+t^2)
       \bigl(e^{-\lambda t}+e^{-(\lambda-pa)t}\bigr)\dd t\\
 &\quad\le C\int_{T_a}^\infty(1+t^2)e^{-\lambda t/2}\dd t.
\end{aligned}
\]
For $\mu>0$, integration by parts gives
\[
 \int_T^\infty(1+t^2)e^{-\mu t}\dd t
 =e^{-\mu T}
   \left(\frac{1+T^2}{\mu}+\frac{2T}{\mu^2}+
         \frac2{\mu^3}\right).
\]
Take $\mu=\lambda/2$ and $T=c_0/a$. With constants independent
of $0<a\le a_*$, this proves
\begin{equation}\label{eq:balancednegative}
 \int_{c_0/a}^\infty[c_{a,\beta}(t)]_-k_\beta(t)(1+t^2)\dd t
 \le C(1+a^{-2})e^{-c/a},\qquad c=\lambda c_0/2>0.
\end{equation}
The ratio of the right side to $a^{r_p}$ tends to zero:
putting $y=1/a$ gives
\[
 a^{-r_p}(1+a^{-2})e^{-c/a}
 =(y^{r_p}+y^{r_p+2})e^{-cy}\longrightarrow0,
\]
because $cy$ dominates each fixed multiple of $\log y$.
Lemma~\ref{lem:reference} thus allows a choice
$0<\eta\le a_*$ for which the negative contribution is at most
$c_+a^{r_p}\eref[f]/2$ for every $0<a\le\eta$.
Keeping the positive strip in \eqref{eq:balancedpositive} gives
\[
 I_{a,\beta}[v]-J_{a,\beta}N[v]
 \ge\tfrac12c_+a^{r_p}\eref[v^{q/2}]\ge0.
\]
The common-cutoff identity justifies the comparison and its limit.
H\"older's inequality gives the Hardy lower bound, and
Lemma~\ref{lem:cutoff} gives its sharpness.

If $a<0$, the preceding reflection has the same kernel, because
$D=0$ and symmetry give $K_{\beta,-t}=K_{\beta,t}$.
It replaces $a$ by $|a|$ and preserves the denominator,
reference energy, and the candidate integral. At $a=0$, the
cutoff upper bound gives the zero best constant as before.
The common width $\eta<\lambda/p$ is admissible on both sides.
This proves Theorem~\ref{thm:extension}(ii).

\subsection{An explicit alternative when \texorpdfstring{$q>2p$}{q>2p}}\label{sec:high-new}

We give an independent proof of
Theorem~\ref{thm:main}\textup{(ii)} for $D>0$ and $q>2p$,
using convexity to obtain a direct choice of extension width. The argument also
retains a quantitative remainder. In this subsection $m=q-p>p>1$;
the convexity estimate controls a bounded range of logarithmic
radial shifts, and Lemma~\ref{lem:chain} controls the remaining tail.

\begin{lemma}[Convexity and a uniform positive remainder]\label{lem:perspective}
Define
\[
 \Phi_\lambda(\xi,\zeta)
 =\left(\xi^{m/q}+\lambda\zeta^{m/q}\right)
   \left|\xi^{1/q}-\zeta^{1/q}\right|^p,
 \qquad \xi,\zeta\ge0,
\]
and let
\begin{equation}\label{eq:beta-width}
 B_*=1+\frac{p-1}{2m}>1.
\end{equation}
Then $\Phi_\lambda$ is convex for $1\le\lambda\le B_*$.
If $b=\lambda\vartheta^m$ with $0<\vartheta<1$, then
$G_{\vartheta,b}\ge0$. For $\vartheta$ in a compact subinterval
of $(0,1)$ and $1\le\lambda\le B_*$, there is a uniform
positive constant $c_*$ such that
\[
 G_{\vartheta,\lambda\vartheta^m}(U,V)
 \ge c_*(U^{q/2}-V^{q/2})^2.
\]
\end{lemma}
\begin{proof}
We first prove the convexity of $\Phi_\lambda$, then identify the
supporting plane that gives $G_{\vartheta,\lambda\vartheta^m}\ge0$.
Finally, we show that the remainder divided by
$(U^{q/2}-V^{q/2})^2$ has a uniform positive lower bound.

\smallskip\noindent
\textbf{Step 1. Convexity of $\Phi_\lambda$.}
Fix $1\le\lambda\le B_*$. For $\xi>0$ write
$\Phi_\lambda(\xi,\zeta)=\xi g_\lambda(\zeta/\xi)$, where
\[
 g_\lambda(z^q)=H_\lambda(z):=(1+\lambda z^m)|1-z|^p.
\]
For $z>0$, $z\ne1$, differentiation gives
\begin{align*}
 H_\lambda'(z)
 &=\lambda m z^{m-1}|1-z|^p
   +p(1+\lambda z^m)(z-1)|1-z|^{p-2},\\
 H_\lambda''(z)
 &=\lambda m(m-1)z^{m-2}|1-z|^p\\
 &\quad+2\lambda mpz^{m-1}(z-1)|1-z|^{p-2}
   +p(p-1)(1+\lambda z^m)|1-z|^{p-2}.
\end{align*}
Since $\dd/\dd(z^q)=(qz^{q-1})^{-1}\dd/\dd z$, substitution yields
\[
\begin{aligned}
 g_\lambda''(z^q)
 &=\frac{zH_\lambda''(z)-(q-1)H_\lambda'(z)}{q^2z^{2q-1}}\\
 &=\frac p{q^2}z^{1-2q}|1-z|^{p-2}M_\lambda(z),\\
 M_\lambda(z)
 &=(q-1)-mz+\lambda z^{m-1}((q-1)z-m).
\end{aligned}
\]
At $\lambda=1$, regrouping with $q-1=(p-1)+m$ gives
\begin{equation}\label{eq:M-perspective}
 M_1(z)=(p-1)(1+z^m)+m(1-z)(1-z^{m-1})
 \ge(p-1)(1+z^m)>0.
\end{equation}
The last product is nonnegative on both sides of $z=1$, since
$m>1$. The change in $M$ is
\[
 M_\lambda(z)-M_1(z)
 =(\lambda-1)z^{m-1}((q-1)z-m).
\]
For $0<z\le1$, its negative part is at most
$(\lambda-1)m$, because $z^{m-1}\le1$ and
$((q-1)z-m)_-\le m$. Thus
\[
 M_\lambda(z)\ge p-1-(B_*-1)m=(p-1)/2>0.
\]
For $z\ge1$ the change is nonnegative, since
$(q-1)z-m\ge q-1-m=p-1>0$.
Hence $g_\lambda''>0$ on $(0,1)$ and $(1,\infty)$.
The derivative $g_\lambda'$ is continuous at $1$, with value zero,
because $p>1$. It is therefore strictly increasing across that
point as well. By continuity at zero, $g_\lambda$ is strictly
convex on $[0,\infty)$.

We now deduce the joint convexity of $\Phi_\lambda$ from the
convexity of $g_\lambda$. For $0<\gamma<1$ and two pairs
$(\xi_i,\zeta_i)$ with $\xi_i>0$ and $\zeta_i\ge0$, set
$\xi=\gamma\xi_1+(1-\gamma)\xi_2$ and
$w_1=\gamma\xi_1/\xi$, $w_2=(1-\gamma)\xi_2/\xi$.
Then $w_1+w_2=1$, and
\[
\begin{aligned}
 \Phi_\lambda(\xi,\gamma\zeta_1+(1-\gamma)\zeta_2)
 &=\xi g_\lambda\left(w_1\frac{\zeta_1}{\xi_1}
                      +w_2\frac{\zeta_2}{\xi_2}\right)\\
 &\le\gamma\xi_1 g_\lambda(\zeta_1/\xi_1)
      +(1-\gamma)\xi_2 g_\lambda(\zeta_2/\xi_2).
\end{aligned}
\]
Equivalently,
\[
 \Phi_\lambda\!\bigl(\gamma\xi_1+(1-\gamma)\xi_2,\,
                     \gamma\zeta_1+(1-\gamma)\zeta_2\bigr)
 \le\gamma\Phi_\lambda(\xi_1,\zeta_1)
       +(1-\gamma)\Phi_\lambda(\xi_2,\zeta_2).
\]
The explicit formula for $\Phi_\lambda$ is continuous on the
closed quadrant, with
$\Phi_\lambda(\xi,0)=\xi$ and $\Phi_\lambda(0,\zeta)=\lambda\zeta$.
Passing to limits extends the convexity inequality to both axes.

\smallskip\noindent
\textbf{Step 2. The supporting inequality for $F$.}
Fix $\vartheta\in(0,1)$ and put $r_0=\vartheta^q$.
We will use the supporting plane of $\Phi_\lambda$ at $(1,r_0)$
to obtain the required lower bound for $F_{\vartheta,\lambda\vartheta^m}$.
The connection between the two functions is
\[
 F_{\vartheta,\lambda\vartheta^m}(U,V)
 =\Phi_\lambda(U^q,r_0V^q)
\]
which follows by taking the $q$th roots in each factor.
We first compute the two coefficients of the supporting plane.
At $(\xi,\zeta)=(1,r_0)$ the derivatives are
\[
 \partial_\xi\Phi_\lambda(1,r_0)=g_\lambda(r_0)-r_0g_\lambda'(r_0),
 \qquad \partial_\zeta\Phi_\lambda(1,r_0)=g_\lambda'(r_0).
\]
The second coefficient after composition is
\[
 r_0g_\lambda'(r_0)
 =\frac{(1-\vartheta)^{p-1}}q
    \bigl(\lambda\vartheta^m(m-q\vartheta)-p\vartheta\bigr)
 =B_{\vartheta,\lambda\vartheta^m}.
\]
Using $g_\lambda(r_0)=(1-\vartheta)^p(1+\lambda\vartheta^m)$
and \eqref{eq:supportsum}, the coefficient of $U^q$ in the tangent plane
after composition is
\[
\begin{aligned}
 \partial_\xi\Phi_\lambda(1,r_0)
 &=g_\lambda(r_0)-r_0g_\lambda'(r_0)
  =(1-\vartheta)^p(1+\lambda\vartheta^m)
   -B_{\vartheta,\lambda\vartheta^m}\\
 &=\frac{(1-\vartheta)^{p-1}}q
   \bigl(q-m\vartheta+p\lambda\vartheta^m\bigr)
 =A_{\vartheta,\lambda\vartheta^m}.
\end{aligned}
\]
To write the supporting plane explicitly, use the one-homogeneity
of $\Phi_\lambda$. Since $m+p=q$, the identity
$\Phi_\lambda(c\xi,c\zeta)=c\Phi_\lambda(\xi,\zeta)$ holds for $c>0$.
Differentiating in $c$ at $c=1$, with $(\xi,\zeta)=(1,r_0)$,
gives Euler's identity at this point:
\[
 \Phi_\lambda(1,r_0)
 =\partial_\xi\Phi_\lambda(1,r_0)
  +r_0\partial_\zeta\Phi_\lambda(1,r_0).
\]
The supporting plane at $(1,r_0)$ is therefore
\[
\begin{aligned}
 L_\lambda(\xi,\zeta)
 &=\Phi_\lambda(1,r_0)
   +\partial_\xi\Phi_\lambda(1,r_0)(\xi-1)
   +\partial_\zeta\Phi_\lambda(1,r_0)(\zeta-r_0)\\
 &=\partial_\xi\Phi_\lambda(1,r_0)\,\xi
   +\partial_\zeta\Phi_\lambda(1,r_0)\,\zeta.
\end{aligned}
\]
The constant term cancels by the preceding identity.
Convexity gives $\Phi_\lambda(\xi,\zeta)\ge L_\lambda(\xi,\zeta)$
for $\xi,\zeta\ge0$, with the axis cases following by continuity.
Substituting $(\xi,\zeta)=(U^q,r_0V^q)$ and using the two
coefficients computed above gives
\[
 F_{\vartheta,\lambda\vartheta^m}(U,V)
 =\Phi_\lambda(U^q,r_0V^q)
 \ge L_\lambda(U^q,r_0V^q)
 =A_{\vartheta,\lambda\vartheta^m}U^q
   +B_{\vartheta,\lambda\vartheta^m}V^q.
\]
This is precisely $G_{\vartheta,\lambda\vartheta^m}(U,V)\ge0$.
For later use, the inequality is strict when $U,V>0$ and $U\ne V$.
Indeed, with $r=r_0V^q/U^q$ the remainder can be written as
\[
 G_{\vartheta,\lambda\vartheta^m}(U,V)
 =U^q\bigl(g_\lambda(r)-g_\lambda(r_0)
                 -g_\lambda'(r_0)(r-r_0)\bigr).
\]
Strict convexity of $g_\lambda$ makes the expression in parentheses
positive unless $r=r_0$, which is equivalent to $U=V$.

\smallskip\noindent
\textbf{Step 3. A uniform positive quadratic remainder.}
Fix a compact interval $I=[\vartheta_-,\vartheta_+]\Subset(0,1)$.
Both $G_{\vartheta,\lambda\vartheta^m}$ and
$(U^{q/2}-V^{q/2})^2$ are homogeneous of degree $q$ in $(U,V)$.
Since the desired inequality is immediate at $(0,0)$, it suffices
to work on the compact set
\[
 \mathcal N=\{(U,V)\in[0,\infty)^2:U^q+V^q=1\}.
\]
For $U\ne V$ define
\[
 R_{\vartheta,\lambda}(U,V)
 =\frac{G_{\vartheta,\lambda\vartheta^m}(U,V)}
        {(U^{q/2}-V^{q/2})^2}.
\]
Step~2 shows that this ratio is positive when $U,V>0$ and $U\ne V$.
To obtain a positive minimum, we check the two endpoints of
$\mathcal N$ and extend the ratio continuously across its diagonal point.

At $(U,V)=(1,0)$, the denominator equals one, and strict convexity
of $g_\lambda$ gives
\[
 R_{\vartheta,\lambda}(1,0)
 =1-A_{\vartheta,\lambda\vartheta^m}
 =g_\lambda(0)-g_\lambda(r_0)+r_0g_\lambda'(r_0)>0.
\]
Here we used $g_\lambda(0)=1$ and the strict inequality between
$g_\lambda(0)$ and the value at zero of the tangent line at $r_0>0$.
For the other endpoint, the formula for $H_\lambda'$ in Step~1 gives
\[
 g_\lambda'(z^q)=\frac{H_\lambda'(z)}{qz^{q-1}}
 \longrightarrow\frac{\lambda(m+p)}q=\lambda
 \qquad(z\to\infty).
\]
Since $g_\lambda'$ is strictly increasing, $g_\lambda'(r_0)<\lambda$.
Using $r_0=\vartheta^q$ and $m+p=q$, we therefore obtain
\[
 R_{\vartheta,\lambda}(0,1)
 =\vartheta^p(\lambda\vartheta^m)-B_{\vartheta,\lambda\vartheta^m}
 =r_0\bigl(\lambda-g_\lambda'(r_0)\bigr)>0.
\]

It remains to examine $U=V$, where the numerator and denominator
both vanish. For $U,V>0$ set $x=U^q/V^q$ and
\[
 f(x)=\Phi_\lambda(x,r_0)=xg_\lambda(r_0/x).
\]
The coefficient identities from Step~2 give
\[
 f(1)=A_{\vartheta,\lambda\vartheta^m}
       +B_{\vartheta,\lambda\vartheta^m},
 \qquad f'(1)=A_{\vartheta,\lambda\vartheta^m},
 \qquad f''(x)=\frac{r_0^2}{x^3}g_\lambda''(r_0/x)
\]
for $x$ near one. Cancelling the common factor $V^q$ in the
numerator and denominator of $R$ now gives
\[
 R_{\vartheta,\lambda}(U,V)
 =\frac{f(x)-f(1)-f'(1)(x-1)}{(\sqrt{x}-1)^2}.
\]
Taylor's formula with integral remainder makes the extension explicit:
\[
 \frac{f(x)-f(1)-f'(1)(x-1)}{(\sqrt{x}-1)^2}
 =(\sqrt{x}+1)^2\int_0^1(1-s)f''(1+s(x-1))\,\dd s
 \longrightarrow 2r_0^2g_\lambda''(r_0)>0
 \quad(x\to1).
\]
This is the positive diagonal limit in \eqref{eq:diagonalratio}.
Moreover, $r_0$ ranges in the compact interval
$[\vartheta_-^q,\vartheta_+^q]\Subset(0,1)$.
Thus, for $x$ in a common neighborhood of one, $r_0/x$ stays
in a compact subinterval of $(0,1)$, where the formula from
Step~1 shows that $g_\lambda''$ is jointly continuous in its
argument and $\lambda$. The integral formula therefore also
proves joint continuity of the extended ratio in $(U,V)$,
$\vartheta$, and $\lambda$ at the diagonal.

Away from the diagonal, joint continuity follows directly from
the formula defining $R$. We have proved that the extended ratio
is continuous and strictly positive everywhere on
$\mathcal N\times I\times[1,B_*]$. Hence
\[
 c_*:=\min_{\substack{(U,V)\in\mathcal N,\ \vartheta\in I\\
                      1\le\lambda\le B_*}}
         R_{\vartheta,\lambda}(U,V)>0.
\]
Multiplying by $(U^{q/2}-V^{q/2})^2$ and restoring the scale
gives the asserted inequality for all $U,V\ge0$, with the same
$c_*$ throughout the stated parameter range.
\end{proof}

Set
\[
 a_0=\frac Dq,\qquad a_1=\frac12\left(\frac Dq+\frac Dp\right).
\]
Then $a_0<a_1<D/p<\lambda_-/p$, so these parameters are admissible.
For $a\in[a_0,a_1]$, the paired coefficients satisfy
\[
 \vartheta=e^{-at},\qquad b=e^{(pa-D)t},\qquad
 \lambda=\frac b{\vartheta^m}
          =e^{(qa-D)t}=e^{q(a-a_0)t}.
\]
For $1\le t\le2$, the first parameter lies in
$[e^{-2a_1},e^{-a_0}]\Subset(0,1)$.
Lemma~\ref{lem:perspective} supplies a common coefficient
$\kappa_0>0$ on this interval times $[1,B_*]$.
The kernel bound $K_{\beta,t}\ge k_*$ on $[1,2]$ then gives
a positive strip contribution
$\kappa\eref[f]$, where $\kappa=Sk_*\kappa_0>0$,
whenever $1\le\lambda\le B_*$ there.

For $t\ge1$ and $a\in[a_0,a_1]$, one has
$\vartheta\le e^{-a_0}<1$ and $0<b\le1$.
Lemma~\ref{lem:continuity} therefore gives a common $C_*>0$ such that
\[
 [G_{\vartheta,b}(U,V)]_-
 \le C_*(U^{q/2}-V^{q/2})^2.
\]
Factoring $e^t$ from the distance in the kernel gives
\[
 K_{\beta,t}(\omega,\sigma)
 =\frac{e^{-\lambda_+t}}{|\sigma-e^{-t}\omega|^{d+sp}}
 \le C_K e^{-\lambda_+t},\qquad
 C_K=(1-e^{-1})^{-(d+sp)}\quad(t\ge1).
\]
The integral of $(U^{q/2}-V^{q/2})^2$ in the two angles and
$\tau$ is $SP_f(t)$. Hence Lemma~\ref{lem:chain} bounds the
negative contribution for $t>T\ge2$ by
\[
 32C_*SC_K\left(\int_T^\infty(1+t)^2e^{-\lambda_+t}\dd t\right)\eref[f].
\]
Twice integrating by parts gives, for every $\mu>0$,
\[
 \int_T^\infty(1+t)^2e^{-\mu t}\dd t
 =e^{-\mu T}
  \left(\frac{(T+1)^2}{\mu}
       +\frac{2(T+1)}{\mu^2}+\frac2{\mu^3}\right).
\]
This expression tends to zero as $T\to\infty$. Choose $T\ge2$
such that
\begin{equation}\label{eq:high-T}
 32C_*\omega_{d-1}C_K e^{-\lambda_+ T}
 \left(\frac{(T+1)^2}{\lambda_+}+\frac{2(T+1)}{\lambda_+^2}+\frac{2}{\lambda_+^3}\right)
 \le\frac\kappa2.
\end{equation}
Now put
\begin{equation}\label{eq:high-delta}
 \delta_a=\min\left\{a_1-a_0,\frac{\log B_*}{qT}\right\}>0.
\end{equation}
For $a_0\le a\le a_0+\delta_a$ and $0<t\le T$,
$1\le\lambda=e^{q(a-a_0)t}\le e^{q\delta_aT}\le B_*$.
Thus all these displacements have nonnegative remainder, and
their strip $[1,2]$ contributes at least $\kappa\eref[f]$.
The possibly negative tail costs at most $\kappa\eref[f]/2$.
At common cutoffs $0<\epsilon<1$, $M\ge T$, this already gives
\[
 I_{a,\beta,\epsilon,M}[v]-J_{a,\beta,\epsilon,M}N[v]
 \ge\frac\kappa2\eref[f].
\]
Passing to the limits proves the same inequality without cutoffs.
The convexity implication after \eqref{eq:outerHolder} yields
\begin{equation}\label{eq:high-remainder}
 \E_{\alpha_H+qa,\beta}[u]
 \ge J_{a,\beta}^{q/p}\Pp_{\alpha_H+qa,\beta}[u]
 +\frac{q\kappa}{2p}J_{a,\beta}^{q/p-1}\eref[v^{q/2}].
\end{equation}
Lemma~\ref{lem:cutoff} supplies sharpness of the main constant.
The starting weight is the high-exponent endpoint computed in
Subsection~\ref{sec:low-new}. Thus \eqref{eq:main-high-alpha} holds
with the direct width $\varepsilon=q\delta_a$, where $\delta_a$ is
specified by \eqref{eq:high-delta}.

\Needspace{14\baselineskip}
\subsection{Quantitative remainders and near-extremizers}\label{sec:proof-consequences}

The sharpness proofs also yield positive lower bounds for the
Hardy deficit. We now collect these estimates and derive their
consequences for equality cases and normalized near-extremizers.
We first give an explicit parameter interval for the low-exponent
remainder, before stating the common corollary.
Assume $D>0$ and $1<p<q<2$, and return to
$a_0=D/(p+n_{p,q})$ from Subsection~\ref{sec:low-new}.
To prepare for Corollary~\ref{cor:remainder}, choose
\begin{equation}\label{eq:low-width}
 n_1=\frac{n_{p,q}+\nsc}{2},\qquad
 \delta_a=\frac{D}{p+n_1}-\frac{D}{p+n_{p,q}}>0.
\end{equation}
Here $\nsc<n_1<n_{p,q}$. For
$a_0\le a\le a_0+\delta_a=D/(p+n_1)$, the exponent
$D/a-p\ge n_1>\nsc$, and therefore
\[
 b_a(t)\le\vartheta_a(t)^{n_1}
          <\vartheta_a(t)^{\nsc}
          \le C_{p,q}(\vartheta_a(t))\qquad(t>0).
\]
On the compact set of pairs $(a,t)$ in this interval times
$[1,2]$, Proposition~\ref{prop:low-exact} gives a common lower
bound $c_{a,\beta}(t)\ge\gamma>0$.
The fixed-interval kernel estimate in
Subsection~\ref{sec:proof-reduction} controls the reference energy.
The other displacements have nonnegative remainder, so
\begin{equation}\label{eq:low-remainder}
 I_{a,\beta}[v]-J_{a,\beta}N[v]\ge c\eref[v^{q/2}],
 \qquad c=\gamma Sk_*>0.
\end{equation}
The convexity implication proved after \eqref{eq:outerHolder} now gives
\[
 \E_{\alpha_H+qa,\beta}[u]-J_{a,\beta}^{q/p}\Pp_{\alpha_H+qa,\beta}[u]
 \ge\frac qp\,cJ_{a,\beta}^{q/p-1}\eref[v^{q/2}].
\]
The width in $\alpha$ is $q\delta_a$. This quantitative argument
uses a strict subinterval and does not require a positive scalar
coefficient at the optimized endpoint itself.

We now collect the consequences of the positive reference-energy
remainders obtained above. The following corollary shows that the
Hardy deficit controls a quadratic difference energy of the transformed
amplitude. This control rules out nonzero equality cases in the
specified test class and forces angular fluctuations to vanish along
normalized near-extremizing sequences. The parameters are fixed in
the neighborhoods constructed above, and the positive remainder
constant may depend on that choice.

\begin{corollary}[A quantitative remainder]\label{cor:remainder}
For fixed parameters in the strict-extension neighborhoods constructed
above, with $a\ne0$, there is $c'>0$ such that, for
$v(\tau,\omega)=e^{a\tau}|u(e^\tau\omega)|$,
\begin{equation}\label{eq:Hardy-remainder}
 \E_{\alpha,\beta}[u]-J_{a,\beta}^\rho\Pp_{\alpha,\beta}[u]
 \ge c'\eref[v^{q/2}].
\end{equation}
The sharp Hardy inequality has no nonzero equality case in the
specified test class. If $\Pp_{\alpha,\beta}[u_j]=1$ and
$\E_{\alpha,\beta}[u_j]\to J_{a,\beta}^\rho$, then
\begin{equation}\label{eq:near-extremizer}
 \eref[v_j^{q/2}]\to0,\qquad
 \|v_j^{q/2}-\overline{v_j^{q/2}}\|_{L^2(\Cyl)}\to0.
\end{equation}
The overline denotes spherical averaging only.
\end{corollary}
\begin{proof}
For each of the fixed parameter choices, the preceding arguments
give $I_{a,\beta}[v]\ge J_{a,\beta}N[v]+c\eref[v^{q/2}]$ with
$c>0$. This is \eqref{eq:globalremainder},
\eqref{eq:low-remainder}, or the corresponding bound in the direct
high-exponent, reflected, or balanced argument.
In the balanced case $c$ may depend on the fixed nonzero $a$.
The H\"older estimate therefore gives
\[
 \E_{\alpha,\beta}[u]\ge N[v]
 \left(J_{a,\beta}+c\frac{\eref[v^{q/2}]}{N[v]}\right)^\rho.
\]
Since $J_{a,\beta}>0$, the inequality
$(J+x)^\rho\ge J^\rho+\rho J^{\rho-1}x$ for $x\ge0$
proves \eqref{eq:Hardy-remainder} with
$c'=\rho cJ_{a,\beta}^{\rho-1}>0$.
For a nonzero test, $f=v^{q/2}$ is a nonzero $L^2$ function;
Lemma~\ref{lem:reference} gives $\eref[f]>0$.
The sharp Hardy deficit is consequently positive.

For the normalized sequence, the same inequality gives
\[
 0\le\eref[v_j^{q/2}]
 \le\frac{\E_{\alpha,\beta}[u_j]-J_{a,\beta}^\rho}{c'}
 \longrightarrow0.
\]
Finally \eqref{eq:Eref-decomp}, applied to $f_j=v_j^{q/2}$,
gives
\[
 \|f_j-\bar f_j\|_{L^2(\Cyl)}^2
 \le\tfrac12\eref[f_j]\longrightarrow0.
\]
This proves both limits in \eqref{eq:near-extremizer}.
\end{proof}

The conclusions of Corollary~\ref{cor:remainder} describe how the sharp
constant can be approached despite the absence of nonzero equality
cases in the test class. For $f_j=v_j^{q/2}$, the limit
$\|f_j-\bar f_j\|_{L^2(\Cyl)}\to0$ means that the transformed amplitudes
approach the subspace of functions independent of the spherical
variable. The spherical averages $\bar f_j$ still depend on $\tau$.
Moreover, \eqref{eq:Eref-decomp} shows that the average squared
increments of $\bar f_j$ over logarithmic shifts $1\le t\le2$ tend to
zero. Thus near equality forces small angular fluctuations and small
averaged changes on these logarithmic scales. These conditions allow
the sequence to spread in $\tau$, as the slow profiles in
Section~\ref{sec:reference} illustrate, so they do not imply convergence
to a fixed nonzero extremal profile.

\section{Endpoint loss and first sharpness thresholds}\label{sec:endpoint-loss}

\subsection{\texorpdfstring{Proof of the endpoint estimates in Theorem~\ref{thm:loss}}{Proof of the endpoint estimates}}

We prove the endpoint estimates \eqref{eq:endpoint-upper} and
\eqref{eq:Jendpoint}, and hence the endpoint loss of sharpness
asserted in Theorem~\ref{thm:loss}.

Fix $|e|=1$, $0\ne\psi\in C_c^\infty(B_{1/4})$, and let
$u_\ell(x)=\psi((x-e)/\ell)$, $0<\ell<1/4$.
For $\alpha$ near either outer endpoint and fixed admissible $\beta$,
$\Pp_{\alpha,\beta}[u_\ell]\asymp\ell^d$.

Define the contribution from the outer region $B(e,1/2)$ by
\[
 \E_{\alpha,\beta,\mathrm{near}}[u_\ell]
 :=\int_{B(e,1/2)}|x|^\alpha
 \left(\int_{\R^d}
 \frac{|u_\ell(x)-u_\ell(y)|^p|y|^\beta}{|x-y|^{d+sp}}
 \,\dd y\right)^{q/p}\dd x.
\]
Here the subscript $\mathrm{near}$ restricts only the outer
$x$-integration; the inner $y$-integration is still over all of
$\R^d$. For $x\in B(e,1/2)$, write the inner integral as
\begin{align*}
 A_{u_\ell,\beta}(x)
 &=I_{\ell,1}(x)+I_{\ell,2}(x),\\
 I_{\ell,1}(x)
 &:=\int_{B(e,3/4)}
    \frac{|u_\ell(x)-u_\ell(y)|^p|y|^\beta}{|x-y|^{d+sp}}\dd y,\\
 I_{\ell,2}(x)
 &:=\int_{\R^d\setminus B(e,3/4)}
    \frac{|u_\ell(x)-u_\ell(y)|^p|y|^\beta}{|x-y|^{d+sp}}\dd y.
\end{align*}
On the first integration region, $|e|=1$ gives
\[
 \frac14\le|y|\le\frac74
 \quad\Longrightarrow\quad
 0<c_\beta\le|y|^\beta\le C_\beta<\infty
 \qquad(y\in B(e,3/4)).
\]
Define the unweighted inner energy density by
\[
 A_{\psi,0}(X)
 :=\int_{\R^d}\frac{|\psi(X)-\psi(Y)|^p}{|X-Y|^{d+sp}}\dd Y.
\]
The subscript $0$ means that the inner weight parameter is zero.
Set $X=(x-e)/\ell$ and make the change of variables
\[
 y=e+\ell Y,\qquad \dd y=\ell^d\dd Y,\qquad
 |x-y|=\ell|X-Y|.
\]
Then the first term satisfies
\begin{align*}
 I_{\ell,1}(x)
 &=\ell^{-sp}\int_{|Y|<3/(4\ell)}
   \frac{|\psi(X)-\psi(Y)|^p|e+\ell Y|^\beta}{|X-Y|^{d+sp}}\dd Y\\
 &\le C_\beta\ell^{-sp}\int_{|Y|<3/(4\ell)}
   \frac{|\psi(X)-\psi(Y)|^p}{|X-Y|^{d+sp}}\dd Y\\
 &\le C_\beta\ell^{-sp}A_{\psi,0}(X)
  =C_\beta\ell^{-sp}A_{\psi,0}\!\left(\frac{x-e}{\ell}\right).
\end{align*}

For the second term, $\operatorname{supp}u_\ell\subset B(e,\ell/4)
\subset B(e,3/4)$, so $u_\ell(y)=0$ on its integration region. Hence
\[
 I_{\ell,2}(x)
 =|u_\ell(x)|^p\int_{\R^d\setminus B(e,3/4)}
        |y|^\beta|x-y|^{-d-sp}\dd y
 \le C|u_\ell(x)|^p.
\]
The constant is uniform for $x\in B(e,1/2)$: on these two regions
$|x-y|\ge1/4$, and $-d<\beta<sp$ ensures integrability at zero
and infinity. Combining the two bounds yields
\[
 A_{u_\ell,\beta}(x)
 \le C\ell^{-sp}A_{\psi,0}\!\left(\frac{x-e}{\ell}\right)
       +C|u_\ell(x)|^p
 \qquad(x\in B(e,1/2)).
\]
The outer weight $|x|^\alpha$ is uniformly bounded on $B(e,1/2)$
for the endpoint neighborhoods under consideration. Raising the
last bound to the power $q/p>1$ and integrating therefore gives
\begin{align*}
 \E_{\alpha,\beta,\mathrm{near}}[u_\ell]
 &\le C\ell^{-sq}\int_{B(e,1/2)}
       A_{\psi,0}\!\left(\frac{x-e}{\ell}\right)^{q/p}\dd x
       +C\int_{B(e,1/2)}|u_\ell(x)|^q\dd x\\
 &\le C\ell^{d-sq}\int_{\R^d}A_{\psi,0}(X)^{q/p}\dd X
       +C\ell^d\|\psi\|_{L^q(\R^d)}^q\\
 &\le C_{\psi,\beta}\ell^{d-sq}.
\end{align*}
The unweighted integral is finite because $s<1$ and
$q(d+sp)/p>d$. Outside $B(e,1/2)$, the support of $u_\ell$
has inner weight comparable to one, so
$A_{u_\ell,\beta}(x)\le C\ell^d|x-e|^{-d-sp}$.
For $\alpha=\alpha_+-\epsilon$ or $\alpha=\alpha_-+\epsilon$,
the resulting outer integral is $O(\ell^{dq/p}/\epsilon)$.
Indeed, near zero it reduces to
$\int_0^{1/2}r^{\alpha+d-1}\dd r$, and at infinity to
$\int_2^\infty r^{\alpha+d-1-(d+sp)q/p}\dd r$; only the
corresponding endpoint contributes $1/\epsilon$.
Consequently
\begin{equation}\label{eq:concentration}
 H_{\alpha,\beta}^{p,q}\le C\left(\ell^{-sq}
 +\epsilon^{-1}\ell^{d(q/p-1)}\right).
\end{equation}
Taking $\ell=\epsilon^{1/\Lambda}$, with
$\Lambda=sq+d(q/p-1)$, proves \eqref{eq:endpoint-upper}.

At $\alpha=\alpha_+-\epsilon$, put $S=\omega_{d-1}$ and
\[
 a_*:=\frac{d+\beta}{p}>0,\qquad
 a=a_*-\frac\epsilon q,\qquad
 \eta:=d+\beta-pa=\frac{p\epsilon}{q}.
\]
For sufficiently small $\epsilon>0$, $a\in[a_*/2,a_*]$.
Writing $z=r\omega$ for $0<r\le1/2$ and using
$|1-r^{-a}|^p=r^{-pa}(1-r^a)^p$, define
\[
 h_a(r):=(1-r^a)^p\int_{\Sph^{d-1}}
              |e-r\omega|^{-d-sp}\dd\omega
       =S+O(r^\delta),\qquad
 \delta:=\min\{1,a_*/2\}>0.
\]
The error is uniform for $0<r\le1/2$ and small $\epsilon$, since
$(1-r^a)^p=1+O(r^{a_*/2})$ and
$|e-r\omega|^{-d-sp}=1+O(r)$ uniformly in $\omega$.
Outside $B(0,1/2)$, the integral is uniformly bounded:
the numerator is $O(|z-e|^p)$ near $e$, and $\beta<sp$
controls the tail. Therefore
\[
\begin{aligned}
 J_{a,\beta}
 &=\int_0^{1/2}r^{\eta-1}h_a(r)\dd r+O(1)\\
 &=\frac{S2^{-\eta}}\eta
   +O\!\left(\int_0^{1/2}r^{\eta+\delta-1}\dd r\right)+O(1)
  =\frac S\eta+O(1).
\end{aligned}
\]
Here the error integral is at most $1/\delta$ and
$2^{-\eta}=1+O(\eta)$. In particular,
\[
 \eta J_{a,\beta}\longrightarrow S,\qquad
 J_{a,\beta}\sim\frac{\omega_{d-1}}{d+\beta-pa}
              =\frac{q\omega_{d-1}}{p\epsilon}
 \qquad(\epsilon\downarrow0).
\]
At $\alpha=\alpha_-+\epsilon$, instead set
$a=(\beta-sp)/p+\epsilon/q$ and
$\eta_-:=sp-\beta+pa=p\epsilon/q$.
Now $a$ stays bounded away from zero on the negative side.
The same polar calculation at infinity gives
\[
 J_{a,\beta}
 =S\int_2^\infty r^{-1-\eta_-}\dd r+O(1)
 =\frac{S2^{-\eta_-}}{\eta_-}+O(1)
 \sim\frac{q\omega_{d-1}}{p\epsilon}
 \qquad(\epsilon\downarrow0).
\]
The error is again uniform: the power and angular corrections
decay as $O(r^{-\delta_-})$ for some fixed $\delta_->0$,
while $\beta>-d$ and cancellation near $e$ control the bounded region.
Since
\[
 \theta=\frac{sq}{sq+d(q/p-1)}<1<q/p,
\]
we obtain
\begin{equation}\label{eq:ratiozero}
 H_{\alpha,\beta}^{p,q}/J_{a,\beta}^{q/p}\longrightarrow0
\end{equation}
at either outer endpoint. This proves the endpoint assertions of
Theorem~\ref{thm:loss}.

\begin{remark}[Dependence on the inner weight]
The exponent in \eqref{eq:endpoint-upper} is independent of $\beta$;
its multiplicative constant may depend on $\beta$.
\end{remark}

\subsection{\texorpdfstring{Proof of the threshold assertions in Theorem~\ref{thm:loss}}{Proof of the threshold assertions}}

We now prove the threshold assertions of Theorem~\ref{thm:loss},
starting with the relative closedness of the sharpness set.

For each fixed $\beta$, let
\begin{equation}\label{eq:sharp-set}
 \mathcal S=\{a\in(-\lambda_+/p,\lambda_-/p):H_{\alpha_H+qa,\beta}^{p,q}=J_{a,\beta}^\rho\}.
\end{equation}

For each fixed $0\ne u\in\test$, the ratio
 $\E_{\alpha_H+qa,\beta}[u]/\Pp_{\alpha_H+qa,\beta}[u]$
is continuous in $a$ on $(-\lambda_+/p,\lambda_-/p)$.
To check this, restrict $a$ to a compact subinterval and let
$\alpha_L,\alpha_U$ be its endpoint weights. The energy integrand
is dominated by
$(|x|^{\alpha_L}+|x|^{\alpha_U})A_{u,\beta}(x)^{q/p}$, which is
integrable by the strict finite-energy range. The denominator
is continuous by compact support away from zero and stays positive.
Dominated convergence thus proves quotient continuity.
For $J_{a,\beta}$, split the integral into neighborhoods of zero and
infinity and a neighborhood of $e$. On the first two use the
extreme powers of the compact $a$-interval; near $e$ the numerator
is uniformly bounded by $C|z-e|^p$. These are integrable
majorants, proving continuity of $J_{a,\beta}$ as well.
By \eqref{eq:Jupper}, membership in $\mathcal S$ is equivalent
to the family of inequalities
\[
 \frac{\E_{\alpha_H+qa,\beta}[u]}{\Pp_{\alpha_H+qa,\beta}[u]}\ge J_{a,\beta}^\rho
 \quad\text{for every }0\ne u\in\test.
\]

For every test this inequality defines a relatively closed set,
and $\mathcal S$ is the intersection of all such sets.
Thus $\mathcal S$ is relatively closed and its complement is
relatively open. This argument does not require continuity of
the best constant itself, which is an infimum over tests.

For $D>0$, define
\begin{equation}\label{eq:first-threshold}
 a_{\sh}=\sup\{b\in[0,\lambda_-/p):[0,b]\subset\mathcal S\}.
\end{equation}
The common strict-extension theorem gives $a_{\sh}>D/(p+n_{p,q})$, while \eqref{eq:ratiozero} gives $a_{\sh}<\lambda_-/p$. 
For each $0\le a<a_{\sh}$, the definition of the supremum supplies
some $b>a$ with $[0,b]\subset\mathcal S$, hence $a\in\mathcal S$.
Taking a sequence $a_j\uparrow a_{\sh}$ and using relative
closedness gives $a_{\sh}\in\mathcal S$. Therefore
$[0,a_{\sh}]\subset\mathcal S$.
For any $0<\delta<\lambda_-/p-a_{\sh}$, the whole interval
$(a_{\sh},a_{\sh}+\delta)$ cannot lie in $\mathcal S$:
otherwise $[0,a_{\sh}+\delta/2]\subset\mathcal S$, contradicting
\eqref{eq:first-threshold}. Hence there exists
$a\in(a_{\sh},a_{\sh}+\delta)\setminus\mathcal S$.
The global upper bound $H_{\alpha_H+qa,\beta}^{p,q}\le J_{a,\beta}^{\rho}$
then gives the strict-loss parameter asserted in
Theorem~\ref{thm:loss}.

Equivalently, $[0,a_{\sh}]$ is the connected component containing
zero of $\mathcal S\cap[0,\lambda_-/p)$, and
\[
 a_{\sh}=\inf\bigl((0,\lambda_-/p)\setminus\mathcal S\bigr).
\]
In the low-exponent range, the additional result gives the stronger bound
\[
 a_{\sh}\ge\frac{D}{p+\nsc}>\frac{D}{p+n_{p,q}}.
\]

For $D<0$, the same definition in the oriented variable $b=-a$ gives a threshold strictly between $|D|/(p+n_{p,q})$ and $\lambda_+/p$. For $D=0$, the small two-sided interval from Theorem~\ref{thm:extension} and endpoint loss give positive first thresholds on both sides of zero.

\begin{remark}[Threshold comparisons]\label{rem:threshold-comparison}
For $D>0$, the endpoint $D/(p+n_{p,q})$ belongs to an elementary
sufficient condition, whereas $a_{\sh}$ is determined by optimality
over all admissible tests. Theorem~\ref{thm:extension} gives
$D/(p+n_{p,q})<a_{\sh}$. In the low-exponent range, the optimized
scalar endpoint $D/(p+\nsc)$ gives the stronger, non-strict lower
bound displayed above. The angular calculation in
Section~\ref{sec:explicit} is a local perturbation test for $p=2$,
$q=4$: a negative second-order coefficient at an admissible $a>0$
produces a strict-loss test and hence $a_{\sh}<a$. The exact relation
between the onset of this instability and $a_{\sh}$ remains open.

The definition of $a_{\sh}$ concerns the component containing zero.
Relative closedness alone permits later components; for example,
a nonnegative sharpness set could have the form
\[
 [0,a_{\sh}]\cup[b,c],
 \qquad a_{\sh}<b\le c<(d+\beta)/p.
\]
Here $a_{\sh}$ is the initial endpoint and $c$ is the largest sharp
parameter. Excluding a return to sharpness requires further information.
\end{remark}

\section{Examples and angular loss of sharpness}\label{sec:explicit}

\subsection{A low-exponent example of the scalar criterion}\label{sec:low-example}

We illustrate the optimized scalar exponent $\nsc$ of
Proposition~\ref{prop:nsc} and its use in
Theorem~\ref{thm:main}\textup{(iii)}. For $p=3/2$ and $q=7/4$ the
elementary exponent in \eqref{eq:n-pq} is $n_{p,q}=2$, whereas
\[
 C_{p,q}(t)=\frac76\bigl((1-t)^{-1/2}-1\bigr)+\frac t6.
\]
The proof only needs $1<\nsc<2$; a one-dimensional numerical
evaluation of \eqref{eq:nsc-main} gives $\nsc\approx1.10137$, at
$t_{\mathrm{max}}\approx0.14657$.

Take $d=2$, $s=1/2$, and $-5/8<\beta<3/4$, so that
$-d<\beta<sp$ and $D>0$; here $D=\tfrac54+2\beta$ and
$\alpha_H=-\tfrac98-\tfrac{7\beta}{6}$. Theorem~\ref{thm:main}(iii)
enlarges the elementary interval
\[
 -\frac98-\frac{7\beta}{6}\le\alpha\le-\frac12-\frac\beta6
\]
to
\begin{equation}\label{eq:low-example-alpha}
 -\frac98-\frac{7\beta}{6}\le\alpha\le
 -\frac98-\frac{7\beta}{6}
 +\frac{\tfrac74(\tfrac54+2\beta)}{3/2+\mathfrak n_{3/2,7/4}}.
\end{equation}
Throughout both intervals the best constant is
$J_{2,1/2,3/2}^{(\beta)}(a)^{7/6}$ with
$a=4\alpha/7+9/14+2\beta/3$. At $\beta=0$ the elementary interval is
$[-9/8,-1/2]$, and the improved one keeps its left endpoint with
right endpoint $-9/8+35/[16(3/2+\nsc)]>-1/2$.

\subsection{An angular second variation for \texorpdfstring{$p=2,q=4$}{p=2, q=4}}

We derive the angular criterion \eqref{eq:secondvariation} for loss
of homogeneous sharpness and justify it on the original smooth test
class. Keep $p=2,q=4$, $-d<\beta<2s$, and
$(\beta-2s)/2<a<(d+\beta)/2$. Using the unweighted coefficient
$\mathfrak c_n$ from Appendix~\ref{app:evaluations}, put
\[
 \mathfrak c_{n,\beta}(r)
 =\mathfrak c_n(r-\beta)-\mathfrak c_0(-\beta).
\]
For the values $r=a,2a$ used below, the exponents lie in
the interval $-2s<\gamma<d$ in \eqref{eq:powerL}. We use a
weighted angular test independent of the logarithmic coordinate.
Let $Y_n$ be a real spherical harmonic with
$\langle Y_n\rangle=0$ and $\langle Y_n^2\rangle=1$, and set
\[
 v_\varepsilon(\tau,\omega)=1+\varepsilon Y_n(\omega).
\]
Since $\widehat x=\omega$ under the identification $x=e^\tau\omega$
of \eqref{eq:cylinderchange}, the same angular factor is written
$Y_n(\omega)$ on the cylinder and $Y_n(\widehat x)$ below; for small
$\varepsilon$ the multiplier is positive. As $v_\varepsilon$ does not
depend on $\tau$, both formal cylinder integrals carry the divergent
factor $\int_\R\dd\tau$, so we work with the quotient of the spherical
integrals, which the logarithmic cutoffs constructed after
\eqref{eq:secondvariation} justify on the stated test class. Let
\begin{align}
 B_{a,n,\beta}&=\mathfrak c_{0,\beta}(a)+\mathfrak c_{n,\beta}(a)-\mathfrak c_{n,\beta}(2a),\label{eq:Bn}\\
 C_{a,n,\beta}&=2\mathfrak c_{n,\beta}(a)-\mathfrak c_{0,\beta}(2a).\label{eq:Cn}
\end{align}

Let $L$ denote the unweighted operator in \eqref{eq:powerL}
and define $L_\beta f=L(|x|^\beta f)-fL(|x|^\beta)$.
Thus $L_\beta f(x)=\PV\int(f(x)-f(y))|y|^\beta
|x-y|^{-d-2s}\dd y$. Put
$\phi(x)=|x|^{-a}$, and $\psi(x)=|x|^{-a}Y_n(\widehat x)$.
The product identity, proved by expanding the integrand at a common
principal-value cutoff, is
\[
 \int\frac{(\phi(x)-\phi(y))(\psi(x)-\psi(y))|y|^\beta}{|x-y|^{d+2s}}\dd y
 =\phi L_\beta\psi+\psi L_\beta\phi-L_\beta(\phi\psi).
\]
On $|x|=1$, its right side is $B_{a,n,\beta}Y_n(\omega)$.
Similarly the quadratic energy density of $\psi$ is
$Z_{a,n,\beta}(\omega)=2\psi L_\beta\psi-L_\beta(\psi^2)$.
Since $Y_n^2$ has spherical average one, its degree-zero component
is one. Applying the weighted power formula to its finite harmonic
decomposition gives $\langle Z_{a,n,\beta}\rangle=C_{a,n,\beta}$.
Consequently
\[
 \Aa_{a,\beta}[v_\varepsilon](\omega)
 =J_{a,\beta}+2\varepsilon B_{a,n,\beta}Y_n(\omega)
      +\varepsilon^2 Z_{a,n,\beta}(\omega).
\]
Squaring, averaging, and using $\langle Y_n\rangle=0$ and
$\langle Y_n^2\rangle=1$ yields
\[
 \langle\Aa_{a,\beta}[v_\varepsilon]^2\rangle
 =J_{a,\beta}^2+(4B_{a,n,\beta}^2+2J_{a,\beta}C_{a,n,\beta})\varepsilon^2
       +O(\varepsilon^3).
\]
The denominator is
$\langle v_\varepsilon^4\rangle=1+6\varepsilon^2+O(\varepsilon^3)$.
Dividing these expansions gives
\begin{equation}\label{eq:secondvariation}
 \frac{\langle\Aa_{a,\beta}[v_\varepsilon]^2\rangle}
 {\langle v_\varepsilon^4\rangle}
 =J_{a,\beta}^2+2\bigl(J_{a,\beta}C_{a,n,\beta}+2B_{a,n,\beta}^2-3J_{a,\beta}^2\bigr)\varepsilon^2
 +o(\varepsilon^2).
\end{equation}
For $n=1$, parity makes this expansion even, with remainder $O(\varepsilon^4)$.

If the coefficient in parentheses is negative, the homogeneous candidate is not sharp. Indeed,
\[
 u_{T,\varepsilon}(x)=|x|^{-a}\zeta\left(\frac{\log|x|}{T}\right)
 (1+\varepsilon Y_n(\widehat x))
\]
belongs to the original test class. 
To justify the limit, fix $h(\omega)=1+\varepsilon Y_n(\omega)$
and decompose the cylindrical difference as
\begin{align*}
 &\zeta_T(\tau)h(\omega)-e^{-at}\zeta_T(\tau+t)h(\sigma)\\
 &\quad=\zeta_T(\tau)\bigl(h(\omega)-e^{-at}h(\sigma)\bigr)
       +e^{-at}h(\sigma)\bigl(\zeta_T(\tau)-\zeta_T(\tau+t)\bigr).
\end{align*}
The second summand has mixed-norm bound
$C\|h\|_\infty\|\zeta_T'\|_4$, by the finite moment in
Lemma~\ref{lem:cutoff}. The first has energy root
$\|\zeta_T\|_4(\int\Aa_{a,\beta}[h]^2\dd\omega)^{1/4}$.
The denominator is exactly
$\|\zeta_T\|_4^4\int h^4\dd\omega$.
Since $\|\zeta_T'\|_4/\|\zeta_T\|_4=O(T^{-1})$, the norm
triangle inequality proves the asserted quotient limit.
Choose a nonzero $\varepsilon$ so small that the angular quotient
is strictly below $J_{a,\beta}^2$, and then choose a finite $T$ large
enough to preserve half that strict gap. This produces an actual
test in the original class, rather than an infinite-cylinder ansatz.

\medskip
\noindent\textit{The model $(d,s,p,q)=(3,1/2,2,4)$.} For $\beta=0$,
\eqref{eq:secondvariation} then produces an explicit strict-loss
example. Here

\begin{equation}\label{eq:f0f1}
 \mathfrak c_0(\gamma)=\pi^2 f_0(\gamma),\quad
 f_0(\gamma)=(1-\gamma)\tan\frac{\pi\gamma}{2},\qquad
 \mathfrak c_1(\gamma)=\pi^2 f_1(\gamma),\quad
 f_1(\gamma)=\frac{\gamma(2-\gamma)}{1-\gamma}\cot\frac{\pi\gamma}{2}.
\end{equation}
The values at removable singularities are taken by continuity. Write
\begin{align}
 j(a)&=2f_0(a)-f_0(2a),\nonumber\\
 b(a)&=f_0(a)+f_1(a)-f_1(2a),\nonumber\\
 c(a)&=2f_1(a)-f_0(2a),\label{eq:jbc}\\
 \Sigma_1(a)&=j(a)c(a)+2b(a)^2-3j(a)^2.\label{eq:Sigma}
\end{align}
Then the second-order coefficient in \eqref{eq:secondvariation} is $2\pi^4\Sigma_1(a)$.

At $a=4/3$, direct trigonometric evaluation gives
\[
 j=\frac{7\sqrt3}{3},\qquad b=\frac{13\sqrt3}{15},\qquad
 c=\frac{31\sqrt3}{9},\qquad
 \Sigma_1(4/3)=-\frac{4586}{225}<0.
\]
By \eqref{eq:secondvariation} and the logarithmic cutoff
construction above, this gives $H_{13/3,0}^{2,4}<J_{4/3,0}^2$.
Thus the homogeneous candidate is not optimal in this example.

\section{Hardy inequalities with singularities along flat subspaces}\label{sec:flat}

The equal-exponent flat problem and its sharp constants have been studied
in \cite{Mallick,Sahu,KS}. We prove mixed-exponent sharpness intervals
for $p<q$ and a sharp inequality throughout the ambient finite-energy
range for $p>q$. The latter flat geometry is not covered by
\Kij\ (Theorems~1 and~2), whose mixed-exponent results concern
the whole-space point singularity and the regional half-space.

The singularity at a point and the singularity along a flat subspace have
different dimensional parameters. The codimension determines the homogeneous
normal power, whereas the ambient dimension enters the tangential integration
factor and the far-field condition for compactly supported tests. We prove
the flat analogue on the intervals obtained from the pointwise scalar criterion. An affine translation-chain estimate controls the tangential increments
and proves strict extension beyond the elementary scalar criterion.

\subsection{Finiteness of test energies and of the homogeneous integral}

Let $n\ge2$, $1\le k<n$, and $m=n-k$. Write $X=(y,z)\in\R^k\times\R^m$,
and set $\Sigma=\R^k\times\{0\}$. Throughout this section,
$0<s<1$, $1<p,q<\infty$, and $-m<\beta<sp$.
The interval and strict-extension assertions assume $p<q$;
Theorem~\ref{thm:flat-reverse} treats $p>q$.
The natural inner weight is the distance weight $|z'|^\beta$.
An ambient weight $|(y',z')|^\beta$ would describe a different
problem and is not used in this flat translation-invariant result. For
$u\in C_c^\infty(\R^n\setminus\Sigma)$, define
\begin{align}
 \E_{\Sigma,\alpha,\beta}[u]
 &=\int_{\R^k\times\R^m}|z|^\alpha
 \left(\int_{\R^k\times\R^m}
 \frac{|u(y,z)-u(y',z')|^p|z'|^\beta}
 {(|y-y'|^2+|z-z'|^2)^{(n+sp)/2}}\dd y'\dd z'
 \right)^{q/p}\dd y\dd z,\label{eq:flat-energy}\\
 \Pp_{\Sigma,\alpha,\beta}[u]
 &=\int_{\R^k\times\R^m}|u(y,z)|^q|z|^{\alpha+q\beta/p-sq}\dd y\dd z,
 \label{eq:flat-denominator}\\
 H_{\Sigma,\alpha,\beta}^{p,q}
 &=\inf_{\substack{0\ne u\in C_c^\infty(\R^n\setminus\Sigma)}}
 \frac{\E_{\Sigma,\alpha,\beta}[u]}{\Pp_{\Sigma,\alpha,\beta}[u]}.
 \label{eq:flat-best}
\end{align}
The subscript $\Sigma$ indicates the singular plane, and both power
weights remain explicit in the notation.
The kernel again has coefficient one. Put
\begin{equation}\label{eq:flat-parameters}
 a=\frac{\alpha+m+q\beta/p-sq}{q},\qquad D_m=m+2\beta-sp,\qquad
 \mathfrak g_{n,k,s,p}
 =\frac{\pi^{k/2}\Gamma((m+sp)/2)}{\Gamma((n+sp)/2)}.
\end{equation}

\begin{proposition}[Finiteness of test energies and of the homogeneous integral]
\label{prop:flat-finite}
Every nonzero test function has finite energy in \eqref{eq:flat-energy}
if and only if
\begin{equation}\label{eq:flat-ambient-range}
 -m<\alpha<sq+n\left(\frac qp-1\right).
\end{equation}
For a unit vector $e\in\R^m$, the homogeneous integral
\begin{equation}\label{eq:flat-J}
\begin{aligned}
 J_{\Sigma,a,\beta}
 &=\mathfrak g_{n,k,s,p}J_{m,s,p}^{(\beta)}(a)\\
 &=\mathfrak g_{n,k,s,p}\int_{\R^m}
 \frac{|1-|w|^{-a}|^p|w|^\beta}{|e-w|^{m+sp}}\dd w.
\end{aligned}
\end{equation}
is finite exactly when
\begin{equation}\label{eq:flat-profile-range}
 \frac{\beta-sp}{p}<a<\frac{m+\beta}{p},
 \qquad\text{equivalently}\qquad
 -m<\alpha<sq+m\left(\frac qp-1\right).
\end{equation}
It vanishes at $a=0$ and is positive at the other exponents in this interval.
\end{proposition}
\begin{proof}
For a fixed outer point with $u\ne0$, integration close to the
inner singular plane requires $\beta>-m$; tangential integration
followed by large normal radii requires $\beta<sp$.
Under these conditions, on bounded sets near $\Sigma$, the test vanishes and its inner difference
energy is bounded, continuous, and strictly positive at every point of
$\Sigma$. Normal integration is finite precisely when $\alpha>-m$.
At infinity, the compact-support expansion \eqref{eq:compacttail}, with $d=n$ and the compact inner moment
$\int|u(y',z')|^p|z'|^\beta\dd y'\dd z'$, is uniform in direction. The remaining tail is therefore comparable
to
\[
 \int_{|X|>R}|z|^\alpha |X|^{-q(n+sp)/p}\dd X.
\]
The angular integral is finite when $\alpha>-m$. The radial integral is
finite precisely when $\alpha+n<q(n+sp)/p$, giving
\eqref{eq:flat-ambient-range}. Smoothness and $s<1$ control the diagonal.
The denominator is finite and positive for every nonzero test and every
real $\alpha$.

The exact tangential integration is
\begin{equation}\label{eq:flat-gamma}
 \int_{\R^k}(|\xi|^2+R^2)^{-(n+sp)/2}\dd\xi
 =\mathfrak g_{n,k,s,p}R^{-m-sp},\qquad R>0.
\end{equation}
This follows from polar coordinates and the beta integral. The inner
weight is independent of the tangential increment. In the candidate,
integration near $w=0$ requires $pa<m+\beta$ when $a>0$, and integration at infinity
imposes $pa>\beta-sp$ when $a<0$; the standing conditions
$-m<\beta<sp$ control the other signs. The cancellation at $w=e$
uses $s<1$. This proves the finiteness range of the homogeneous integral and endpoint divergence.
\end{proof}

For $k>0$ and $q>p$, the upper endpoint of
\eqref{eq:flat-profile-range} is strictly smaller than that of
\eqref{eq:flat-ambient-range}. In particular, a finite-energy test problem
can coexist with an infinite homogeneous candidate. These two ranges
must not be identified. For $p>q$ their order is reversed: the ambient
range is the smaller one and is nonempty exactly when
$q(n+sp)/p>k$. All assertions about finite-energy tests below retain
that ambient restriction.

\subsection{The sharp inequality}

\begin{theorem}[Flat-subspace Hardy inequality]\label{thm:flat}
Suppose $m=n-k$, $0<s<1$, $1<p<q<\infty$, and
$-m<\beta<sp$ with $m+2\beta>sp$.
Set $a=(\alpha+m+q\beta/p-sq)/q$. With $J_{\Sigma,a,\beta}$ defined in
\eqref{eq:flat-J}, for every $u\in C_c^\infty(\R^n\setminus\Sigma)$,
\begin{equation}\label{eq:flat-main}
\begin{aligned}
 &\int_{\R^k\times\R^m}|z|^\alpha
 \left(\int_{\R^k\times\R^m}
 \frac{|u(y,z)-u(y',z')|^p|z'|^\beta}
 {(|y-y'|^2+|z-z'|^2)^{(n+sp)/2}}\dd y'\dd z'\right)^{q/p}
 \dd y\dd z\\
 &\quad\ge
 J_{\Sigma,a,\beta}^{q/p}
 \int_{\R^k\times\R^m}|u(y,z)|^q|z|^{\alpha+q\beta/p-sq}\dd y\dd z
\end{aligned}
\end{equation}
holds with the best possible constant on each of the following intervals.
In parts \textup{(i)} and \textup{(ii)}, the width $\varepsilon>0$
depends only on $n,k,s,p,q,\beta$.
\begin{enumerate}[label=\textup{(\roman*)}]
\item If $2\le q\le2p$,
\[
 sq-m-q\beta/p\le\alpha\le sq-m-q\beta/p+\frac{q(m+2\beta-sp)}{p+1}+\varepsilon.
\]
\item If $q>2p$,
\[
 sq-m-q\beta/p\le\alpha\le s(q-p)+(2-q/p)\beta+\varepsilon.
\]
\item If $1<p<q<2$, with $\nsc$ defined in \eqref{eq:nsc-main},
\[
 sq-m-q\beta/p\le\alpha\le sq-m-q\beta/p+\frac{q(m+2\beta-sp)}{p+\nsc}.
\]
\end{enumerate}
The best constant is zero at $\alpha=sq-m-q\beta/p$ and is positive at the other
parameters displayed above. The
widths are chosen inside \eqref{eq:flat-profile-range}; they are proved
in Proposition~\ref{prop:flat-extension}.
\end{theorem}

The theorem includes $m=1$ whenever $-1<\beta<sp$ and $1+2\beta>sp$. Then $\Sph^0$ has
counting measure on its two points, and $k_{\beta,1}$ is explicitly
\[
 k_{\beta,1}(t)=e^{(1+\beta)t}
 \bigl(|1-e^t|^{-1-sp}+|1+e^t|^{-1-sp}\bigr).
\]
It has the same short-scale order $|t|^{-1-sp}$ and tail
orders $e^{-(sp-\beta)t}$ and $e^{(1+\beta)t}$ as the general $k_\beta$.
The scalar support proof is dimension independent, and the
cutoff proof uses only these estimates. Thus neither requires
a positive-dimensional sphere.

\subsection{The flat cylinder and the lower bound}

We first prove the lower bound in Theorem~\ref{thm:flat} on the
intervals supplied directly by the scalar criterion. The additional
widths in parts~\textup{(i)} and~\textup{(ii)} will be proved in
Proposition~\ref{prop:flat-extension}.

Keep the tangential coordinate $y$ unscaled, and write
\begin{equation}\label{eq:flat-cylinder}
 u(y,e^\tau\omega)=e^{-a\tau}v(y,\tau,\omega),\qquad
 \mathcal K_{\beta,t}(\eta,\omega,\sigma)
 =\frac{e^{(m+\beta)t}}
 {(|\eta|^2+|\omega-e^t\sigma|^2)^{(n+sp)/2}}.
\end{equation}
Use measure $\dd y\dd\tau\dd\omega$ on
$\mathcal C_\Sigma=\R^k\times\R\times\Sph^{m-1}$.
Use $y'=y+e^\tau\eta$, $z'=e^{\tau+t}\sigma$, and
$|z'|^\beta=e^{\beta(\tau+t)}$. The inner scaling factor is
$e^{-(pa+sp-\beta)\tau}$; the outer one cancels because
$\alpha+m=q(a+s-\beta/p)$. This gives
\begin{equation}\label{eq:flat-cylinder-energy}
 \E_{\Sigma,\alpha,\beta}[u]=\int_{\mathcal C_\Sigma}
 \mathcal A_{\Sigma,a,\beta}[v]^{q/p},\qquad
 \Pp_{\Sigma,\alpha,\beta}[u]=\int_{\mathcal C_\Sigma}|v|^q.
\end{equation}
\begin{equation}\label{eq:flat-cylinder-A}
\begin{aligned}
 \mathcal A_{\Sigma,a,\beta}[v](y,\tau,\omega)
 ={}&\int_\R\int_{\R^k}\int_{\Sph^{m-1}}
 \mathcal K_{\beta,t}(\eta,\omega,\sigma)\\
 &\quad\times|v(y,\tau,\omega)-e^{-at}v(y+e^\tau\eta,\tau+t,\sigma)|^p
 \dd\sigma\dd\eta\dd t.
\end{aligned}
\end{equation}
The reversal identity is
\begin{equation}\label{eq:flat-reversal}
 \mathcal K_{\beta,-t}(-e^{-t}\eta,\sigma,\omega)e^{-kt}
 =e^{-(m+2\beta-sp)t}\mathcal K_{\beta,t}(\eta,\omega,\sigma).
\end{equation}
Here $e^{-kt}$ is the Jacobian of the reversed tangential increment;
the transformation $(y,\tau)\mapsto(y+e^\tau\eta,\tau+t)$ has
Jacobian one. Furthermore,
\begin{equation}\label{eq:flat-k}
 k_{\Sigma,\beta}(t):=\int_{\R^k}\int_{\Sph^{m-1}}
 \mathcal K_{\beta,t}(\eta,\omega,\sigma)\dd\sigma\dd\eta
 =\mathfrak g_{n,k,s,p}k_{\beta,m}(t),
\end{equation}
where $k_{\beta,m}$ is the quantity above computed with $d=m$.
It is independent of $\omega$, as in the point case, and integrating
in the other angular variable gives the same value.

\begin{proof}[Proof of the lower bound in Theorem~\ref{thm:flat}]
Replacing $u$ by $|u|$ can only decrease its difference energy. Thus
take $v\ge0$, set $N_\Sigma[v]=\int_{\mathcal C_\Sigma}v^q$ and
$I_{\Sigma,a,\beta}[v]=\int_{\mathcal C_\Sigma}
v^{q-p}\mathcal A_{\Sigma,a,\beta}[v]$. H\"older gives
\[
 I_{\Sigma,a,\beta}[v]\le
 \E_{\Sigma,\alpha,\beta}[u]^{p/q}N_\Sigma[v]^{1-p/q}.
\]
For $t>0$, put $\vartheta=e^{-at}$ and
$b=e^{-(m+2\beta-sp-pa)t}$. Reversing the negative displacements by
\eqref{eq:flat-reversal} yields the exact identity
\begin{equation}\label{eq:flat-pair}
\begin{aligned}
 I_{\Sigma,a,\beta}[v]-J_{\Sigma,a,\beta}N_\Sigma[v]
 ={}&\int_0^\infty\int_{\mathcal C_\Sigma}
 \int_{\R^k}\int_{\Sph^{m-1}}
 \mathcal K_{\beta,t}(\eta,\omega,\sigma)G_{\vartheta,b}(U,V)
 \dd\sigma\dd\eta\dd y\dd\tau\dd\omega\dd t,
\end{aligned}
\end{equation}
where $U=v(y,\tau,\omega)$,
$V=v(y+e^\tau\eta,\tau+t,\sigma)$, and $G$ is defined in
\eqref{eq:Gtwo}. Integrating either $U^q$ or $V^q$ against the kernel
gives $k_{\Sigma,\beta}(t)N_\Sigma[v]$. This proves the subtraction in
\eqref{eq:flat-pair} by \eqref{eq:supportsum}.
As in \eqref{eq:pairidentity}, the identity is first established for
$\epsilon<|t|<T$ and only then passed to the limit.

When $0<a\le D_m/(p+n_{p,q})$, one has
$b=\vartheta^{D_m/a-p}\le\vartheta^{n_{p,q}}$.
Lemma~\ref{lem:scalar-support} makes the paired remainder nonnegative.
For $1<p<q<2$, Propositions~\ref{prop:low-exact} and~\ref{prop:nsc}
replace this condition by $0<a\le D_m/(p+\nsc)$.
H\"older now proves \eqref{eq:flat-main} on the stated intervals without
the additional widths. Proposition~\ref{prop:flat-extension} proves
those widths in parts \textup{(i)} and \textup{(ii)}. At $a=0$ the lower bound
is simply nonnegativity. The matching upper bound, including this zero
case, is proved next.
\end{proof}

\subsection{\texorpdfstring{Proof of optimality in Theorem~\ref{thm:flat}}{Proof of optimality for flat singular sets}}

We construct a two-scale sequence whose energy quotients converge
to $J_{\Sigma,a,\beta}^{q/p}$, proving the optimality assertion in
Theorem~\ref{thm:flat}. The same construction also supplies the upper
bound for Theorem~\ref{thm:flat-reverse}.

Let $\E_{m,\alpha,\beta}$ and $\Pp_{m,\alpha,\beta}$ denote the
point-singularity functionals in the normal space $\R^m$,
defined by \eqref{eq:energy}--\eqref{eq:denominator} with
$d=m$ and kernel coefficient one. This notation also applies
when $m=1$, using the explicit formula for $k_{\beta,1}$ above.

\begin{lemma}[Tangential cutoff limit]\label{lem:flat-product}
Let $1<p,q<\infty$ and assume both
\eqref{eq:flat-profile-range} and \eqref{eq:flat-ambient-range}.
Fix nonzero
$\chi\in C_c^\infty(\R^k)$ and
$w\in C_c^\infty(\R^m\setminus\{0\})$. Then
\begin{equation}\label{eq:flat-product-limit}
 \lim_{R\to\infty}
 \frac{\E_{\Sigma,\alpha,\beta}[\chi(y/R)w(z)]}
 {\Pp_{\Sigma,\alpha,\beta}[\chi(y/R)w(z)]}
 =\mathfrak g_{n,k,s,p}^{q/p}
 \frac{\E_{m,\alpha,\beta}[w]}{\Pp_{m,\alpha,\beta}[w]}.
\end{equation}
\end{lemma}
\begin{proof}
Write $Y=y/R$ and $y'=y+\xi$. The product difference is
\[
 \chi(Y)(w(z)-w(z'))
 +w(z')\bigl(\chi(Y)-\chi(Y+\xi/R)\bigr).
\]
Use the mixed norm with outer measure $|z|^\alpha\dd Y\dd z$
and inner measure
\[
 |z'|^\beta(|\xi|^2+|z-z'|^2)^{-(n+sp)/2}\dd\xi\dd z'.
\]
The first summand has norm
\[
 \|\chi\|_{L^q(\R^k)}\,
 \mathfrak g_{n,k,s,p}^{1/p}\E_{m,\alpha,\beta}[w]^{1/q}.
\]
We prove that the second summand tends to zero without interchanging
the inner and outer norms. Put $\rho=q/p$ and denote its inner
$p$th-power integral by $B_R(Y,z)$. Choose $L\ge1$ with
$\supp\chi\subset B_L$.

On $|Y|\le2L$, the pointwise Lipschitz bound
$|\chi(Y)-\chi(Y+\xi/R)|\le C\min\{|\xi|/R,1\}$ gives
\[
 B_R(Y,z)\le F_R(z)^p,\qquad
 F_R(z)^p=C^p\int_{\R^m}|w(z')|^p|z'|^\beta
 \int_{\R^k}\frac{\min\{|\xi|/R,1\}^p}
 {(|\xi|^2+|z-z'|^2)^{(n+sp)/2}}\dd\xi\dd z'.
\]
For $R\ge1$, $F_R\le F_1$, and $F_R(z)\to0$ by dominated
convergence. The majorant is bounded on bounded normal sets:
near the diagonal use $|\xi|^p\le|(\xi,z-z')|^p$ and $s<1$;
away from it use kernel integrability and the compact support of $w$,
which stays away from the singular plane. Tangential integration also
gives $F_1(z)^p\le C_{w,\beta}|z|^{-m-sp}$ for large $|z|$.
The finiteness range of the homogeneous integral therefore implies
\[
 \int_{|Y|\le2L}\int_{\R^m}|z|^\alpha B_R(Y,z)^\rho\dd z\dd Y
 \le |B_{2L}|\int_{\R^m}|z|^\alpha F_R(z)^q\dd z\longrightarrow0.
\]

For $|Y|>2L$, one has $\chi(Y)=0$. The remaining factor is supported
where $|Y+\xi/R|\le L$, hence $|\xi|\ge R|Y|/2$ and the allowed
$\xi$-volume is at most $C R^k$. Since $w$ has compact support,
for all sufficiently large $R$ this gives the uniform estimate
\[
 B_R(Y,z)\le C_{w,\beta,\chi}R^k
       (R^2|Y|^2+|z|^2)^{-(n+sp)/2}.
\]
Integration in $z$, by $z=R|Y|Z$, and then in $Y$ yields
\begin{equation}\label{eq:flat-product-far-error}
\begin{aligned}
 &\int_{|Y|>2L}\int_{\R^m}|z|^\alpha B_R(Y,z)^\rho\dd z\dd Y\\
 &\quad\le C R^{\alpha+m-\rho(m+sp)}
 \int_{|Y|>2L}|Y|^{\alpha+m-\rho(n+sp)}\dd Y
 \longrightarrow0.
\end{aligned}
\end{equation}
Indeed, the normal integral is finite because $\alpha>-m$ and
$\alpha+m<\rho(n+sp)$; the last $Y$-integral is finite because
$\alpha+n<\rho(n+sp)$. The power of $R$ is negative because
$\alpha+m<\rho(m+sp)$. These are precisely the two upper restrictions
assumed in the lemma. Thus the error tends to zero for either order
of $p,q$.
The norm triangle inequality in both
directions proves convergence of $R^{-k/q}\E_{\Sigma,\alpha,\beta}^{1/q}$.
Finally,
\[
 \Pp_{\Sigma,\alpha,\beta}[\chi(y/R)w(z)]
 =R^k\|\chi\|_q^q\Pp_{m,\alpha,\beta}[w],
\]
which proves \eqref{eq:flat-product-limit}.
\end{proof}

Take a nonzero nonnegative $\zeta\in C_c^\infty(\R)$ and set
\begin{equation}\label{eq:flat-cutoff}
 u_{R,T}(y,z)=\chi(y/R)|z|^{-a}
 \zeta\left(\frac{\log|z|}{T}\right).
\end{equation}
For each fixed $T$, Lemma~\ref{lem:flat-product} removes the
tangential cutoff as $R\to\infty$. The normal logarithmic cutoff
argument in Lemma~\ref{lem:cutoff}, with $d=m$, then gives
the limit below for $p<q$. For $p>q$, use the direct exponential-tail
error estimate \eqref{eq:reverse-cutoff-error}, which requires no
mixed-norm interchange; it uses the same estimates for $k_{\beta,m}$.
Thus for either order
\begin{equation}\label{eq:flat-sharp-limit}
 \lim_{T\to\infty}\lim_{R\to\infty}
 \frac{\E_{\Sigma,\alpha,\beta}[u_{R,T}]}
 {\Pp_{\Sigma,\alpha,\beta}[u_{R,T}]}
 =J_{\Sigma,a,\beta}^{q/p}.
\end{equation}
For $m=1$ the same proof uses counting measure on $\Sph^0$.

For a single sequence, take $T_j\to\infty$ with the normal
quotient error below $1/j$. For each fixed $T_j$, choose $R_j$
so large that the tangential quotient error in
\eqref{eq:flat-product-limit} is also below $1/j$.
Then $u_{R_j,T_j}$ belongs to the stated smooth test class, and the ratios
 $\E_{\Sigma,\alpha,\beta}[u_{R_j,T_j}]/
 \Pp_{\Sigma,\alpha,\beta}[u_{R_j,T_j}]$
converge to $J_{\Sigma,a,\beta}^{q/p}$ as $j\to\infty$,
as required by \eqref{eq:flat-sharp-limit}.
The two limits are taken in this order.

\begin{remark}[One-sided comparison does not transfer a lower bound]
\label{rem:flat-comparison}
For $p<q$, Lemma~\ref{lem:flat-product} proves, throughout
\eqref{eq:flat-profile-range},
\begin{equation}\label{eq:flat-best-comparison}
 H_{\Sigma,\alpha,\beta}^{p,q}
 \le\mathfrak g_{n,k,s,p}^{q/p}H_{m,\alpha,\beta}^{p,q}.
\end{equation}
This is an upper bound, not an equality theorem. In particular,
sharpness for the point problem alone does not prove the flat lower
bound on a strictly extended interval. Conversely, any strict loss
of homogeneous sharpness for the normal point problem implies strict
loss for the flat problem at the same weight. For $m\ge2$, the
point-endpoint loss theorem therefore gives flat loss near both
endpoints of \eqref{eq:flat-profile-range}. The upper normal endpoint
is an interior point of the ambient finite-energy interval when $k>0$.
\end{remark}

\subsection{A reference estimate for tangential shifts}

We establish the reference estimate needed in the proof of
Proposition~\ref{prop:flat-extension}. It controls the negative part
of the integrated remainder and yields the additional widths in
Theorem~\ref{thm:flat}\textup{(i), (ii)}.

For $f\in L^2(\mathcal C_\Sigma)$ define
\begin{align}
 \mathcal D_{\Sigma,\beta,t}[f]
 &=\int_{\mathcal C_\Sigma}\int_{\R^k}\int_{\Sph^{m-1}}
 \mathcal K_{\beta,t}(\eta,\omega,\sigma)
 |f(y,\tau,\omega)-f(y+e^\tau\eta,\tau+t,\sigma)|^2
 \dd\sigma\dd\eta\dd y\dd\tau\dd\omega,
 \label{eq:flat-D}\\
 \mathcal Q_\Sigma[f]
 &=\int_1^2\int_{\mathcal C_\Sigma}\int_{|\eta|<1}
 \int_{\Sph^{m-1}}
 |f(y,\tau,\omega)-f(y+e^\tau\eta,\tau+h,\sigma)|^2
 \dd\sigma\dd\eta\dd y\dd\tau\dd\omega\dd h.
 \label{eq:flat-Q}
\end{align}
The shifts $(y,\tau)\mapsto(y+e^\tau\eta,\tau+t)$ have
Jacobian one and preserve the $L^2$ norm. The reference energy averages
their squared differences over $|\eta|<1$ and $1\le t\le2$, and is
therefore finite for $L^2$ functions.

\begin{lemma}[An estimate for affine shifts in $L^2$]\label{lem:flat-affine}
For $F\in L^2(\R^k\times\R)$, put
\[
 U_{\eta,h}F(y,\tau)=F(y+e^\tau\eta,\tau+h),\qquad
 P_F(\eta,h)=\|F-U_{\eta,h}F\|_2^2,
\]
and
\[
 Q_F=\int_1^2\int_{|\zeta|<1}P_F(\zeta,r)\dd\zeta\dd r.
\]
There is $C=C(k)$ such that, for $t\ge0$ and $\eta\in\R^k$,
\begin{equation}\label{eq:flat-affine-chain}
 P_F(\eta,t)\le
 C\bigl(1+t+\log(1+e^{-t}|\eta|)\bigr)^2 Q_F.
\end{equation}
\end{lemma}
\begin{proof}
Each $U_{\eta,h}$ preserves the $L^2$ norm, and is therefore
unitary, because its coordinate map has Jacobian one. Composition has the exact form
\[
 U_{\eta,h}U_{\zeta,r}=U_{\eta+e^h\zeta,h+r}.
\]
Write $g=(\eta,h)$, use this product on parameter pairs, and set
$S=B_1\times[1,2]$ and $S_0=B_{1/4}\times[5/4,7/4]$.

With $\delta=1/8$, left multiplication sends
$(\zeta,r)\in S_0$ to $(\eta+e^h\zeta,h+r)$.
If $|\eta|\le\delta$ and $|h|\le\delta$, then
$|\eta+e^h\zeta|\le1/8+e^{1/8}/4<1$ and
$9/8\le h+r\le15/8$. Thus it maps $S_0$ into $S$.

For $v\in S_0$, use the exact identity
$F-U_gF=(F-U_{gv}F)+U_g(U_vF-F)$.
The norm triangle inequality, $(a+b)^2\le2a^2+2b^2$, and
unitarity therefore give
\[
 P_F(g)\le2P_F(gv)+2P_F(v).
\]
Left multiplication in parameter space has Jacobian $e^{kh}$.
Averaging the last bound over $S_0$ therefore gives
\begin{equation}\label{eq:flat-local-chain}
 P_F(g)\le
 \frac{2(e^{k\delta}+1)}{|S_0|}Q_F
 \qquad(|\eta|\le\delta,\ |h|\le\delta).
\end{equation}
This argument controls a whole neighborhood of the identity by an
average over the fixed region $|\zeta|<1$, $1\le r\le2$.

For arbitrary $t\ge0$, set
\[
 H=t+\log(1+\delta^{-1}e^{-t}|\eta|).
\]
Then $H\ge t$ and $|e^{-H}\eta|\le\delta$, and
\[
 (\eta,t)=(0,H)(e^{-H}\eta,0)(0,t-H).
\]
Split each of the two pure logarithmic shifts into increments of
absolute size at most $\delta$. Together with the middle tangential
shift, the resulting number $N$ of factors is at most
\[
 C_\delta\bigl(1+t+\log(1+e^{-t}|\eta|)\bigr).
\]

If these factors are $g_1,\ldots,g_N$, telescoping gives
\[
 F-U_{g_1}\cdots U_{g_N}F
 =\sum_{j=1}^N U_{g_1}\cdots U_{g_{j-1}}(F-U_{g_j}F).
\]
Unitarity and Cauchy--Schwarz thus bound the squared norm by
$N\sum_{j=1}^N P_F(g_j)$. Applying \eqref{eq:flat-local-chain}
to each factor proves \eqref{eq:flat-affine-chain}.
\end{proof}

\begin{lemma}[Control of flat logarithmic displacements]
\label{lem:flat-reference}
For every $t_0>0$ there is $C_{t_0}=C(n,k,s,p,\beta,t_0)<\infty$ such that
\begin{equation}\label{eq:flat-reference-control}
 \mathcal D_{\Sigma,\beta,t}[f]
 \le C_{t_0}k_{\Sigma,\beta}(t)(1+t^2)\mathcal Q_\Sigma[f]
 \qquad(t\ge t_0,\ f\in L^2(\mathcal C_\Sigma)).
\end{equation}
Furthermore, $\mathcal Q_\Sigma[f]=0$ implies $f=0$.
\end{lemma}
\begin{proof}
Let $S_m=|\Sph^{m-1}|$, with counting measure if $m=1$, and write
$f=F+g$, where
\[
 F(y,\tau)=\frac1{S_m}\int_{\Sph^{m-1}}f(y,\tau,\omega)\dd\omega,
 \qquad\int_{\Sph^{m-1}}g(y,\tau,\omega)\dd\omega=0.
\]
Independent angular integration in \eqref{eq:flat-Q} gives exactly
\begin{equation}\label{eq:flat-Q-split}
 \mathcal Q_\Sigma[f]=S_m^2Q_F+2S_m|B_1|\|g\|_2^2.
\end{equation}
For each fixed $\eta$ and $t$, integrating $\mathcal K_{\beta,t}$ in
either angular variable gives a quantity independent of the remaining
angle. Hence, by Fubini's theorem, the mixed $F,g$ term in
$\mathcal D_{\Sigma,\beta,t}$ vanishes.
The $g$ contribution is at most $4k_{\Sigma,\beta}(t)\|g\|_2^2$, using
$|a-b|^2\le2|a|^2+2|b|^2$ and the measure-preserving shifts.

The mean contribution is
\[
 \int_{\R^k}\int_{\Sph^{m-1}}\int_{\Sph^{m-1}}
 \mathcal K_{\beta,t}(\eta,\omega,\sigma)P_F(\eta,t)
 \dd\sigma\dd\omega\dd\eta.
\]
Apply Lemma~\ref{lem:flat-affine} and change variables
$\eta=e^t\xi$. The kernel times tangential measure becomes
\begin{equation}\label{eq:flat-scaled-tail}
 \mathcal K_{\beta,t}(e^t\xi,\omega,\sigma)e^{kt}
 =\frac{e^{-(sp-\beta)t}}
 {(|\xi|^2+|e^{-t}\omega-\sigma|^2)^{(n+sp)/2}}.
\end{equation}
For $t\ge t_0$, one has
$1-e^{-t_0}\le|e^{-t}\omega-\sigma|\le2$.
In particular,
\[
 \int_{\R^k}
 \frac{1+\log^2(1+|\xi|)}
 {(|\xi|^2+(1-e^{-t_0})^2)^{(n+sp)/2}}\dd\xi<\infty,
\]
since $n+sp>k$. The mean contribution is thus at most
$C_{t_0}e^{-(sp-\beta)t}(1+t^2)Q_F$.
The same scaled kernel, without the logarithmic factor, gives
$k_{\Sigma,\beta}(t)\asymp_{t_0}e^{-(sp-\beta)t}$ for $t\ge t_0$.
Combining these bounds with \eqref{eq:flat-Q-split} proves
\eqref{eq:flat-reference-control}.

If the reference energy vanishes, then $g=0$ and $Q_F=0$.
Equation~\eqref{eq:flat-local-chain} makes $F$ invariant under
all sufficiently short pure $\tau$ translations, hence under all
$\tau$ translations. An $L^2(\R^k\times\R)$ function with that
property is zero. This proves the last assertion.
\end{proof}

\Needspace{15\baselineskip}
\subsection{Strict extension}

\begin{proposition}[Strict extension for flat singular sets]
\label{prop:flat-extension}
Fix $m+2\beta>sp$ and the other parameters of Theorem~\ref{thm:flat}.
The elementary flat interval extends by a positive width in $a$
beyond $a_0=(m+2\beta-sp)/(p+n_{p,q})$. More precisely, there are
$\eta>0$ and $c>0$ such that, for every nonnegative transformed test $v$,
\begin{equation}\label{eq:flat-strict-remainder}
 I_{\Sigma,a,\beta}[v]-J_{\Sigma,a,\beta}N_\Sigma[v]
 \ge c\mathcal Q_\Sigma[v^{q/2}]
 \qquad(a_0\le a\le a_0+\eta),
\end{equation}
and
\begin{equation}\label{eq:flat-strict-sharp}
 H_{\Sigma,sq-m-q\beta/p+qa,\beta}^{p,q}=J_{\Sigma,a,\beta}^{q/p}
 \qquad(0\le a\le a_0+\eta).
\end{equation}
The width is chosen so that \eqref{eq:flat-profile-range} holds.
For $1<p<q<2$, the optimized low-exponent interval in
Theorem~\ref{thm:flat} is retained; no additional width beyond its
optimized endpoint is asserted in this proposition.
\end{proposition}
\begin{proof}
Choose a compact neighborhood $\mathcal U$ of $a_0$ such that
$a>0$ and $m+2\beta-sp-pa>0$ throughout $\mathcal U$.
Use \eqref{eq:flat-pair} with $f=v^{q/2}$ and put
\[
 c_{a,\beta}(t)=c(e^{-at},e^{-(m+2\beta-sp-pa)t}).
\]
Lemma~\ref{lem:short} makes $c_{a,\beta}(t)\ge0$ for
$0<t<t_0<1$ uniformly on a sufficiently small $\mathcal U$.
At $a_0$, the elementary scalar support has $c_{a_0,\beta}(t)>0$
for every $t>0$. Compactness on $1\le t\le2$ and the positive
lower bound of the kernel for $|\eta|<1$ imply, after shrinking
$\mathcal U$,
\[
 \int_1^2 c_{a,\beta}(t)\mathcal D_{\Sigma,\beta,t}[f]\dd t
 \ge c_*\mathcal Q_\Sigma[f],\qquad c_*>0.
\]
Lemma~\ref{lem:flat-reference} bounds the other negative
contributions by $C_{t_0}\mathcal N_{\Sigma,\beta}(a)\mathcal Q_\Sigma[f]$,
where
\[
 \mathcal N_{\Sigma,\beta}(a)=\int_{t_0}^\infty[c_{a,\beta}(t)]_-
 k_{\Sigma,\beta}(t)(1+t^2)\dd t.
\]
Lemma~\ref{lem:continuity} bounds the scalar negative part uniformly
on $\mathcal U$. Since $k_{\Sigma,\beta}(t)=\mathfrak g_{n,k,s,p}k_{\beta,m}(t)$
has an exponentially decaying tail as $t\to+\infty$, dominated convergence
and strict scalar support at $a_0$ give
$\mathcal N_{\Sigma,\beta}(a)\to0$ as $a\to a_0$.
Choose $\eta>0$ so that
$C_{t_0}\mathcal N_{\Sigma,\beta}(a)\le c_*/2$ on $[a_0,a_0+\eta]$.
Then \eqref{eq:flat-pair} gives \eqref{eq:flat-strict-remainder}.
H\"older yields the sharp Hardy lower bound, and
\eqref{eq:flat-sharp-limit} gives optimality.
The displacement pairing is first performed for
$\epsilon<|t|<T$; the reference estimate controls its negative
part when these cutoffs are removed.
\end{proof}

Lemma~\ref{lem:flat-affine} supplies the additional estimate for
tangential shifts needed to extend the point-singularity proof. Its
factorization of a large shift into small shifts gives the bound in
\eqref{eq:flat-reference-control}, which controls the negative
contribution in \eqref{eq:flat-pair}.

\Needspace{15\baselineskip}
\subsection{The case \texorpdfstring{$p>q$}{p>q}}

\begin{theorem}[Flat-subspace inequality for $p>q$]
\label{thm:flat-reverse}
Let $n\ge2$, $1\le k<n$, $m=n-k$, $0<s<1$, and $1<q<p<\infty$.
Assume
\begin{equation}\label{eq:flat-reverse-range}
 -m<\beta<sp,\qquad -m<\alpha<sq+n(q/p-1).
\end{equation}
Then, with $a=(\alpha+m+q\beta/p-sq)/q$ and
$J_{\Sigma,a,\beta}$ as in \eqref{eq:flat-J},
\[
 \E_{\Sigma,\alpha,\beta}[u]
 \ge J_{\Sigma,a,\beta}^{q/p}\Pp_{\Sigma,\alpha,\beta}[u]
 \qquad(u\in C_c^\infty(\R^n\setminus\Sigma)).
\]
The constant is optimal on the stated test class. It is zero if $a=0$
and positive otherwise.
\end{theorem}
\begin{proof}
Since $q/p<1$ and $n>m$, the ambient upper bound in
\eqref{eq:flat-reverse-range} implies the upper bound for finiteness of the homogeneous integral.
In particular, $J:=J_{\Sigma,a,\beta}<\infty$. Replace the test by
its absolute value, and use \eqref{eq:flat-cylinder-energy} with
$\rho=q/p$, $N=\int_{\mathcal C_\Sigma}v^q$.
Assume first $a\ne0$. For $0<\epsilon<T$ restrict the displacement
to $\epsilon<|t|<T$. That is, replace the $t$-integral in
\eqref{eq:flat-cylinder-A} by its restriction to this annulus, and write
\begin{equation}\label{eq:flat-truncated-A}
\begin{aligned}
 \mathcal A_\epsilon^T[v](y,\tau,\omega)
 ={}&\int_{\epsilon<|t|<T}\int_{\R^k}\int_{\Sph^{m-1}}
 \mathcal K_{\beta,t}(\eta,\omega,\sigma)\\
 &\quad\times|v(y,\tau,\omega)-e^{-at}v(y+e^\tau\eta,\tau+t,\sigma)|^p
 \dd\sigma\dd\eta\dd t,
\end{aligned}
\end{equation}
and
$J_\epsilon^T=\int_{\epsilon<|t|<T}k_{\Sigma,\beta}(t)|1-e^{-at}|^p\dd t$.
At each outer point apply Lemma~\ref{lem:reverse-calibrated} with
measure $\mathcal K_{\beta,t}\dd t\dd\eta\dd\sigma$,
$F=|U-\theta V|$, $W=|1-\theta|$, and $\theta=e^{-at}$.
Lemma~\ref{lem:reverse-picone} then gives
\[
 (\mathcal A_\epsilon^T[v])^\rho
 \ge (J_\epsilon^T)^{\rho-1}
 \int_{\epsilon<|t|<T}\int_{\R^k}\int_{\Sph^{m-1}}
 \mathcal K_{\beta,t}|1-\theta|^{p-2}(1-\theta)
 (U^q-\theta V^q)\dd\sigma\dd\eta\dd t.
\]
Here $U=v(y,\tau,\omega)$ and
$V=v(y+e^\tau\eta,\tau+t,\sigma)$. For every fixed $t\ne0$,
the outer shift preserves $\dd y\dd\tau$, and integrating the kernel
in either angular variable gives the same value. Fubini's theorem
then gives
\begin{equation}\label{eq:flat-Uq-Vq}
\begin{aligned}
 &\int_{\mathcal C_\Sigma}\int_{\R^k}\int_{\Sph^{m-1}}
 \mathcal K_{\beta,t}(\eta,\omega,\sigma)U^q
 \dd\sigma\dd\eta\dd y\dd\tau\dd\omega\\
 &\qquad=\int_{\mathcal C_\Sigma}\int_{\R^k}\int_{\Sph^{m-1}}
 \mathcal K_{\beta,t}(\eta,\omega,\sigma)V^q
 \dd\sigma\dd\eta\dd y\dd\tau\dd\omega
 =k_{\Sigma,\beta}(t)N,
\end{aligned}
\end{equation}
the first identity integrating in $\sigma$ and the second integrating
in $\omega$ after the shear $y\mapsto y+e^\tau\eta$, which
leaves $\dd y\dd\eta$ unchanged. All signed expressions are integrable at
these finite displacement cutoffs. Integrating the preceding inequality
therefore gives
\[
 \int_{\mathcal C_\Sigma}(\mathcal A_\epsilon^T[v])^\rho
 \ge (J_\epsilon^T)^{\rho-1}
 \int_{\epsilon<|t|<T}k_{\Sigma,\beta}(t)|1-\theta|^p\dd t\,N
 =(J_\epsilon^T)^\rho N.
\]
Nested cutoffs and monotone convergence prove the lower bound.
For $a=0$ it is nonnegativity. Lemma~\ref{lem:flat-product} applies
under \eqref{eq:flat-reverse-range}; its proof above covers $p>q$.
Combining it with the normal logarithmic cutoffs and
\eqref{eq:reverse-cutoff-error} gives \eqref{eq:flat-sharp-limit},
including $a=0$. A diagonal choice of the two cutoff scales lies
in the original smooth test class and proves optimality.
\end{proof}

\subsection{The roles of ambient dimension and codimension}

The dimensional rules are
\[
 \text{homogeneous power: }\frac{\alpha+m+q\beta/p-sq}{q},\qquad
 \text{tangential factor: }\mathfrak g_{n,k,s,p},
\]
\[
 \text{test-energy range: }-m<\alpha<sq+n(q/p-1),\qquad
 \text{profile range: }-m<\alpha<sq+m(q/p-1).
\]
Thus a formal replacement of the ambient dimension by the codimension
in every statement would give an incorrect finiteness range.

\section{Regional Hardy inequalities on a half-space}\label{sec:halfspace}

Let $n\ge2$ and write $X=(y,r)\in\R^{n-1}\times(0,\infty)=:\R^n_+$.
Both integrations below are over the half-space. Thus the result
concerns a regional difference energy, as in \Kij\ (Theorem~2).
The full-space problem with singular set $\{r=0\}$ in
Section~\ref{sec:flat} includes interactions between the two sides of
that plane and has a different homogeneous constant.

\subsection{The one-sided homogeneous integral and finiteness}

For $0<s<1$, $1<p,q<\infty$, $-1<\beta<sp$, and
$u\in C_c^\infty(\R^n_+)$ put
\begin{align}
 \E_{+,\alpha,\beta}[u]
 &=\int_{\R^n_+}r^\alpha
 \left(\int_{\R^n_+}
 \frac{|u(y,r)-u(y',r')|^p(r')^\beta}
 {(|y-y'|^2+|r-r'|^2)^{(n+sp)/2}}\dd y'\dd r'
 \right)^{q/p}\dd y\dd r,\label{eq:half-energy}\\
 \Pp_{+,\alpha,\beta}[u]
 &=\int_{\R^n_+}|u(y,r)|^q r^{\alpha+q\beta/p-sq}\dd y\dd r,
 \label{eq:half-denominator}
\end{align}
and let $H_{+,\alpha,\beta}^{p,q}$ be the infimum of their quotient
over the nonzero tests in this class. Set
\begin{equation}\label{eq:half-parameters}
 a=\frac{\alpha+1+q\beta/p-sq}{q},\quad
 D_+=1+2\beta-sp,\quad
 g_+=\frac{\pi^{(n-1)/2}\Gamma((1+sp)/2)}{\Gamma((n+sp)/2)}.
\end{equation}
The one-sided homogeneous candidate is
\begin{equation}\label{eq:half-J}
\begin{aligned}
 J_{+,a,\beta}
 &=g_+\int_0^\infty\frac{|1-t^{-a}|^p t^\beta}{|1-t|^{1+sp}}\dd t\\
 &=g_+\int_0^1\frac{|1-t^a|^p
 (t^{\beta-pa}+t^{sp-\beta-1})}{(1-t)^{1+sp}}\dd t.
\end{aligned}
\end{equation}
For the second equality, split the first integral at one, substitute
$t\mapsto1/t$ in its upper part, and use
$|1-t^{-a}|^p=t^{-pa}|1-t^a|^p$ in its lower part.
This calculation fixes the power sign in the inverted formula.

\begin{proposition}[Half-space finiteness]\label{prop:half-finite}
Every nonzero test in $C_c^\infty(\R^n_+)$ has finite energy if and only if
\begin{equation}\label{eq:half-ambient-range}
 -1<\alpha<sq+n(q/p-1).
\end{equation}
The integral in \eqref{eq:half-J} is finite if and only if
\begin{equation}\label{eq:half-profile-range}
 \frac{\beta-sp}{p}<a<\frac{1+\beta}{p},
 \qquad\text{equivalently}\qquad
 -1<\alpha<sq+q/p-1.
\end{equation}
It is zero at $a=0$ and positive elsewhere in this range.
\end{proposition}
\begin{proof}
The test vanishes in a strip near the boundary. Its inner difference
energy extends to a bounded, strictly positive function on each
bounded boundary set, so outer normal integration requires $\alpha>-1$.
Outside a ball containing the support the inner integral is comparable,
uniformly in the outer direction, to
$|X|^{-n-sp}\int|u(y',r')|^p(r')^\beta\dd y'\dd r'$.
Writing $X=R\omega$ in the upper hemisphere gives an angular factor
$\int_{\Sph^{n-1}_+}\omega_n^\alpha\dd\omega$, finite precisely
for $\alpha>-1$, and the radial integral
$\int^\infty R^{\alpha+n-1-q(n+sp)/p}\dd R$.
This proves \eqref{eq:half-ambient-range}; smoothness and $s<1$
control the diagonal. The inner boundary and inner infinity
conditions are respectively $\beta>-1$ and $\beta<sp$.

Near $t=1$ the integrand in \eqref{eq:half-J} is
$O(|1-t|^{p-1-sp})$. At zero, a positive $a$ requires
$\beta-pa>-1$; at infinity, a negative $a$ requires
$\beta-sp-pa<0$. The standing bounds on $\beta$ cover the opposite
signs, proving \eqref{eq:half-profile-range}.
\end{proof}

For $p<q$ the finiteness range of the homogeneous integral is smaller than the ambient range.
For $p>q$ the ambient range is smaller and is nonempty exactly when
$q(n+sp)/p>n-1$. A one-dimensional weight condition alone therefore
does not establish finiteness of this $n$-dimensional energy.

% Layout: keep this heading with its theorem.
\Needspace{12\baselineskip}
\subsection{Sharp half-space inequalities}

\begin{theorem}[Regional half-space Hardy inequality]\label{thm:halfspace}
Let $n\ge2$, $0<s<1$, and $-1<\beta<sp$.
\begin{enumerate}[label=\textup{(\roman*)}]
\item Suppose $1<q<p<\infty$ and \eqref{eq:half-ambient-range} holds.
For every $u\in C_c^\infty(\R^n_+)$,
\begin{equation}\label{eq:half-main}
 \E_{+,\alpha,\beta}[u]
 \ge J_{+,a,\beta}^{q/p}\Pp_{+,\alpha,\beta}[u],
 \qquad a=\frac{\alpha+1+q\beta/p-sq}{q},
\end{equation}
with the optimal constant. This is the complementary-order
half-space comparison with \Kij\ (Theorem~2), on the overlap of
its hypotheses and the smooth test class specified here.
\item Suppose $1<p<q<\infty$ and $D_+=1+2\beta-sp>0$.
Set $\alpha_{H,+}=sq-1-q\beta/p$.
Then \eqref{eq:half-main} holds with the optimal constant on each
of the following intervals:
\[
\begin{array}{ll}
 2\le q\le2p:
 &\alpha_{H,+}\le\alpha\le
 \alpha_{H,+}+\dfrac{qD_+}{p+1}+\varepsilon,\\[5pt]
 q>2p:
 &\alpha_{H,+}\le\alpha\le
 \alpha_{H,+}+D_++\varepsilon,\\[5pt]
 1<p<q<2:
 &\alpha_{H,+}\le\alpha\le
 \alpha_{H,+}+\dfrac{qD_+}{p+\nsc}.
\end{array}
\]
Here $\nsc$ is defined in \eqref{eq:nsc-main}; in the first two
lines some $\varepsilon=\varepsilon(n,s,p,q,\beta)>0$ exists,
chosen within \eqref{eq:half-profile-range}.
\end{enumerate}
In either part the constant vanishes when $a=0$ and is positive at
every other displayed parameter. Optimality holds on $C_c^\infty(\R^n_+)$.
\end{theorem}

\subsection{Cylinder, lower bounds, and strict extension}

We prove the lower bounds in Theorem~\ref{thm:halfspace}, treating
parts~\textup{(i)} and~\textup{(ii)} separately. For part~\textup{(ii)},
we also prove the positive extension widths in its first two
parameter ranges. Optimality is proved in the next subsection.

Write $u(y,e^\tau)=e^{-a\tau}v(y,\tau)$ and use
$y'=y+e^\tau\eta$, $r'=e^{\tau+t}$. The transformed kernel and its
angular integral are
\begin{equation}\label{eq:half-kernel}
\begin{aligned}
 K^+_{\beta,t}(\eta)
 &=\frac{e^{(1+\beta)t}}
 {(|\eta|^2+|1-e^t|^2)^{(n+sp)/2}},\\
 k^+_\beta(t)&=\int_{\R^{n-1}}K^+_{\beta,t}(\eta)\dd\eta
 =g_+\frac{e^{(1+\beta)t}}{|1-e^t|^{1+sp}}.
\end{aligned}
\end{equation}
The powers of $e^\tau$ cancel because $\alpha+1=q(a+s-\beta/p)$.
On $\mathcal C_+=\R^{n-1}\times\R$ with measure $\dd y\dd\tau$,
\begin{equation}\label{eq:half-cylinder}
\begin{gathered}
 \E_{+,\alpha,\beta}[u]=\int_{\mathcal C_+}(A^+_{a,\beta}[v])^{q/p},
 \qquad \Pp_{+,\alpha,\beta}[u]=N_+[v]:=\int_{\mathcal C_+}|v|^q,\\
\begin{aligned}
 A^+_{a,\beta}[v](y,\tau)
 ={}&\int_\R\int_{\R^{n-1}}K^+_{\beta,t}(\eta)\\
 &\quad\times|v(y,\tau)-e^{-at}v(y+e^\tau\eta,\tau+t)|^p\dd\eta\dd t.
\end{aligned}
\end{gathered}
\end{equation}
Direct scaling gives the reversal identity
\begin{equation}\label{eq:half-reversal}
 K^+_{\beta,-t}(-e^{-t}\eta)e^{-(n-1)t}
 =e^{-D_+t}K^+_{\beta,t}(\eta).
\end{equation}
The factor on the left is the tangential Jacobian. Thus the reversal
exponent is $1+2\beta-sp$. The quantity $k^+_\beta$ satisfies
\begin{equation}\label{eq:half-k-estimates}
\begin{aligned}
 k^+_\beta(t)&\asymp |t|^{-1-sp} &&(0<|t|\le1),\\
 k^+_\beta(t)&\asymp e^{-(sp-\beta)t} &&(t\ge1),\\
 k^+_\beta(t)&\asymp e^{(1+\beta)t} &&(t\le-1).
\end{aligned}
\end{equation}

\begin{proof}[Proof of Theorem~\ref{thm:halfspace}\textup{(i)}: lower bound]
It suffices to consider $v\ge0$. Suppose $a\ne0$, and fix
$0<\delta<T$. Let $A^+_{\delta,T}$ denote the inner integral restricted
to $\delta<|t|<T$, and put
$J^+_{\delta,T}=\int_{\delta<|t|<T}k^+_\beta(t)|1-e^{-at}|^p\dd t$.
These reference integrals are positive and finite. Applying
Lemmas~\ref{lem:reverse-calibrated} and~\ref{lem:reverse-picone} gives,
with $\theta=e^{-at}$, $U=v(y,\tau)$, and
$V=v(y+e^\tau\eta,\tau+t)$,
\[
 (A^+_{\delta,T}[v])^{q/p}
 \ge (J^+_{\delta,T})^{q/p-1}
 \int_{\delta<|t|<T}\int_{\R^{n-1}}K^+_{\beta,t}(\eta)
 |1-\theta|^{p-2}(1-\theta)(U^q-\theta V^q)\dd\eta\dd t.
\]
For fixed $\eta,t$, the map $(y,\tau)\mapsto
(y+e^\tau\eta,\tau+t)$ has Jacobian one. Integration in
$y,\tau,\eta$ therefore turns the last parenthesis into
$k^+_\beta(t)(1-\theta)N_+[v]$. The resulting lower bound is
$(J^+_{\delta,T})^{q/p}N_+[v]$. Monotone convergence as
$\delta\downarrow0$, $T\uparrow\infty$ proves \eqref{eq:half-main}.
At $a=0$ the right side is zero. All signed integrals are finite
before taking limits, and no negative powers of the test are used.
\end{proof}

\begin{proof}[Proof of Theorem~\ref{thm:halfspace}\textup{(ii)}: lower bound and extension]
For $v\ge0$ set $I_+[v]=\int_{\mathcal C_+}v^{q-p}A^+_{a,\beta}[v]$.
Outer H\"older gives
$I_+[v]\le\E_{+,\alpha,\beta}[u]^{p/q}N_+[v]^{1-p/q}$.
Using the same displacement cutoffs in all terms, \eqref{eq:half-reversal} gives
\begin{equation}\label{eq:half-pair}
 I_+[v]-J_{+,a,\beta}N_+[v]
 =\int_0^\infty\int_{\mathcal C_+}\int_{\R^{n-1}}
 K^+_{\beta,t}(\eta)G_{\theta,b}(U,V)\dd\eta\dd y\dd\tau\dd t,
\end{equation}
where $\theta=e^{-at}$ and $b=e^{-(D_+-pa)t}$.
As in \eqref{eq:flat-pair}, integration of either $U^q$ or $V^q$
gives $k^+_\beta(t)N_+[v]$, justifying the homogeneous subtraction.
Lemma~\ref{lem:scalar-support} proves the lower bound for
$0<a\le D_+/(p+n_{p,q})$; the exact low-exponent criterion improves
this to $0<a\le D_+/(p+\nsc)$ if $1<p<q<2$.

We first bound all tangential differences by one fixed reference integral, which will control the negative contribution when the first two intervals are extended.
Put
\begin{align*}
 D^+_{\beta,t}[f]
 &=\int_{\mathcal C_+}\int_{\R^{n-1}}K^+_{\beta,t}(\eta)
 |f(y,\tau)-f(y+e^\tau\eta,\tau+t)|^2\dd\eta\dd y\dd\tau,\\
 Q_+[f]
 &=\int_1^2\int_{|\eta|<1}\int_{\mathcal C_+}
 |f(y,\tau)-f(y+e^\tau\eta,\tau+h)|^2\dd y\dd\tau\dd\eta\dd h.
\end{align*}
Lemma~\ref{lem:flat-affine}, with $k=n-1$, bounds the squared norm
of each increment by
$C(1+t+\log(1+e^{-t}|\eta|))^2Q_+[f]$.
For $t\ge t_0>0$, substitute $\eta=e^t\xi$:
\[
 K^+_{\beta,t}(e^t\xi)e^{(n-1)t}
 =e^{-(sp-\beta)t}(|\xi|^2+|1-e^{-t}|^2)^{-(n+sp)/2}.
\]
Here $|1-e^{-t}|\ge1-e^{-t_0}>0$. The integral of this kernel
against $1+\log^2(1+|\xi|)$ is finite since $n+sp>n-1$. Hence
\begin{equation}\label{eq:half-reference}
 D^+_{\beta,t}[f]\le C_{t_0}k^+_\beta(t)(1+t^2)Q_+[f],
 \qquad t\ge t_0.
\end{equation}
Also $Q_+[f]=0$ implies $f=0$ by the local affine-chain argument.

Let $a_0=D_+/(p+n_{p,q})$, $f=v^{q/2}$, and
$c_a(t)=c(e^{-at},e^{-(D_+-pa)t})$. On a sufficiently small
compact neighborhood of $a_0$, Lemma~\ref{lem:short} makes $c_a(t)$
nonnegative for $0<t<t_0<1$. At $a_0$ it is strictly positive for
every $t>0$. On $1\le t\le2$, $|\eta|<1$, the kernel has a positive
lower bound. Continuity thus supplies a uniform contribution
$c_*Q_+[f]$ from those finite scales. By \eqref{eq:half-reference}
the negative part of the scalar remainder contributes at most
\[
 C_{t_0}Q_+[f]\int_{t_0}^\infty[c_a(t)]_-
 k^+_\beta(t)(1+t^2)\dd t.
\]
Lemma~\ref{lem:continuity} and the exponential decay as $t\to+\infty$
in \eqref{eq:half-k-estimates} imply that this integral tends to
zero as $a\to a_0$. Choose $\eta_0>0$ so it is at most
$c_*/(2C_{t_0})$ for $a_0\le a\le a_0+\eta_0$. Equation
\eqref{eq:half-pair} gives
$I_+[v]-J_{+,a,\beta}N_+[v]\ge(c_*/2)Q_+[v^{q/2}]$.
Outer H\"older proves strict extension, with width $q\eta_0$
in $\alpha$. The estimates control the negative part before
the displacement cutoffs are removed. At $a=0$ the lower bound
is nonnegativity. Absolute values of tests complete both proofs.
\end{proof}

\subsection{\texorpdfstring{Proof of optimality in Theorem~\ref{thm:halfspace}}{Proof of optimality for the half-space}}

We prove the optimality assertion in both parts of
Theorem~\ref{thm:halfspace} by first removing a tangential cutoff
and then taking a logarithmic cutoff in the normal variable.

For $w\in C_c^\infty(0,\infty)$,
$\chi\in C_c^\infty(\R^{n-1})$, and $u_R(y,r)=\chi(y/R)w(r)$,
the tangential cutoff proof applies with normal variable in $(0,\infty)$.
On bounded $Y=y/R$ use the same $F_R$ estimate as in
Lemma~\ref{lem:flat-product}: $F_1(r)^p$ is bounded near zero and
decays as $r^{-1-sp}$ at infinity. On unbounded $Y$ the error satisfies
\[
 B_R(Y,r)\le C R^{n-1}(R^2|Y|^2+r^2)^{-(n+sp)/2}.
\]
The integration in \eqref{eq:flat-product-far-error}, with $m=1$,
bounds its outer integral by
$C R^{\alpha+1-(q/p)(1+sp)}\to0$; its $Y$-integral is finite
by \eqref{eq:half-ambient-range}. Thus, under both finiteness ranges,
\begin{equation}\label{eq:half-product}
 \lim_{R\to\infty}
 \frac{\E_{+,\alpha,\beta}[\chi(y/R)w(r)]}
 {\Pp_{+,\alpha,\beta}[\chi(y/R)w(r)]}
 =g_+^{q/p}
 \frac{\E_{(0,\infty),\alpha,\beta}[w]}
 {\Pp_{(0,\infty),\alpha,\beta}[w]}.
\end{equation}
Here the right side uses the one-dimensional regional functionals
obtained from \eqref{eq:half-energy}--\eqref{eq:half-denominator}
by omitting the tangential coordinates and setting the kernel
coefficient to one. The proof applies to either order of $p,q$.

Take $w_T(r)=r^{-a}\zeta(\log r/T)$ with
$0\ne\zeta\in C_c^\infty(\R)$. The corresponding normal integral is
$k^+_\beta/g_+$. By \eqref{eq:half-k-estimates}, its product
with $e^{-pat}$ has exponential tails in the finiteness range of the homogeneous integral and
a finite $p$th displacement moment. The direct estimate
\eqref{eq:reverse-cutoff-error} gives an error
$O(T^{1/q-1})+O(e^{-cT})$, against a denominator norm of order
$T^{1/q}$. The normal quotient tends to
$(J_{+,a,\beta}/g_+)^{q/p}$, including $a=0$.
First let $R\to\infty$ in \eqref{eq:half-product}, then $T\to\infty$.
A diagonal sequence lies in $C_c^\infty(\R^n_+)$ and proves
all optimality assertions in Theorem~\ref{thm:halfspace}.

The difference from the full-space codimension-one candidate is explicit:
\[
 J_{\Sigma,a,\beta}
 =J_{+,a,\beta}
 +g_+\int_0^\infty
 \frac{|1-t^{-a}|^p t^\beta}{(1+t)^{1+sp}}\dd t
 \qquad(m=1).
\]
The additional term is the interaction with the opposite normal side.

For comparison with the zero-extension theorem of~\cite{BBZ}, put
$q=p$, $\alpha=\beta=0$, and $a=(1-sp)/p$. Then
$J_{+,a,0}=D_{n,s,p}$, whereas
\[
 h_{n,s,p}(\R^n_+)=J_{+,a,0}+\frac{2g_+}{sp}
 =g_+\Lambda_{s,p}.
\]
The term $2g_+/(sp)$ is the exterior contribution in
\eqref{eq:intro-regional-full}. This additive constant conversion
uses the equality $q=p$: for a mixed-exponent energy the power
$q/p$ acts on the inner sum, so an additive conversion of best
constants does not follow from that identity.

\subsection{Loss near the homogeneous integral's finiteness endpoints}

\begin{proposition}[Half-space loss of homogeneous sharpness]
\label{prop:half-loss}
Let $n\ge2$, $1<p<q<\infty$, $0<s<1$, and $-1<\beta<sp$.
Put $\alpha_N=sq+q/p-1$. Within the finiteness range of the homogeneous integral,
\[
 \frac{H_{+,\alpha,\beta}^{p,q}}{J_{+,a,\beta}^{q/p}}
 \longrightarrow0
 \quad\text{as }\alpha\downarrow-1
 \quad\text{or as }\alpha\uparrow\alpha_N.
\]
The homogeneous constant is therefore strictly nonoptimal sufficiently
close to either endpoint, while compact interior tests have finite energy.
\end{proposition}
\begin{proof}
Write $\rho=q/p>1$. At $\alpha=-1+\epsilon$,
$a=(\beta-sp)/p+\epsilon/q$, and the large-$t$ part of
\eqref{eq:half-J} gives $J_{+,a,\beta}\sim g_+q/(p\epsilon)$.
At $\alpha=\alpha_N-\epsilon$,
$a=(1+\beta)/p-\epsilon/q$, and the small-$t$ part gives the same
asymptotic. These follow by comparison with
$\int_1^\infty t^{-1-p\epsilon/q}\dd t$ and
$\int_0^1t^{-1+p\epsilon/q}\dd t$, respectively.
After subtracting these leading powers, the corrections are uniformly
integrable near the corresponding endpoint; the other pieces are bounded.

For the lower endpoint take
$u_\ell(X)=\psi((X-e_n)/\ell)$, where
$0\ne\psi\in C_c^\infty(B_{1/4})$ and $0<\ell<1/4$.
The denominator is comparable to $\ell^n$, uniformly for small $\epsilon$.
In a fixed neighborhood of $e_n$ the weights are bounded. Scaling
the differences gives $C\ell^{n-sq}$, including the intervening
annulus between the support scale and that neighborhood.
Away from $e_n$ the inner moment is $O(\ell^n)$, giving the bound
\[
 C\ell^{n\rho}
 \int_{\substack{X\in\R^n_+\\|X-e_n|>1/2}}
 r^{-1+\epsilon}|X-e_n|^{-\rho(n+sp)}\dd X
 \le C\epsilon^{-1}\ell^{n\rho}.
\]
On bounded sets the normal integral is at most $C/\epsilon$.
At infinity the hemisphere angular integral is $O(1/\epsilon)$
and the radial tail converges uniformly since $\rho(n+sp)>n$.
Consequently
\[
 H_{+,-1+\epsilon,\beta}^{p,q}
 \le C\bigl(\ell^{-sq}+\epsilon^{-1}\ell^{n(\rho-1)}\bigr)
 \le C\epsilon^{-\theta},\qquad
 \theta=\frac{sq}{sq+n(\rho-1)}<1,
\]
by choosing $\ell=\epsilon^{1/(sq+n(\rho-1))}$.
Since $J^{q/p}$ has order $\epsilon^{-\rho}$, the first limit follows.

For the upper endpoint, $n\ge2$ and $\rho>1$ imply
$\alpha_N<sq+n(\rho-1)$. Any one fixed nonzero compact interior
test thus has a bounded quotient as $\alpha\uparrow\alpha_N$:
dominated convergence applies with weights in a compact subinterval
of the ambient range. Hence $H_{+,\alpha,\beta}^{p,q}\le C$,
whereas $J^{q/p}\to\infty$, proving the second limit.
\end{proof}

The sharp-interval statements and their proofs also apply on the
half-line after omitting tangential variables and setting $g_+=1$;
the affine estimate then reduces to the ordinary translation chain.
The preceding upper-endpoint proof uses $n\ge2$ explicitly.

\section{Further questions}\label{sec:questions}

The main variational questions are to locate $a_{\sh}$ and determine
the geometry of the sharpness set $\mathcal S$, including the possibility
of a return to sharpness after its first loss. The behavior of
near-extremizing sequences as the parameter approaches $a_{\sh}$
is also unresolved. Matching lower bounds
for $H_{\alpha,\beta}^{p,q}$ near the finite-energy endpoints would
clarify the concentration of minimizing sequences there.

For $1<p<q<2$, the quadratic scalar remainder coefficient vanishes
at a value of $\vartheta$ maximizing \eqref{eq:nsc-main}
(Remark~\ref{rem:nsc-limit}). A further extension of the sharpness
interval calls for an estimate of the integrated remainder that
retains a positive lower bound at this parameter.

For flat singularities and regional half-space energies with $p<q$,
an additional range appears when the tangential space has positive dimension:
compact tests can have finite energy while the homogeneous normal
candidate is infinite. The optimal constants and comparison profiles
in this range remain to be determined. Extensions to curved singular
sets would also have to account for geometric errors absent from the
flat calculation.

Corollary~\ref{cor:TL-completion} extends positive optimal constants
to the difference-norm completion of the smooth test class.
Comparison with the maximal measurable finite-energy class depends
on density, while attainment in a larger energy space depends on
compactness.

\appendix
\section{Classical evaluations of the homogeneous integral}\label{app:evaluations}

This appendix fixes the coefficient-one normalization of the established
constant formulas used in the article. With the kernel of
\eqref{eq:energy}, the operator acting on $\R^d$ satisfies
\begin{equation}\label{eq:L-normalization}
 Lf(x)=\PV\int_{\R^d}\frac{f(x)-f(y)}{|x-y|^{d+2s}}\dd y,
 \qquad
 L=\frac{(-\Delta)^s}{C^{\mathrm{FL}}_{d,s}},
 \qquad
 C^{\mathrm{FL}}_{d,s}
 =\frac{4^s\Gamma(d/2+s)}{\pi^{d/2}|\Gamma(-s)|}
 =\frac{4^s s\,\Gamma(d/2+s)}{\pi^{d/2}\Gamma(1-s)}.
\end{equation}
The master representation is the one-dimensional form
\begin{equation}\label{eq:Jradial}
 J_{d,s,p}^{(\beta)}(a)=\int_0^1 |1-t^a|^p
 \bigl(t^{d+\beta-1-pa}+t^{sp-\beta-1}\bigr)\Phi_{d,s,p}(t)\dd t
\end{equation}
of the radial inversion in~\cite[Section~3]{FS}, extended to the
mixed-exponent formula in \Kij\ (Theorem~1, Eq.~(4)). Evaluating
\eqref{eq:Jradial} in closed form, or instead applying the power formula
for $p=2$ to the two terms of
$|1-|z|^{-a}|^2=2(1-|z|^{-a})-(1-|z|^{-2a})$, gives the entries of
Table~\ref{tab:app-evaluations}. Every entry evaluates the candidate;
none of them asserts sharpness. The first row is derived in
Section~\ref{sec:constant-evaluation} and the $p=2$ rows in
Section~\ref{sec:quadratic-constant}.

\begin{table}[tbp]
\centering
\small
\caption{Established evaluations of the homogeneous integral
$J^{(\beta)}_{d,s,p}(a)$ of \eqref{eq:Ja}, in the coefficient-one
normalization \eqref{eq:L-normalization}.}
\label{tab:app-evaluations}
\begin{tabular}{@{}p{1.45in}p{3.05in}p{1.35in}@{}}
\toprule
Evaluation & Closed form & Range and source\\
\midrule
Angular factor
 & $\displaystyle \Phi_{d,s,p}(r)=\frac{2\pi^{d/2}}{\Gamma(d/2)}\,
 {}_2F_1\Bigl(\tfrac{d+sp}{2},1+\tfrac{sp}{2};\tfrac d2;r^2\Bigr)$
 & $0<r<1$;\newline \cite[Eq.~(16)]{DKWeighted}\\[2mm]
$d=3$, all $p$
 & $\displaystyle \Phi_{3,s,p}(r)=\frac{2\pi}{(1+sp)r}
 \bigl((1-r)^{-1-sp}-(1+r)^{-1-sp}\bigr)$
 & $0<r<1$\\[2mm]
Power, $p=2$
 & $\displaystyle (-\Delta)^s|x|^{-b}=\kappa_{d,s}(b)|x|^{-b-2s}$,\newline
 $\displaystyle \kappa_{d,s}(b)=4^s
 \frac{\Gamma\bigl(\tfrac{d-b}{2}\bigr)\Gamma\bigl(\tfrac{b+2s}{2}\bigr)}
 {\Gamma\bigl(\tfrac b2\bigr)\Gamma\bigl(\tfrac{d-b-2s}{2}\bigr)}$
 & $-2s<b<d$;\newline \cite[Eqs.~(3.3)--(3.5)]{FLS}\\[3mm]
$p=2$, no weight
 & $\displaystyle J_{d,s,2}(a)=
 \frac{2\kappa_{d,s}(a)-\kappa_{d,s}(2a)}{C^{\mathrm{FL}}_{d,s}}$
 & $-s<a<\tfrac d2$\\[3mm]
$p=2$, weight $|z|^\beta$
 & $\displaystyle J^{(\beta)}_{d,s,2}(a)=
 \frac{2\kappa_{d,s}(a-\beta)-\kappa_{d,s}(2a-\beta)
       -\kappa_{d,s}(-\beta)}{C^{\mathrm{FL}}_{d,s}}$
 & $\tfrac{\beta-2s}{2}<a<\tfrac{d+\beta}{2}$,\newline $-d<\beta<2s$\\[3mm]
Classical constant,\newline $\alpha=0$, $a=\tfrac{d-2s}{2}$
 & $\displaystyle \frac{2\pi^{d/2}|\Gamma(-s)|
 \bigl[\Gamma\bigl(\tfrac{d+2s}{4}\bigr)\bigr]^2}
 {\Gamma\bigl(\tfrac{d+2s}{2}\bigr)
 \bigl[\Gamma\bigl(\tfrac{d-2s}{4}\bigr)\bigr]^2}$
 & $d>2s$;\newline \cite[Eq.~(1.6)]{FS}\\[3mm]
$d=3$, $s=\tfrac12$, $p=2$
 & $\displaystyle \pi^2\Bigl[2(1-a)\tan\frac{\pi a}{2}
 -(1-2a)\tan(\pi a)\Bigr]$
 & $-\tfrac12<a<\tfrac32$,\newline limits at the poles\\
\bottomrule
\end{tabular}
\end{table}

Two remarks. Both closed forms for $\Phi_{d,s,p}$ are stated for
$0<r<1$; for $r>1$ they extend by the exact relation
$\Phi_{d,s,p}(1/r)=r^{d+sp}\Phi_{d,s,p}(r)$. The classical constant in
the sixth row is the value of the candidate at the weight $\alpha=0$,
that is, at $a=(d-2s)/2$; it is the constant of the one-weight
fractional Hardy inequality, and the spectral normalization
$4^s[\Gamma(\tfrac{d+2s}{4})/\Gamma(\tfrac{d-2s}{4})]^2$
of~\cite{Herbst,Beckner95,Yafaev} equals $C^{\mathrm{FL}}_{d,s}/2$
times it, as noted in Section~\ref{sec:problem}.

\subsection{Radial inversion}\label{sec:constant-evaluation}

We derive \eqref{eq:Jradial}. For $(\beta-sp)/p<a<(d+\beta)/p$,
polar coordinates and the angular factor $\Phi_{d,s,p}$ of
\eqref{eq:intro-Phi} give
$J^{(\beta)}_{d,s,p}(a)=\int_0^\infty r^{d+\beta-1}|1-r^{-a}|^p
\Phi_{d,s,p}(r)\dd r$. Since $|e-r^{-1}\omega|=r^{-1}|e-r\omega|$,
one has $\Phi_{d,s,p}(1/r)=r^{d+sp}\Phi_{d,s,p}(r)$. Substituting
$r=1/t$ in the part $r>1$ and using
$|1-t^{-a}|^p=t^{-pa}|1-t^a|^p$ in the part $r<1$ gives
\eqref{eq:Jradial}; all changes of variables apply to nonnegative
integrands. At zero, integrability is equivalent to $a<(d+\beta)/p$
when $a>0$ and to $a>(\beta-sp)/p$ when $a<0$. Near one,
$\Phi_{d,s,p}(t)\asymp(1-t)^{-1-sp}$ and $|1-t^a|^p=O((1-t)^p)$, so
$s<1$ suffices.

For $d=3$, direct integration of
$2\pi\int_{-1}^1(1-2t\xi+t^2)^{-(3+sp)/2}\dd\xi$ gives the second row of
Table~\ref{tab:app-evaluations}. At $sp=1$ that row reads
$\Phi_{3,1/2,2}(t)=4\pi(1-t^2)^{-2}$, and \eqref{eq:Jradial} becomes
\begin{equation}\label{eq:Phi3-half}
 J_{3,1/2,2}^{(\beta)}(a)=4\pi\int_0^1
 \frac{(t^{2+\beta-2a}+t^{-\beta})(1-t^a)^2}{(1-t^2)^2}\dd t.
\end{equation}
The hypergeometric entry is the Euler integral combined with the
quadratic transformation in~\cite{DLMF}.

\subsection{Gamma and trigonometric evaluations when \texorpdfstring{$p=2$}{p=2}}\label{sec:quadratic-constant}

For $p=2$ the power entry of Table~\ref{tab:app-evaluations} converts
\eqref{eq:Jradial} into Gamma functions. Write
\begin{equation}\label{eq:cGamma}
 \mathfrak c_n(\gamma)=
 \frac{\pi^{d/2}\Gamma(1-s)}{s\Gamma(d/2+s)}
 \frac{\Gamma((d+n-\gamma)/2)\Gamma((n+\gamma+2s)/2)}
 {\Gamma((n+\gamma)/2)\Gamma((d+n-\gamma-2s)/2)},
\end{equation}
the coefficient of the operator in \eqref{eq:L-normalization} on a
homogeneous power times a spherical harmonic. Then
\begin{equation}\label{eq:powerL}
 \PV\int_{\R^d}\frac{f(x)-f(y)}{|x-y|^{d+2s}}\dd y
 =\mathfrak c_n(\gamma)|x|^{-\gamma-2s}Y_n(\widehat x),
 \qquad f(x)=|x|^{-\gamma}Y_n(\widehat x),
\end{equation}
for $-2s<\gamma<d$ and $x\ne0$, with removable singularities
interpreted by continuity. For the coefficient, let
$\widehat f(\xi)=\int e^{-ix\cdot\xi}f(x)\dd x$ (the hat on a function
denotes the transform, in contrast with the direction $\widehat x$ of a
point). The Riesz-kernel Fourier formula gives the radial coefficient
$A_n(\gamma)=2^{d-\gamma}\pi^{d/2}\Gamma((d+n-\gamma)/2)/
\Gamma((n+\gamma)/2)$ and the angular factor $i^{-n}Y_n(\widehat\xi)$;
multiplication by $|\xi|^{2s}$ and inversion therefore produce
$4^sA_n(\gamma)/A_n(\gamma+2s)$, which is \eqref{eq:cGamma}. This
proves \eqref{eq:powerL} first for $0<\gamma<d-2s$. It extends to
$-2s<\gamma<d$ by analyticity in $\gamma$, as
in~\cite[Section~1.1 and Proposition~3]{DKK}: the power is locally
integrable for $\gamma<d$, the kernel times the power is integrable at
infinity for $\gamma>-2s$, and the symmetric second difference is
$O(|x-y|^2)$ uniformly on compact parameter intervals. The factor $4^s$
of the fractional Laplacian has thus cancelled in \eqref{eq:cGamma},
which is the only normalization statement used below.

\begin{proposition}[Weighted quadratic evaluation]\label{prop:quadratic-evaluation}
Let $d\ge2$, $0<s<1$, $-d<\beta<2s$, and
$(\beta-2s)/2<a<(d+\beta)/2$. Then
\begin{equation}\label{eq:Jgamma}
 J_{d,s,2}^{(\beta)}(a)
 =2\mathfrak c_0(a-\beta)-\mathfrak c_0(2a-\beta)
  -\mathfrak c_0(-\beta).
\end{equation}
\end{proposition}
\begin{proof}
For a power times a spherical harmonic, the exact product identity
\[
 L_\beta f=L(|x|^\beta f)-fL(|x|^\beta)
\]
and \eqref{eq:powerL} give
\[
 L_\beta\bigl(|x|^{-r}Y_n(\widehat x)\bigr)
 =\bigl(\mathfrak c_n(r-\beta)-\mathfrak c_0(-\beta)\bigr)
 |x|^{\beta-r-2s}Y_n(\widehat x).
\]
This uses $-\beta,r-\beta\in(-2s,d)$ for $r=a,2a$,
which follows from the stated strict inequalities. Integrate
$(1-|z|^{-a})^2=2(1-|z|^{-a})-(1-|z|^{-2a})$
against $|z|^\beta|e-z|^{-d-2s}$ with one common cutoff
$|z-e|>\epsilon$. Apply the formula to the two right-hand
integrals and pass to the limit. The left-hand integral converges
absolutely, proving \eqref{eq:Jgamma}.
\end{proof}

At $\beta=0$ the proposition reduces to
$J_{d,s,2}(a)=2\mathfrak c_0(a)-\mathfrak c_0(2a)$; at
$a=(d-2s)/2$ it gives the classical constant of
Table~\ref{tab:app-evaluations}. For a two-weight example, retain
$d=3$, $s=1/2$, $p=2$, $q=4$, and $a=1/4$. If $-5/8\le\beta<1$,
then $a\le D/3$ and $\alpha=-2\beta$, and
Theorem~\ref{thm:main} gives
\begin{equation}\label{eq:two-weight-example}
 H_{-2\beta,\beta}^{2,4}
 =\bigl[2\mathfrak c_0(1/4-\beta)
       -\mathfrak c_0(1/2-\beta)-\mathfrak c_0(-\beta)\bigr]^2,
\end{equation}
whose value at $\beta=0$ is $H_{0,0}^{2,4}=\pi^4(3\sqrt2-4)^2/4$.
All these identities use the coefficient-one kernel.

\section{An all-parameter scalar comparison for \texorpdfstring{$p>q$}{p>q}}\label{app:reverse-scalar}

The scalar support inequality from the $p<q$ analysis has an unrestricted counterpart when $p>q$. Because $q-p<0$, this statement is first made only for strictly positive $U,V$.

\begin{proposition}[All-parameter scalar support when $p>q$]\label{prop:reverse-scalar}
Let $1<q<p<\infty$, $0<\vartheta<1$, and $b\ge0$. Define
\begin{align}
 A_{\vartheta,b}
 &=\frac{(1-\vartheta)^{p-1}}q
       \bigl(q-(q-p)\vartheta+pb\bigr),\label{eq:reverse-support-A}\\
 B_{\vartheta,b}
 &=\frac{(1-\vartheta)^{p-1}}q
       \bigl(b(q-p-q\vartheta)-p\vartheta\bigr).\label{eq:reverse-support-B}
\end{align}
For all $U,V>0$,
\begin{equation}\label{eq:reverse-scalar-support}
 |U-\vartheta V|^p(U^{q-p}+bV^{q-p})
 \ge A_{\vartheta,b}U^q+B_{\vartheta,b}V^q,
\end{equation}
with equality if and only if $U=V$. Moreover, there is a constant
$c_{p,q}(\vartheta)>0$, independent of $b$, such that
\begin{equation}\label{eq:reverse-uniform-b}
\begin{aligned}
 &|U-\vartheta V|^p(U^{q-p}+bV^{q-p})
       -A_{\vartheta,b}U^q-B_{\vartheta,b}V^q\\
 &\hspace{25mm}\ge c_{p,q}(\vartheta)
                  (U^{q/2}-V^{q/2})^2.
\end{aligned}
\end{equation}
\end{proposition}

\begin{proof}
For $c>0$ put $h_c(t)=|t^{1/q}-c|^p$, $t>0$. Away from $t=c^q$,
\begin{equation}\label{eq:reverse-hsecond}
 h_c''(t)=\frac{p}{q^2}t^{1/q-2}|t^{1/q}-c|^{p-2}
            \bigl((p-q)t^{1/q}+(q-1)c\bigr)>0.
\end{equation}
Its first derivative is continuous at $c^q$ since $p>1$. Thus $h_c$ is strictly convex on $(0,\infty)$, including across that point.

Set $X=U^q$ and $Y=V^q$. The left side of \eqref{eq:reverse-scalar-support} is
\begin{align}
 \Phi_b(X,Y)
 &:=|X^{1/q}-\vartheta Y^{1/q}|^p
       (X^{1-p/q}+bY^{1-p/q})\nonumber\\
 &=\vartheta^p X h_{1/\vartheta}(Y/X)
       +bYh_\vartheta(X/Y).\label{eq:reverse-perspective}
\end{align}
Each term is a convex function of $(X,Y)$. For example,
$Xh(Y/X)$ is convex because the ratio for a convex combination is a convex combination of $Y_1/X_1$ and $Y_2/X_2$ with weights proportional to $X_1,X_2$. Hence $\Phi_b$ is convex and positively one-homogeneous. Its supporting plane at $(1,1)$ has no constant term by Euler's identity. Direct differentiation gives the coefficients in \eqref{eq:reverse-support-A}--\eqref{eq:reverse-support-B}. Strict convexity of the first ratio function makes contact with this plane possible only when $Y/X=1$, proving the equality statement. Here strictness means changing the ratio $Y/X$: scaling both
variables by the same positive factor is a direction in which the
function is linear.

Let $G_b$ denote the difference in \eqref{eq:reverse-uniform-b}. Both convex terms have nonnegative supporting-plane remainders, so
\[
 G_b(U,V)=G_0(U,V)+b\,G^{(1)}(U,V),\qquad G^{(1)}\ge0.
\]
It therefore suffices to bound $G_0$. By homogeneity set $V=1$ and consider
\[
 Q(z)=\frac{G_0(z,1)}{(z^{q/2}-1)^2},\qquad z>0,
\]
with its continuous extension at $z=1$. The limiting values are
\begin{align}
 \lim_{z\to1}Q(z)
 &=\frac{2p}{q^2}\vartheta(1-\vartheta)^{p-2}
             \bigl((q-1)+(p-q)\vartheta\bigr)>0,\label{eq:reverse-diagonal-gap}\\
 \lim_{z\to\infty}Q(z)&=1-A_{\vartheta,0}>0,\qquad
 \lim_{z\downarrow0}Q(z)=+\infty.\label{eq:reverse-end-gap}
\end{align}
Indeed,
\[
 A_{\vartheta,0}=(1-\vartheta)^{p-1}
              \left(1+\frac{p-q}{q}\vartheta\right)<1,
\]
because this expression equals one at $\vartheta=0$ and has strictly negative derivative for $0<\vartheta<1$. At zero, the term $\vartheta^p z^{q-p}$ diverges. Strict positivity on the remaining compact interval yields $\inf Q>0$. Taking this infimum as $c_{p,q}(\vartheta)$ proves the bound for every $b\ge0$.
\end{proof}

\begin{remark}[The scalar inequality is not, by itself, an integral proof]\label{rem:reverse-not-direct}
The factors $U^{q-p}$ and $V^{q-p}$ in Proposition~\ref{prop:reverse-scalar} have negative exponents. They cannot be inserted without regularization at zeros of a compactly supported test function. In addition, the reduction based on $\int v^{q-p}\Aa_{a,\beta}[v]$ uses an outer H\"older inequality with $q/p>1$, which is unavailable here. The treatment of $p>q$ avoids this obstruction by choosing the H\"older weight from the homogeneous power, so that only nonnegative powers of the test function occur; see~\cite[Theorem~1]{KijTL}.
\end{remark}

\section{Auxiliary estimates for \texorpdfstring{$p>q$}{p>q}}\label{sec:reverse-order}

In this appendix $1<q<p<\infty$. In this exponent order the sharp
constant for point singularities is known: \Kij\ (Theorem~1) proves
\eqref{eq:intro-Kij} with the constant
$J^{(\beta)}_{d,s,p}(a)^{q/p}$ on his test classes, and our test class,
whose functions vanish near the origin, is contained in either of his
classes. We therefore do not repeat that proof. What the flat and
half-space arguments of Sections~\ref{sec:flat}
and~\ref{sec:halfspace} use are the three estimates below. They are
stated for a general kernel, so that they apply unchanged to
$k_\beta$, $k_{\Sigma,\beta}$ and $k^+_\beta$.

For $q>1$ write $\varphi_q(t)=|t|^{q-2}t$, with $\varphi_q(0)=0$.

\begin{lemma}[A discrete Picone remainder without negative test powers]\label{lem:reverse-picone}
For $\theta>0$ and $U,V\ge0$, let
\begin{equation}\label{eq:reverse-Picone-def}
 D_{q,\theta}(U,V)
 =|U-\theta V|^q-\varphi_q(1-\theta)(U^q-\theta V^q).
\end{equation}
Then $D_{q,\theta}\ge0$, with equality exactly when $U=V$. If $\theta\ne1$, there is $d_q(\theta)>0$ such that
\begin{equation}\label{eq:reverse-Picone-coercive}
 D_{q,\theta}(U,V)\ge d_q(\theta)(U^{q/2}-V^{q/2})^2.
\end{equation}
The constant can be bounded below uniformly when $\theta$ lies in a compact subset of $(0,\infty)\setminus\{1\}$.
\end{lemma}
\begin{proof}
For $0<\theta<1$, the triangle inequality and convexity give
\[
 U^q\le\bigl(|U-\theta V|+\theta V\bigr)^q
 \le (1-\theta)^{1-q}|U-\theta V|^q+\theta V^q.
\]
This proves \eqref{eq:reverse-Picone-def}; strict convexity identifies equality with $U=V$. For $\theta>1$, use the identity
\[
 D_{q,\theta}(U,V)=\theta^qD_{q,1/\theta}(V,U).
\]
At $\theta=1$, the remainder is simply $|U-V|^q$. This is a scalar specialization of the discrete Picone inequality in~\cite[Proposition~4.2]{BF}; the above argument includes zero test values directly.

For the quantitative assertion, divide by $(U^{q/2}-V^{q/2})^2$ and use homogeneity. With $V=1$, the limits at the diagonal, infinity, and zero are respectively
\begin{equation}\label{eq:reverse-Picone-limits}
\begin{aligned}
 &\frac{2(q-1)}q\,\theta|1-\theta|^{q-2}>0,\\
 &1-\varphi_q(1-\theta)>0,\\
 &\theta^q+\theta\varphi_q(1-\theta)>0.
\end{aligned}
\end{equation}
For example, when $\theta>1$ the last expression is
$\theta[\theta^{q-1}-(\theta-1)^{q-1}]>0$. Strictness elsewhere and compactness prove \eqref{eq:reverse-Picone-coercive}, including its parameter-uniform version.
\end{proof}

\begin{lemma}[Weighted H\"older inequality]\label{lem:reverse-calibrated}
Let $1<q<p$, let $\mu$ be a positive measure, and let $F,W\ge0$ with
$0<\int W^p\dd\mu<\infty$. The function $W$ is a chosen
reference weight; equality below holds for $F=W$. In the application,
$W=|1-e^{-at}|$ is the normalized difference of the homogeneous reference
profile at logarithmic displacement $t$. Then
\begin{equation}\label{eq:reverse-calibrated}
 \left(\int F^p\dd\mu\right)^{q/p}
 \ge
 \left(\int W^p\dd\mu\right)^{q/p-1}
 \int W^{p-q}F^q\dd\mu.
\end{equation}
\end{lemma}
\begin{proof}

If $\int F^p\dd\mu=\infty$, the assertion is immediate.
Otherwise H\"older with conjugate exponents $p/(p-q)$ and
$p/q$ gives
\[
 \int W^{p-q}F^q\dd\mu
 \le\left(\int W^p\dd\mu\right)^{(p-q)/p}
      \left(\int F^p\dd\mu\right)^{q/p}.
\]
The reference integral is positive and finite, so dividing by
its power $(p-q)/p=1-q/p$ proves the claim. In particular,
the exponent of the reference integral on the lower-bound side
is $q/p-1<0$; no concavity inequality is reversed.
\end{proof}

\begin{lemma}[Logarithmic cutoff error in the mixed norm]\label{lem:reverse-cutoff}
Let $T\ge1$ and $0\ne\zeta\in C_c^\infty((-1,1))$, and put
$v_T(\tau)=\zeta(\tau/T)$. Let $k\ge0$ be measurable with
$k(t)\asymp|t|^{-1-sp}$ as $t\to0$ and, for some
$\lambda_+,\lambda_->0$,
\[
 k(t)\le Ce^{-\lambda_+t}\ \ (t\to+\infty),\qquad
 k(t)\le Ce^{\lambda_-t}\ \ (t\to-\infty).
\]
If $-\lambda_+/p<a<\lambda_-/p$, then $w_a(t)=k(t)e^{-pat}$ satisfies
$\int_\R w_a(t)(1+|t|)^p\dd t<\infty$, and
\begin{equation}\label{eq:reverse-cutoff-error}
\begin{aligned}
 e_T(\tau)^p&=\int_\R w_a(t)
 \bigl|\zeta(\tau/T)-\zeta((\tau+t)/T)\bigr|^p\dd t
 \ \text{satisfies}\\
 \|e_T\|_{L^q(\R)}&\le C T^{1/q-1}+Ce^{-c'T},
 \qquad
 \|v_T\|_{L^q(\R)}=T^{1/q}\|\zeta\|_q .
\end{aligned}
\end{equation}
\end{lemma}
\begin{proof}
On $|\tau|\le2T$ the Lipschitz bound on $\zeta$ gives
$e_T(\tau)\le C/T$. If $|\tau|>2T$ the integrand vanishes unless
$|\tau+t|<T$, hence $|t|\ge|\tau|-T$, and the exponential bounds on $k$
give $e_T(\tau)\le Ce^{-c(|\tau|-T)/p}$. Integrating the $q$th power
over $|\tau|\le2T$ contributes $CT^{1-q}$, and over $|\tau|>2T$ at
most $Ce^{-c'qT}$. The last identity is immediate.
\end{proof}

Lemma~\ref{lem:reverse-cutoff} is the step that replaces the
mixed-norm interchange used in~\cite{KijTL}: it needs only the
one-dimensional asymptotics of the kernel and applies to any test
profile that is radial in the logarithmic variable.

\section*{Statements and Declarations}

\paragraph{Use of AI tools.}
An AI assistant (ChatGPT, OpenAI) was used during preparation to assist
with mathematical calculations and proof drafts, literature searches,
and the organization and wording of the manuscript. Responsibility for
all mathematical claims, references, and the final text rests with the
authors.

% AUTHOR ACTION BEFORE SUBMISSION:
% Add the verified funding and competing-interests declarations here.
% Review the AI-use statement before submission.

\end{document}